\documentclass[twoside]{article}
\usepackage{mor}
\usepackage{amsmath}
\usepackage{subfigure}

\received{}                              %
\revised{}                               %
\pubyear{0x}                             %
\pubmonth{Xxxxxxx}                       %
\volume{xx}                              %
\issue{x}                                %
\pages{xxx--xxx}                         %
\firstpage{0xxx}                         %
\DOI{10.1287/moor.xxxx.xxxx}             %
\startpagenumber{1}                      %

\newcommand{\R}{{\mathbb R}}
\newcommand{\N}{{\mathbb N}}
\newcommand{\Q}{{\mathbb Q}}
\newcommand{\Z}{{\mathbb Z}}
\newcommand{\conv}{{\rm conv}}
\newcommand{\cone}{{\rm cone}}

\newcommand{\ri}{\mathrm{ri}}
\newcommand{\intt}{\mathrm{int}}
\newcommand{\spann}{\mathrm{span}}

\title{An analysis of mixed integer linear sets based on lattice point
free convex sets}
\ShortTitle{}
\ShortAuthors{}
\NumberOfAuthors{3}
\FirstAuthor{Kent Andersen}
\FirstAuthorAddress{Department of Mathematics, Otto-von-Guericke
Universit\"at Magdeburg, Germany}
\FirstAuthorEmail{andersen@mail.math.uni-magdeburg.de}
\FirstAuthorURL{}
\SecondAuthor{Quentin Louveaux}
\SecondAuthorAddress{Department of Electrical Engineering and Computer Science, University of Li\`ege, Belgium}
\SecondAuthorEmail{Q.louveaux@ulg.ac.be}
\SecondAuthorURL{}
\ThirdAuthor{Robert Weismantel}
\ThirdAuthorAddress{Department of Mathematics, Otto-von-Guericke
Universit\"at Magdeburg, Germany}
\ThirdAuthorEmail{weismant@mail.math.uni-magdeburg.de}
\ThirdAuthorURL{}
\FourthAuthor{}
\FourthAuthorAddress{}
\FourthAuthorEmail{}
\FourthAuthorURL{}
\FifthAuthor{}
\FifthAuthorAddress{}
\FifthAuthorEmail{}
\FifthAuthorURL{}


\MSCcodes{Primary:   90C11   ; Secondary:  90C10    } 

\ORMScodes{Primary:  Programming, integer, cutting-plane/facet, theory  ; Secondary: Mathematics  , convexity } 
\begin{document}
\maketitle

\begin{abstract}
Split cuts are cutting planes for mixed integer programs 
whose validity is derived from maximal lattice point free 
polyhedra of the form $S:=\{ x : \pi_0 \leq \pi^T x \leq \pi_0+1 \}$ 
called split sets. The set obtained by adding all split 
cuts is called the split closure, and the split closure 
is known to be a polyhedron. A split set $S$ has 
max-facet-width equal to one in the sense that 
$\max\{ \pi^T x : x \in S \}-\min\{ \pi^T x : x \in S \} \leq 1$. 

In this paper we consider using general lattice point 
free rational polyhedra to derive 
valid cuts for mixed integer linear sets. We say 
that lattice point free polyhedra with max-facet-width 
equal to $w$ have width size $w$. 
A split cut of width size $w$ is then a valid inequality 
whose validity follows from a lattice point 
free rational polyhedron of width size $w$. 
The $w^{\textrm{th}}$ split closure is the set 
obtained by adding all valid inequalities of width 
size at most $w$.

In general, a relaxation of a mixed integer set can be 
obtained by adding \emph{any} family of valid inequalities 
to the linear relaxation. Our main result is a sufficient 
condition for the addition of a family of rational inequalities 
to result in a polyhedral relaxation. We then show that a 
corollary is that the $w^{\textrm{th}}$ split closure 
is a polyhedron.

Given this result, a natural question is which width size $w^*$ is 
required to design a finite cutting plane proof for the validity of 
an inequality. Specifically, for this value $w^*$, a finite 
cutting plane proof exists that uses lattice point free 
rational polyhedra of width size at most 
$w^*$, but no finite cutting plane proof that only uses 
lattice point free rational polyhedra of width size smaller than
$w^*$. We characterize $w^*$ based on the faces of 
the linear relaxation.
\end{abstract}
\normalsize

\section{Introduction.}  \label{sect1}

We consider a polyhedron in $\R^n$ of the form
\begin{equation} \label{eq:1}
P:=\conv(\{ v^i {\}}_{i \in V}) + \cone( \{ r^j {\}}_{j \in E} ),
\end{equation}
where $V$ and $E$ are finite index sets, $\{ v^i {\}}_{i \in V}$ 
denotes the vertices of $P$ and 
$\{ r^j {\}}_{j \in E}$ denotes the 
extreme rays of $P$. We assume $P$ is rational, {\it i.e.}, 
we assume $\{ r^j {\}}_{j \in E} \subset \Z^n$ and 
$\{ v^i {\}}_{i \in V} \subset \Q^n$. 
%

We are interested in points in $P$ that have 
integer values on certain coordinates. 
For simplicity assume 
the first $p>0$ coordinates must have integer values, 
and let $q:=n-p$. The set $N_I:=\{1,2,\ldots,p\}$ is used to 
index the integer constrained variables and the set 
$P_I:=\{ x \in P : x_j \in \Z\textrm{ for all }j \in N_I \}$ denotes 
the mixed integer points in $P$. 

The following concepts from convex analysis are needed 
(see \cite{Rockafellar} for a presentation 
of the theory of convex analysis). 
For a convex set $C \subseteq \R^n$, the interior of $C$ 
is denoted $\intt(C)$, 
and the relative interior of $C$ is denoted $\ri(C)$ 
(where $\ri(C)=\intt(C)$ when $C$ is full dimensional). 

We consider the generalization of \emph{split sets} 
(see \cite{CooKanSch90}) to lattice point free rational polyhedra 
(see \cite{lovasz}). A split set is of the form 
$S^{(\pi,\pi_0)} := \{ x \in \R^p : \pi_0 \leq \pi^T x \leq \pi_0+1 \}$, 
where $(\pi,\pi_0) \in \Z^{p+1}$ and $\pi \neq 0$. 
Clearly a split set does not have integer points in 
its interior. In general, a lattice 
point free convex set is a convex set that does 
not contain integer points in its relative interior. 
Lattice point free convex sets that are maximal wrt. inclusion 
are known 
to be polyhedra. We call lattice point free 
rational polyhedra that are maximal wrt. inclusion for 
\emph{split polyhedra}. A split polyhedron 
is full dimensional and can be written as the sum of a 
polytope $\mathcal{P}$ and 
a linear space $\mathcal{L}$. 

A lattice point free convex set is an object that assumes 
integrality of \emph{all} coordinates. For \emph{mixed} 
integrality in $\R^{p+q}$, 
we use a lattice point free convex set $C^x \subset \R^p$ to form a 
\emph{mixed} integer lattice point free convex set $C \subset \R^n$ of the 
form $C:=\{ (x,y) \in \R^p \times \R^q : x \in C^x \}$. A 
\emph{mixed integer split polyhedron} is then a polyhedron of the form 
$L:=\{ (x,y) \in \R^p \times \R^q : x \in L^x \}$, 
where $L^x$ is a split polyhedron in $\R^p$.

An important measure in this paper of the size of a mixed integer split 
polyhedron $L$ is the \emph{facet width} of $L$. The facet width measures 
how wide a mixed integer split polyhedron is parallel to a given facet. 
Specifically, given any facet $\pi^T x \geq \pi_0$ of a 
mixed integer split polyhedron $L$, the width of $L$ along 
$\pi$ is defined to be the number 
$w(L,\pi):=\max_{x \in L} \pi^T x - \min_{x \in L} \pi^T x$. 
The \emph{max-facet-width} of a mixed integer split polyhedron 
$L$ measures how wide $L$ is along any facet of $L$, 
{\it i.e.}, the max-facet-width $w_f(L)$ of $L$ is defined to be the 
largest of the numbers $w(L,\pi)$ over \emph{all} facet defining 
inequalities $\pi^T x \geq \pi_0$ for $L$. 

Any mixed integer lattice point free convex set $C \subseteq \R^n$ 
gives a relaxation of $\conv(P_I)$ 
\begin{alignat}{2}
R(C,P) & :=\conv(\{x \in P : x \notin \ri(C) \})\notag
\end{alignat}
that satisfies $\conv(P_I) \subseteq R(C,P) \subseteq P$. 
The set $R(C,P)$ might exclude fractional points in 
$\ri(C) \cap P$ and give a 
tighter approximation of $\conv(P_I)$ than $P$. 

Mixed integer split polyhedra $L$ give as 
tight relaxations of $P_I$ of the form above as possible. 
Specifically, if $C,C' \subseteq \R^n$ are 
mixed integer lattice point free convex sets that satisfy $C \subseteq C'$, 
then $R(C',P) \subseteq R(C,P)$. 
For a general mixed integer lattice point 
free convex set $C$, 
the set $R(C,P)$ may not be a polyhedron. However, 
it is sufficient to consider mixed integer split polyhedra, and we 
show $R(L,P)$ is a polyhedron when $L$ is a mixed 
integer split polyhedron (Lemma \ref{poly_lem}).

Observe that the set of mixed integer split polyhedra with 
max-facet-width equal to one are exactly the split sets 
$S^{(\pi,\pi_0)} = \{ x \in \R^n : \pi_0 \leq \pi^T x \leq \pi_0+1 \}$, 
where $(\pi,\pi_0) \in \Z^{n+1}$, $\pi_j = 0$ for $j>p$ 
and $\pi \neq 0$. In \cite{CooKanSch90}, Cook et. al. considered 
the set of split sets
$${\cal{L}}^1:= \{ L \subseteq \R^n : L\textrm{ is a mixed integer split polyhedron satisfying }w_f(L) \leq 1 \}$$
and showed that the \emph{split closure}
$$\textrm{SC}^1 := \cap_{L \in {\cal{L}}^1} R(L,P)$$
is a polyhedron. A natural generalization of the split closure is to allow 
for mixed integer split polyhedra that have max-facet-width larger than one. 
For any $w>0$, define the set of mixed integer split polyhedra 
$${\cal{L}}^w:= \{ L \subseteq \R^n : L\textrm{ is a mixed integer split polyhedron satisfying }w_f(L) \leq w \}$$
with max-facet-width at most $w$. We define 
the \emph{$w^{\textrm{th}}$ split closure} to be the set
$$\textrm{SC}^w := \cap_{L \in {\cal{L}}^w} R(L,P).$$

We prove that for any family 
${\bar{\cal{L}}} \subseteq {\cal{L}}^w$ of mixed 
integer split polyhedra with bounded max-facet-width 
$w>0$, the set $\cap_{L \in \bar{{\cal{L}}}} R(L,P)$ 
is a polyhedron (Theorem \ref{poly_spl_cl}). 
The proof is based on an analysis of cutting planes from an inner 
representation of the linear relaxation $P$. In fact, our proof does 
not use an outer description of $P$ at all. Many of our 
arguments are obtained by 
generalizing results of 
Andersen et. al. \cite{andcorli} from the first 
split closure to the $w^{\textrm{th}}$ split closure.

Given a family 
$\{ (\delta^l)^T x \geq \delta^l_0 {\}}_{i \in I}$ of rational cutting 
planes, we provide a sufficient condition for the set 
$\{ x \in P : (\delta^l)^T x \geq \delta^l_0 \textrm{ for all }l \in I
\}$ to be a polyhedron (Theorem \ref{pol_thm}). This condition 
(Assumption \ref{bas_ass}) concerns the number of intersection 
points between hyperplanes 
defined from the cuts 
$\{ (\delta^l)^T x \geq \delta^l_0 {\}}_{i \in I}$ and 
line segments either of the form 
$\{ v^i+ \alpha r^j : \alpha\geq 0 \}$, or of the form 
$\{ \beta v^i + (1-\beta) v^k : \beta \in [0,1] \}$, where 
$i,k \in V$ denote two vertices of $P$ 
and $j \in E$ denotes an extreme ray of $P$. 
We then show that this condition is 
satisfied by the collection of facets of the sets $R(L,P)$ for 
$L \in \bar{{\cal{L}}}$ for any 
family $\bar{{\cal{L}}} \subseteq {\cal{L}}^w$ 
of split polyhedra with bounded 
max-facet-width $w>0$. It follows 
that the $w^{\textrm{th}}$ split 
closure is a polyhedron. 

Finite cutting plane proofs for the validity of an inequality for
$P_I$ can be designed by using mixed integer split polyhedra. 
A measure of the complexity of a finite cutting plane 
proof is the max-facet-width of the mixed integer split polyhedron 
with the largest max-facet-width in the proof. A measure of the 
complexity of a valid inequality $\delta^T x \geq \delta_0$ for $P_I$ 
is the smallest integer $w(\delta,\delta_0)$ for which there exists 
a finite cutting plane proof of validity of 
$\delta^T x \geq \delta_0$ for $P_I$ only using mixed integer split 
polyhedra with max-facet-width at most $w(\delta,\delta_0)$. We give a 
formula for $w(\delta,\delta_0)$ (Theorem \ref{split_dim_ineq}) that 
explains geometrically why mixed integer split polyhedra 
of large width size can be necessary.

The remainder of the paper is organized as follows. In Sect. 2 
we present the main results on lattice point free convex sets that 
are needed in the remainder of the paper. We also present 
the construction of polyhedral relaxations of $P_I$ from 
mixed integer split polyhedra. Most 
results in Sect. 2 can also be found in a paper of Lov\'asz \cite{lovasz}. 
In Sect. 3 we discuss cutting planes from the viewpoint of an 
inner representation of $P$. The main result in Sect. 3 is a 
sufficient condition for a set obtained by adding an infinite 
family of cutting planes to be a polyhedron. The structure 
of the relaxation $R(L,P)$ of $P_I$ obtained from 
a given mixed integer split polyhedron $L$ is characterized in 
Sect. 4. The main outcome is that the 
$w^{\textrm{th}}$ split closure is a polyhedron. Finally, 
in Sect. 5, we discuss the complexity of finite cutting 
plane proofs for the validity of an inequality for $P_I$.

\section{Lattice point free convex sets and polyhedral relaxations}

We now discuss the main object of this paper, namely
lattice point free convex sets, which are defined as follows
\begin{definition} (Lattice point free convex sets)\newline
Let $L\subseteq \R^p$ be a convex set. If 
$\ri(L)\cap \Z^p = \emptyset$, then $L$ is called 
\emph{lattice point free}.
\end{definition}

The discussion of lattice point free convex sets 
in this section is based on a paper of Lov\'asz \cite{lovasz}. 
We are mainly interested in lattice point free convex sets 
that are maximal wrt. inclusion. Our point of departure is the 
following characterization of maximal lattice point free convex 
sets.
\begin{lemma}\label{splitbody_box_rational_polytope}
Every maximal lattice point free convex set $L \subseteq \R^p$ is a polyhedron.
\end{lemma}

As mentioned in the introduction, 
we call maximal lattice point free rational 
polyhedra for \emph{split polyhedra}. Maximal 
lattice point free polyhedra are not necessarily rational. 
The polyhedron $C=\{ (x_1,x_2) : x_2 = \sqrt{2} x_1, x_1 \geq 0 \}$ 
is an example of a maximal lattice point free set which is 
\emph{not} a rational polyhedron. However, we will only 
use maximal lattice point free convex sets to describe (mixed) integer 
points in rational polyhedra, and for this purpose split 
polyhedra suffice.


We next argue that the recession cone $0^+(L)$ of a 
split polyhedron $L$ must be a linear space. This fact 
follows from the following operation to enlarge 
any lattice point free convex set $C \subseteq \R^p$. 
Let $r \in 0^+(C) \cap \Q^p$ be a rational vector 
in the recession cone of $C$. 
We claim that also $C'=C+\spann(\{ r \})$ is lattice 
point free. Indeed, if $\bar{x}-\mu r \in \ri(C')$ is 
integer with 
$\mu > 0$ and $\bar{x} \in \ri(C)$, then there exists a 
positive integer $\mu^I > \mu$ such that 
$\bar{x}-\mu r + \mu^I r=\bar{x}+(\mu^I-\mu) r \in \ri(C)\cap \Z^p$, 
which contradicts that $C$ is lattice point free. Since the 
recession cone of a split polyhedron is rational, we therefore 
have
\begin{lemma}\label{inclusion_property_unbounded}
Let $L \subseteq \R^p$ be a split polyhedron. Then $L$ can be 
written in the form 
$L=\mathcal{P}+\mathcal{L}$, where $\mathcal{P} \subseteq\R^p$ 
is a rational polytope and 
$\mathcal{L} \subseteq \R^p$ is a linear space with 
an integer basis.
\end{lemma}

Observe that Lemma \ref{inclusion_property_unbounded} implies that 
every split polyhedron $L \subseteq \R^p$ is 
full dimensional. Indeed, if this was 
not the case, then we would have 
$L \subseteq \{ x : \R^p : \pi^T x = \pi_0 \}$ for some 
$(\pi,\pi_0) \in \Z^{p+1}$ which implies 
$L \subseteq \{ x : \R^p : \pi_0 \leq \pi^T x \leq \pi_0+1 \}$, and 
this contradicts that $L$ is maximal and lattice point free.

\begin{lemma}\label{splitbody_bounded_full_rational}
Every split polyhedron $L$ in $\R^p$ is full 
dimensional.
\end{lemma}

We are interested in using split polyhedra to characterize 
\emph{mixed} integer sets. Let $L^x \subseteq \R^p$ be 
a split polyhedron. We can then use the set 
$L:=\{ (x,y) \in \R^p \times \R^q : x \in L^x \}$ for 
mixed integer sets. We call $L$ a 
\emph{mixed integer split polyhedron}.

We now consider how to measure the size of a 
mixed integer split polyhedron. 
Let $L \subseteq \R^n$ be a mixed integer split polyhedron 
in $\R^n$ written in 
the form 
$$L:=\{ x \in \R^n : (\pi^k)^T x \geq \pi^k_0 \textrm{ for }k \in N_f(L) \},$$ 
where $N_f(L):=\{1,2,\ldots,n_f(L) \}$, 
$n_f(L)$ denotes 
the number of facets of $L$, $(\pi^k,\pi^k_0) \in \Z^{n+1}$ 
for $k \in N_f(L)$ and $\pi^k_j = 0$ for $j \notin N_I$. 
We assume that for every $k \in N_f(L)$, 
$\pi^k_0$ does not have a common divisor 
with all the integers $\pi^k_j$ for $j=1,2,\ldots,p$. Note that, 
since $L$ is full dimensional, the representation of $L$ 
under this assumption is unique.

Given a vector $v \in \Z^n$ that satisfies $v_j = 0$ for $j \notin N_I$, 
the number of parallel hyperplanes $v^T x = v_0$ 
that intersect a mixed integer split polyhedron 
$L \subseteq \R^n$ for varying $v_0 \in \R$ gives a 
measure of how wide $L$ is along 
the vector $v$. Define 
$Z(N_I):=\{ v \in \Z^n : v_j = 0 \textrm{ for all } j \notin N_I \}$. 
The \emph{width} of $L$ along a vector $v \in Z(N_I)$ is 
defined to be the number 
$$w(L,v):=\max\{ v^T x : x \in L \} - \min\{ v^T x : x \in L \}.$$
By considering the width of $L$ along all the facets of $L$, and 
choosing the largest of these numbers, we obtain a measure of how 
wide $L$ is.

\begin{definition}
(The max-facet-width of a mixed integer split polyhedron).\\
Let $L \subseteq \R^n$ be a mixed integer split polyhedron, and let 
$(\pi^k)^T x \geq \pi^k_0$ 
denote the facets of $L$, where 
$k \in N_f(L)$, $(\pi^k,\pi^k_0) \in \Z^{n+1}$ and $\pi^k \in Z(N_I)$. 
The \emph{max-facet-width} of $L$ is defined to be the number 
$$w_f(L):=\max\{ w(L,\pi^k) : k \in N_f(L) \}.$$
\end{definition}


The max-facet-width measures the 
size of a mixed integer split polyhedron. We now 
use this measure to also measure the 
size of a general mixed integer lattice point free rational polyhedron. 
For this, we use the following result 
proven in \cite{farkas} : for every 
mixed integer lattice point free rational polyhedron 
$Q \subseteq \R^n$, 
there exists a mixed integer split polyhedron 
$L \subseteq \R^n$ that satisfies 
$\ri(Q) \subseteq \intt(L)$. Hence there 
exists a mixed integer split polyhedron $L$ that excludes 
at least the same points as $Q$. A natural measure 
of the size of $Q$ is then the smallest max-facet-width 
of a mixed integer split polyhedron with this property.

\begin{definition}(Width size of any mixed integer lattice point free 
rational polyhedron)\newline
Let $Q \subseteq \R^n$ be a mixed integer lattice point free 
rational polyhedron. 
The \emph{width size} of $Q$ is defined to be the number 
\begin{alignat}{2}
\textrm{width-size}(Q) & :=\min\{ \textrm{max-facet-width}(L) : L \textrm{ is a mixed integer split polyhedron s.t. }\ri(Q) \subseteq \intt(L) \}.\notag
\end{alignat}
\end{definition}

\subsection{Polyhedral relaxations from mixed integer split polyhedra}

As mentioned in the introduction, 
any mixed integer lattice point free convex set $C \subseteq \R^n$ 
gives a relaxation of $\conv(P_I)$ 
\begin{alignat}{2}
R(C,P) & := \conv(\{x \in P : x \notin \ri(C) \})\notag
\end{alignat}
that satisfies $\conv(P_I) \subseteq R(C,P) \subseteq P$. 
Since mixed integer split polyhedra $L$ are maximal wrt. inclusion, 
the sets $R(L,P)$ for mixed integer split polyhedra $L$ are as 
tight relaxations as possible wrt. this operation. 
Figure 
\ref{first_fig} shows the set $R(L,P)$ for a 
polytope $P$ with five vertices and a split 
polyhedron $L$. 

\begin{figure}
\centering
\mbox{\subfigure[A polytope $P$ and a split polyhedron $L$]{\includegraphics[width=5cm]{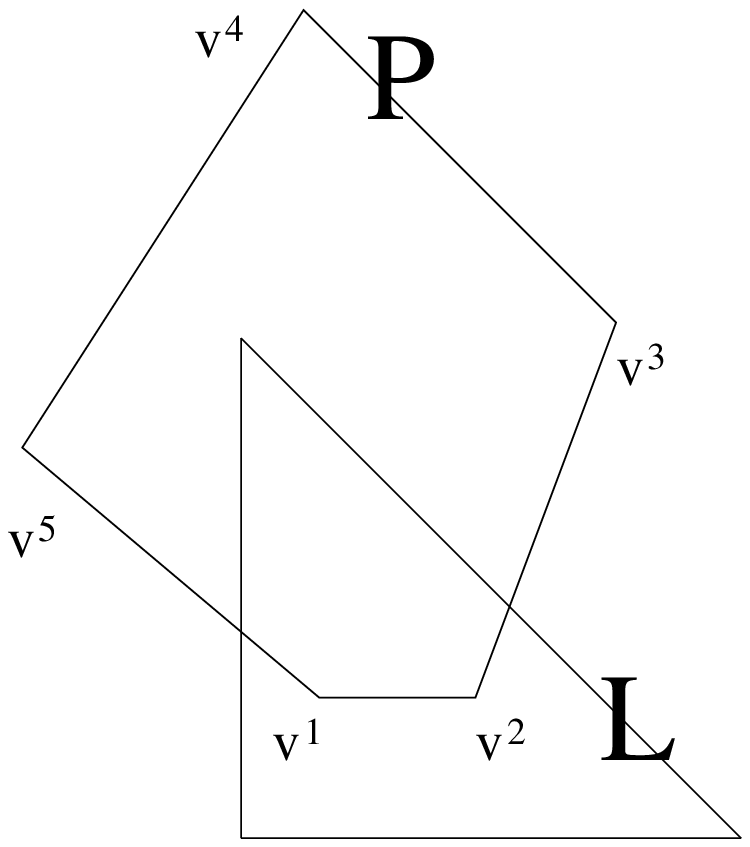}}\quad
\subfigure[The only cut that can be derived from $L$]{\includegraphics[width=5cm]{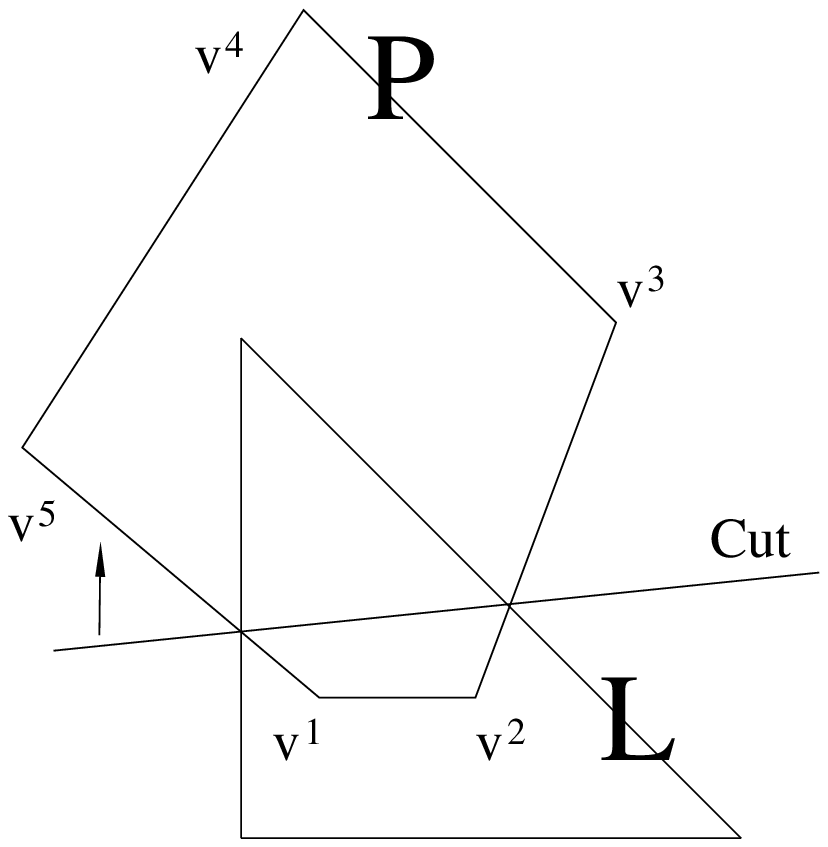}}
\subfigure[The strengthened relaxation of $P_I$]{\includegraphics[width=5cm]{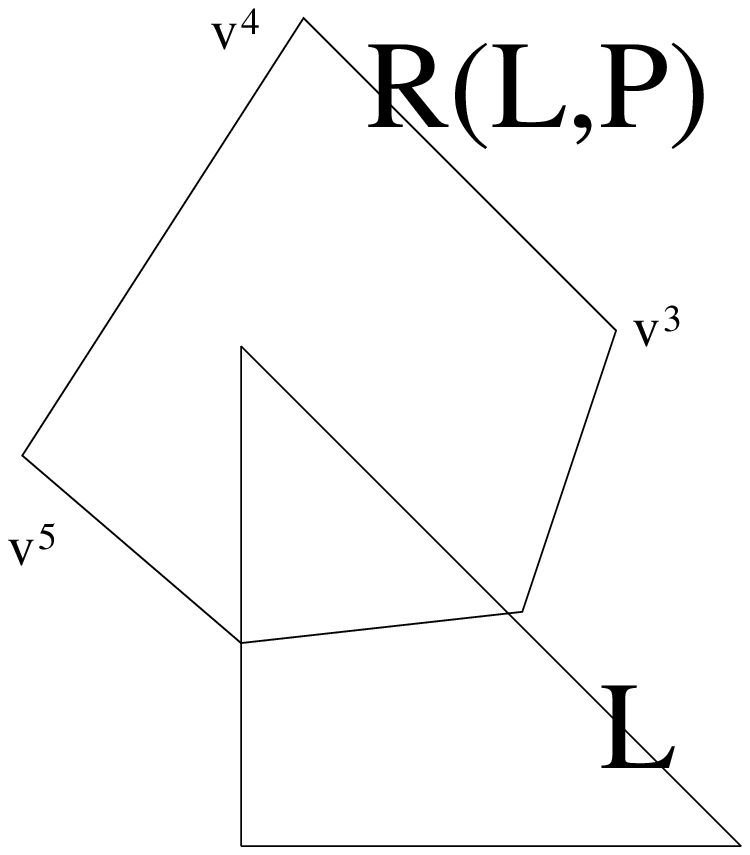}}}
\caption{Strengthening a linear relaxation $P$ by using a split
polyhedron $L$}
\label{first_fig}
\end{figure}

For the example in Figure \ref{first_fig}, the set 
$R(L,P)$ is a polyhedron. We now show that, in general, 
mixed integer split polyhedra give polyhedral relaxations 
$R(L,P)$ of $P_I$.
\begin{lemma}\label{poly_lem}
Let $L \subseteq \R^n$ be a full dimensional polyhedron whose recession 
cone $0^+(L)$ is a linear space. Then the following set $R(L,P)$ 
is a polyhedron.
$$R(L,P) := \conv(\{ x \in P : x \notin \intt(L) \}).$$
\end{lemma}
\begin{proof}
Let 
$(l^i)^T x \geq \l^i_0$ for $i \in I$ denote 
the facets of $L$, where $I:=\{1,2,\ldots,n_f \}$ and 
$n_f$ denotes the number 
of facets of $L$. Also suppose $P=\{ x \in \R^n : D x \leq d \}$, 
where $D \in \Q^{m \times n}$ and $d \in \Q^m$. 
Observe that $L$ has the property 
that, if $x^r \in 0^+(L)$, 
then $(l^i)^T x^r = 0$ for all $i \in I$. This 
follows from the fact that the recession cone 
$0^+(L)$ of $L$ is a linear space. 
We claim $R(L,P)$ is the projection 
of the following polyhedron onto the space of 
$x$-variables.
\begin{alignat}{2}
  x  & = \sum_{i \in I} x^i, & \label{eqn1}\\
  D x^i & \leq \lambda^i d, & \textrm{ for }i \in I,\label{eqn2}\\
  (l^i)^T x^i & \leq \lambda^i l^i_0, & \textrm{ for }i \in I, \label{eqn3}\\
  \sum_{i \in I} \lambda^i & = 1, & \label{eqn4}\\
   \lambda^i & \geq 0, & \textrm{ for }i \in I.\label{eqn5}
\end{alignat}

The above construction was also 
used by Balas 
for disjunctive programming \cite{Balas_disj}. 
Let $S(L,P)$ denote the set of $x \in \R^n$ that can be 
represented in the form (\ref{eqn1})-(\ref{eqn5}) above. 
We need to prove $R(L,P)=S(L,P)$. A result in 
Cornu\'ejols \cite{Gerard} shows that 
$\textrm{cl}( \conv(\cup_{i \in I} P^i))=\textrm{cl}(R(L,P))=S(L,P)$, 
where $P^i:=\{ x \in \R^n : D x \leq d\textrm{ and }(l^i)^T x \leq
l^i_0 \}$ for $i \in I$. 
It follows that $R(L,P) \subseteq S(L,P)$, so 
we only have to show the other inclusion.

We now show $S(L,P) \subseteq R(L,P)$. Let 
$\bar{x} \in S(L,P)$. By definition this means 
there exists $\{ \bar{x}^i {\}}_{i \in I}$ 
and $\{ \bar{\lambda}^i {\}}_{i \in I}$ such that 
$\bar{x}$, 
$\{ \bar{x}^i {\}}_{i \in I}$ and 
$\{ \bar{\lambda}^i {\}}_{i \in I}$ satisfy 
(\ref{eqn1})-(\ref{eqn5}). Let 
$\bar{I}:=\{ i \in I : \bar{x}^i \neq 0 \}$. 
We can assume $|\bar{I}|$ is as small as possible. Furthermore 
we can assume $|\bar{I}| \geq 2$.

Let $\bar{I}^0:=\{ i \in \bar{I} : \bar{\lambda}^i = 0 \}$, 
and let $i_0 \in \bar{I}^0$ be arbitrary. 
We claim $\bar{x}^{i_0} \in 0^+(R(L,P))$. 
To show this, 
we first argue that there exists $i' \in I$ 
such that $(l^{i'})^T \bar{x}^{i_0} < 0$. Suppose, for a 
contradiction, that $(l^i)^T \bar{x}^{i_0} \geq 0$ for 
all $i \in I$. This implies $x^{i_0} \in 0^+(L)$, 
and therefore $(l^i)^T \bar{x}^{i_0} = 0$ for 
all $i \in I$. We now show this contradicts the 
assumption that $|\bar{I}|$ is as small as possible. 
Indeed, choose $\bar{i} \in \bar{I} \setminus \{i_0\}$ 
arbitrarily. Define 
$\tilde{x}^{\bar{i}}:=$
$\bar{x}^{i_0}+$
$\bar{x}^{\bar{i}}$, 
$\tilde{x}^{i_0}:=0$, 
$\tilde{x}^i:=\bar{x}^i$ 
for $i \in I \setminus \{ i_0,\bar{i} \}$ and 
$\tilde{\lambda}^i:=\bar{\lambda}^i$ for $i \in I$. We have 
that $\bar{x}$, 
$\{ \tilde{x}^i {\}}_{i \in I}$ and 
$\{ \tilde{\lambda}^i {\}}_{i \in I}$ satisfy 
(\ref{eqn1})-(\ref{eqn5}), and 
$\{ \tilde{x}^i {\}}_{i \in I}$ 
gives a representation of $\bar{x}$ with 
fewer non-zero vectors than 
$\{ \bar{x}^i {\}}_{i \in I}$. 
This contradicts the minimality of 
$|\bar{I}|$. Therefore there exists $i' \in I$ 
such that $(l^{i'})^T \bar{x}^{i_0} < 0$. 

We can now show $\bar{x}^{i_0} \in 0^+(R(L,P))$. 
Let $x^R \in R(L,P)$ be arbitrary and 
define $\bar{x}^{i_0}(\alpha):=$
$x^R+ \alpha \bar{x}^{i_0}$ 
for $\alpha\geq 0$. Since $(l^{i'})^T \bar{x}^{i_0} < 0$, 
there exists $\bar{\alpha}>0$ such that 
$(l^{i'})^T \bar{x}^{i'_0}(\alpha) \leq l^{i'}_0$ for 
all $\alpha \geq \bar{\alpha}$. This implies 
$\bar{x}^{i_0}(\alpha) \in R(L,P)$ 
for all $\alpha \geq \bar{\alpha}$. 
Hence $\bar{x}^{i_0} \in 0^+(R(L,P))$. 

We can now write $\bar{x}=$
$\sum_{i \in \bar{I}^{>0}}$ 
$\bar{\lambda}^i
\frac{\bar{x}^i}{\bar{\lambda}^i}+$
$\sum_{i \in \bar{I}^0} \bar{x}^i$, where 
$\bar{I}^{>0}:=\{ i \in \bar{I} : \bar{\lambda}^i > 0 \}$, 
$\frac{\bar{x}^i}{\bar{\lambda}^i} \in R(L,P)$ for $i \in \bar{I}^{>0}$, 
$\bar{x}^i \in 0^+(R(L,P))$ for $i \in \bar{I}^0$, 
$\sum_{i \in \bar{I}^{>0}} \bar{\lambda}^i = 1$ and 
$\bar{\lambda}^i > 0$ for $i \in \bar{I}^{>0}$. 
Therefore $\bar{x} \in R(L,P)$.
\end{proof}

Lemma \ref{poly_lem} implies that, for every 
\emph{finite} collection $\cal{L}$ of mixed integer 
split polyhedra, the set 
$$\textrm{Cl}(P,{\cal{L}}):=\cap_{L \in {\cal{L}}} R(L,P),$$
is a polyhedron. A next natural question is 
under which conditions the same is true for an 
\emph{infinite} collection of mixed integer 
split polyhedra. As mentioned, we will show that a 
sufficient condition for this to be the case is that it is 
possible to provide an upper bound $w^*$ on the 
max-facet-width of the mixed integer split polyhedra 
in an infinite collection $\cal{L}$ of mixed integer 
split polyhedra. Therefore, we consider the set of \emph{all} 
mixed integer split polyhedra whose max-facet-width is bounded 
by a given constant $w>0$
$${\cal{L}}^w:= \{ L \subseteq \R^n : L\textrm{ is a mixed integer split polyhedron satisfying
}w_f(L) \leq w \}.$$
An extension of the (first) split closure can 
now be defined.
\begin{definition}
(The $w^{th}$ split closure).\newline
Given $w>0$, the \emph{$w^{th}$ split closure} 
of $P$ is defined to be the set
$$\textrm{Cl}_w(P,{\cal{L}}^w):=\cap_{L \in {\cal{L}}^w} R(L,P).$$
\end{definition}

A natural question is which 
condition a mixed integer split polyhedron $L$ must satisfy in 
order to have $R(L,P) \neq P$. The following lemma shows 
that $R(L,P) \neq P$ exactly when there is a vertex 
of $P$ in the interior of $L$.
\begin{lemma}\label{int_vert_nec}
Let $L \subset \R^n$ be a mixed integer split polyhedron. 
Then $R(L,P) \neq P$ if and only if there is a vertex 
of $P$ in the interior of $L$.
\end{lemma}
\begin{proof}
If $v^i$ is a vertex of $P$ 
in the interior of $L$, where $i \in V$, then 
$v^i$ can not be expressed as a convex combination of 
points in $P$ that are not in the interior of $L$, 
and therefore $v^i \notin R(L,P)$. 
Conversely, when $L$ does \emph{not} contain a vertex 
of $P$ in its interior, then 
$\delta^T v^i \geq \delta_0$ for every valid inequality 
$\delta^T x \geq \delta_0$ for $R(L,P)$ and $i \in V$. Since the extreme 
rays of $R(L,P)$ are the same as the extreme rays of $P$, 
we have $\delta^T r^j \geq 0$ for every extreme ray 
$j \in E$.
\end{proof}

\section{Cutting planes from inner representations of polyhedra}

The focus in this section is on analyzing the effect of adding 
cutting planes (or cuts) to the linear relaxation $P$ of $P_I$. 
We define cuts to be inequalities that cut off some 
vertices of $P$. In other words, 
we say an inequality 
$\delta^T x \ge \delta_0$ is a cut for $P$ 
if $\delta^T v^i < \delta_0$ 
for some $i \in V$. 
Let $V^c_{(\delta,\delta_0)}:=$
$\{ i \in V : \delta^T v^i < \delta_0 \}$ 
index the vertices of $P$ that are cut off 
by $\delta^T x \ge \delta_0$, and let 
$V^s_{(\delta,\delta_0)}:=$
$\{ i \in V : \delta^T v^i \geq \delta_0 \}$ 
index the vertices of $P$ that 
satisfy $\delta^T x \ge \delta_0$. 

A cut $\delta^T x \ge \delta_0$ is called 
\emph{non-negative} if $\delta^T r^j \geq 0$ 
for all $j \in E$. 
Throughout this section we only 
consider non-negative cutting planes. Observe that 
non-negativity is a 
necessary condition for valid cuts for a mixed 
integer set. Indeed, if $\delta^T x \ge \delta_0$ is a valid 
cut for the mixed integer points in $P$, and $j \in E$ 
is an arbitrary extreme ray of $P$, then the halfline 
$\{ x^I + \mu r^j : \mu \geq 0 \}$ contains an infinite 
number of mixed integer points for any mixed integer point 
$x^I \in P$. Therefore, if we had $\delta^T r^j < 0$ for some 
extreme ray $r^j$ of $P$, a 
contradiction to the validity of $\delta^T x \ge \delta_0$ 
for the mixed integer points in $P$ would be obtained.

\subsection{The vertices created by the addition of a cut}
Adding a non-negative cut $\delta^T x \ge \delta_0$ to the 
linear relaxation $P$ of $P_I$ creates a polyhedron with different 
vertices than $P$. We now analyze the new vertices that are 
created. For simplicity let 
$\Lambda:=\{ \lambda \in \R_+^{|V|} : \sum_{i \in V} \lambda_i = 1 \}$, 
$\Lambda^c_{(\delta,\delta_0)}:=\{ \lambda \in \Lambda : \sum_{i \in V^c_{(\delta,\delta_0)}} \lambda_i = 1 \}$ and 
$\Lambda^s_{(\delta,\delta_0)}:=\{ \lambda \in \Lambda : \sum_{k \in V^s_{(\delta,\delta_0)}} \lambda_k = 1 \}$. 
Also, for any $\lambda \in \Lambda$, define $v_{\lambda}:=\sum_{i \in V} \lambda_i v^i$, and for any $\mu \geq 0$, define 
$r_{\mu}:=\sum_{j \in E} \mu_j r^j$. 
We now argue that the new vertices that are created by adding 
the cut $\delta^T x \ge \delta_0$ to $P$ are 
\emph{intersection points} \cite{balasint}. 
Intersection points are defined as follows. 
Given an extreme 
ray $j \in E$ that satisfies 
$\delta^T r^j > 0$, and a convex combination $\lambda^c \in \Lambda^c_{(\delta,\delta_0)}$, the halfline 
$\{ v_{\lambda} + \alpha r^j : \alpha \geq 0 \}$ 
intersects the hyperplane 
$\{ x \in \R^n : \delta^T x = \delta_0 \}$. 
For $j \in E$ and $\lambda^c \in \Lambda^c_{(\delta,\delta_0)}$, define  

\begin{alignat}{2} \label{alpha-prime}
{\alpha}_j'(\delta,{\delta}_0,\lambda^c) & := \left\{ \begin{array}{ll}
\frac{\delta_0 - \delta^T v_{\lambda^c}}{\delta^T r^j} & \textrm{if   } 
\delta^T r^j > 0, \\
+\infty  & \textrm{otherwise. }
\end{array} \right.
\end{alignat}

\noindent The number ${\alpha}_j'(\delta,{\delta}_0,\lambda^c)$ is 
the value of $\alpha$ for which the point $v_{\lambda^c}+\alpha r^j$ 
is on the hyperplane ${\delta}^T x = {\delta}_0$. 
When there is no such point, we 
define ${\alpha}_j'(\delta,{\delta}_0,\lambda^c) = +\infty$. 
If ${\alpha}_j'(\delta,{\delta}_0,\lambda^c) < +\infty$, the point 
$v_{\lambda^c}+ {\alpha}_j'(\delta,{\delta}_0,\lambda^c) r^j$ is 
called the intersection point associated with the 
convex combination $\lambda^c \in \Lambda^c_{(\delta,\delta_0)}$ 
and the extreme ray $r^j$ of $P$.
Observe that ${\alpha}_j'(\delta,{\delta}_0,\lambda^c)$ is 
linear in $\lambda^c$.

Given a convex combination $\lambda^c \in \Lambda^c_{(\delta,\delta_0)}$, 
and a vertex $k \in V^s_{(\delta,\delta_0)}$, the line segment 
between $v_{\lambda}$ and $v^k$ 
intersects the hyperplane 
$\{ x \in \R^n : \delta^T x = \delta_0 \}$. 
For $k \in V$ and 
$\lambda^c \in \Lambda^c_{(\delta,\delta_0)}$, define  

\begin{alignat}{2} \label{beta-prime}
{\beta}_k'(\delta,{\delta}_0,\lambda^c) & := \left\{ \begin{array}{ll}
\frac{\delta_0 - \delta^T v_{\lambda^c}}{\delta^T ( v^k-v_{\lambda^c})} & \textrm{if   } 
k \in V^s_{(\delta,\delta_0)}, \\
+\infty  & \textrm{otherwise. }
\end{array} \right.
\end{alignat}

\noindent The number ${\beta}_k'(\delta,{\delta}_0,\lambda^c)$ denotes 
the value of $\beta$ for which the point 
$v_{\lambda^c}+\beta (v^k-v_{\lambda^c})$ 
is on the hyperplane ${\delta}^T x = {\delta}_0$. Observe that 
${\beta}_k'(\delta,{\delta}_0,\lambda^c) \in ]0,1]$ whenever 
${\beta}_k'(\delta,{\delta}_0,\lambda^c) < +\infty$. If 
${\beta}_k'(\delta,{\delta}_0,\lambda^c) < +\infty$, the point 
$v_{\lambda^c}+{\beta}_k'(\delta,{\delta}_0,\lambda^c)
(v^k-v_{\lambda^c})$ 
is called the intersection point associated with the 
convex combination $\lambda^c \in \Lambda^c_{(\delta,\delta_0)}$ 
and the vertex $v^k$ of $P$. For 
the polytope $P$ of Figure \ref{first_fig} and a cut 
${\delta}^T x \ge {\delta}_0$, Figure 
\ref{cut_fig} gives an example of how to compute the 
intersection points for a given 
convex combination $\lambda^c=(\frac{1}{2},\frac{1}{2})$.

\begin{figure}
\centering
\mbox{\subfigure[The polytope $P$ from Figure \ref{first_fig} and a cut with $V^c(\delta,\delta_0)=\{1,2\}$]{\includegraphics[width=5cm]{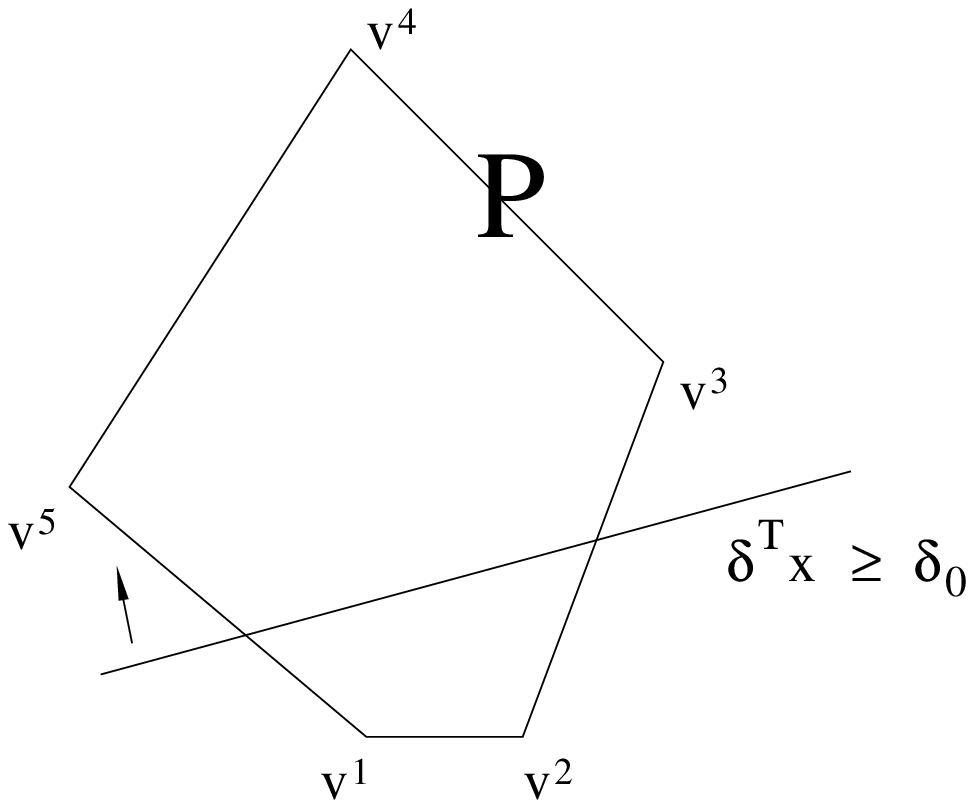}}\quad
\subfigure[The convex combination $v_{\lambda^c}$ of $v^1$ and $v^2$ for $\lambda^c=(\frac{1}{2},\frac{1}{2})$]{\includegraphics[width=5cm]{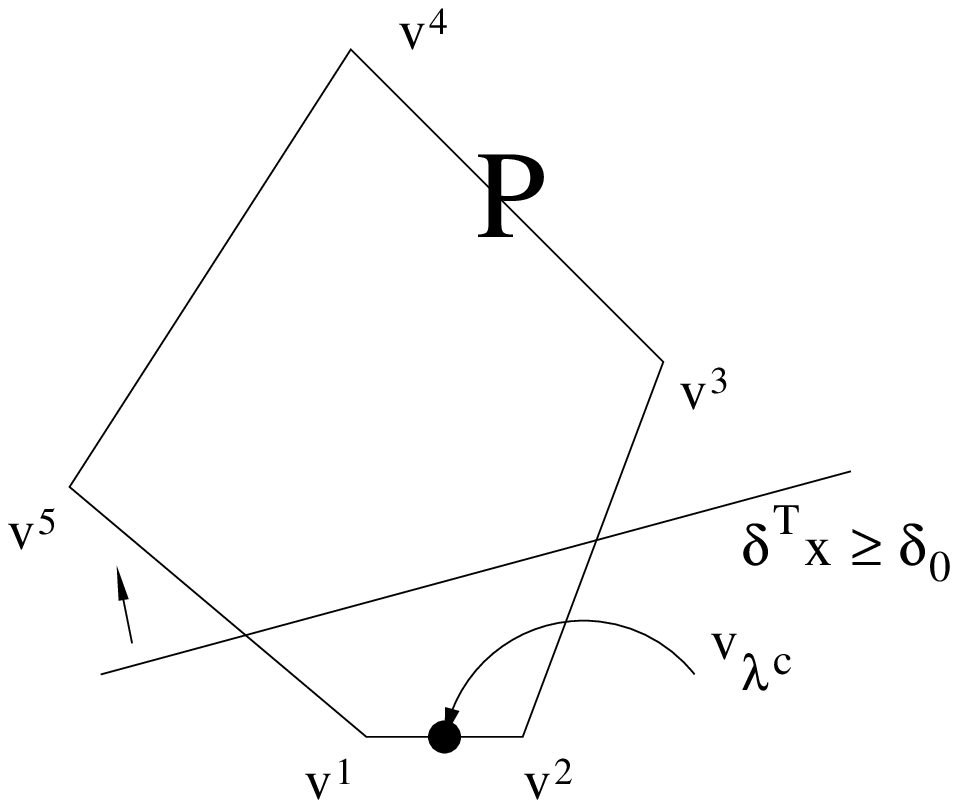}}
\subfigure[The line segments that determine 
$\beta_k'(\delta,\delta_0,\lambda^c)$ for $k=3,4,5$]{\includegraphics[width=5cm]{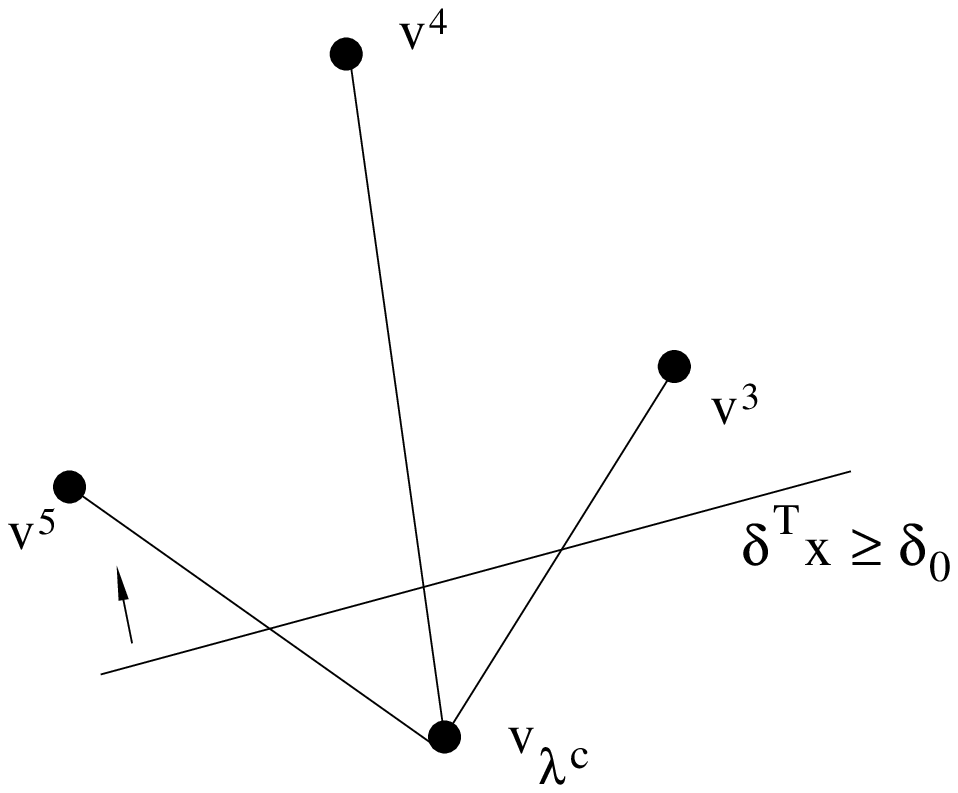}}}
\caption{Determining the intersection points from a polytope $P$ and a
cut ${\delta}^T x \geq {\delta}_0$}
\label{cut_fig}
\end{figure}

An important property of an intersection point of the form 
$v_{\lambda^c}+{\beta}_k'(\delta,{\delta}_0,\lambda^c) v^k$ for 
$k \in V^s_{(\delta,\delta_0)}$ and 
$\lambda^c \in \Lambda^c_{(\delta,\delta_0)}$ is the following. 
\begin{lemma}\label{beta_prime_properties}
Let $\delta^T x \geq \delta_0$ be a non-negative cut, and let $k \in V^s_{(\delta,\delta_0)}$. 
For every $\lambda^c \in \Lambda^c_{(\delta,\delta_0)}$, the intersection 
point 
$v_{\lambda^c}+\beta_k'(\delta,{\delta}_0,\lambda^c) (v^k-v_{\lambda^c})$ 
is a convex combination of the intersection points 
$v^i+\beta_k'(\delta,{\delta}_0,e^i) (v^k-v^i)$ 
for $i \in V^c_{(\delta,\delta_0)}$. 
\end{lemma}
\begin{proof}
Define 
$C:=\conv(\{ v^k \} \cup \{ v^i {\}}_{i \in V^c_{(\delta,\delta_0)}})$. 
Trivially we have 
$v_{\lambda^c}+\beta_k'(\delta,{\delta}_0,\lambda^c) (v^k-v_{\lambda^c}) \in C$. 
We will show the vertices of the polytope 
$\{ x \in C : \delta^T x = \delta_0 \}$ are given by the points 
$v^i+\beta_k'(\delta,{\delta}_0,e^i) (v^k-v^i)$ 
for $i \in V^c_{(\delta,\delta_0)}$ from which the result follows. 
If $\delta^T v^k = \delta_0$, the result is trivial, so we 
assume $\delta^T v^k > \delta_0$.

Therefore suppose $\bar{x} \in \{ x \in C : \delta^T x = \delta_0 \}$ is 
a vertex of $\{ x \in C : \delta^T x = \delta_0 \}$. We may write 
$\bar{x}=\lambda_0 v^k + \sum_{i \in V^c_{(\delta,\delta_0)}} \lambda_i v^i$, 
where $\lambda_0+\sum_{i \in V^c_{(\delta,\delta_0)}} \lambda_i =1$, 
$\lambda_0 \geq 0$ and $\lambda_i \geq 0$ for all $i \in V^c_{(\delta,\delta_0)}$. Using 
$\lambda_0=1-\sum_{i \in V^c_{(\delta,\delta_0)}} \lambda_i$, we can write 
$\bar{x}=v^k+\sum_{i \in V^c_{(\delta,\delta_0)}} \lambda_i (v^i-v^k)$. 
Multiplying with $\delta$ on both sides then gives 
$\sum_{i \in V^c_{(\delta,\delta_0)}} \frac{\lambda_i}{\eta_{i,k}}=1$, 
where $\eta_{i,k}:=\frac{\delta^T v^k-\delta_0}{\delta^T (v^k-v^i)}$. 

We can now write 
$\bar{x}=v^k+\sum_{i \in V^c_{(\delta,\delta_0)}} \lambda_i (v^i-v^k)=$
$\sum_{i \in V^c_{(\delta,\delta_0)}} \frac{\lambda_i}{\eta_{i,k}} v^k+$
$\sum_{i \in V^c_{(\delta,\delta_0)}} \lambda_i (v^i-v^k)=$
$\sum_{i \in V^c_{(\delta,\delta_0)}} \frac{\lambda_i}{\eta_{i,k}} (v^k+\eta_{i,k}(v^i-v^k))$. 
Since $v^k+\eta_{i,k}(v^i-v^k)=v^i+\beta_k'(\delta,{\delta}_0,e^i) (v^k-v^i)$ for 
$i \in V^c_{(\delta,\delta_0)}$, the result follows.
\end{proof} 

Lemma \ref{beta_prime_properties} shows that the only 
vectors $\lambda^c \in \Lambda^c_{(\delta,\delta_0)}$ 
for which the intersection points of the type 
$v_{\lambda^c}+{\beta}_k'(\delta,{\delta}_0,\lambda^c)
(v^k-v_{\lambda^c})$ can be vertices of 
$\{ x \in P : \delta^T x \geq \delta_0 \}$ are the 
unit vectors.

In order to characterize the vertices of 
$\{ x \in P : \delta^T x \geq \delta_0 \}$, we 
first give a representation of 
$\{ x \in P : \delta^T x \geq \delta_0 \}$ in 
a higher dimensional space. 
Note that any point which is a convex combination of the vertices of 
$P$ can be written as a convex combination of two points 
$v_{\lambda^s}$ and $v_{\lambda^c}$, where 
$\lambda^s \in \Lambda^s_{(\delta,\delta_0)}$ and 
$\lambda^c \in \Lambda^c_{(\delta,\delta_0)}$. We may 
therefore write $P$ in the form
$$P=\{ x \in \R^n : x= x^v+r_{\mu}, \mu \geq 0, \lambda^s \in
\Lambda^s_{(\delta,\delta_0)}, \lambda^c \in
\Lambda^c_{(\delta,\delta_0)}\textrm{ and }x^v \in \conv(v_{\lambda^c},v_{\lambda^s})\}.$$
Consider the set obtained from $P$ by fixing the convex combination 
$\lambda^c \in \Lambda^c_{(\delta,\delta_0)}$
$$P(\lambda^c):=\{ x \in \R^n : x= x^v +r_{\mu}, \mu \geq 0, \lambda^s
\in \Lambda^s_{(\delta,\delta_0)} \textrm{ and }x^v \in \conv(v_{\lambda^c},v_{\lambda^s})\}.$$
Observe that we may write $P(\lambda^c)$ in the form 
$$P(\lambda^c)=\{ x \in \R^n : x= v_{\lambda^c}+\sum_{k \in V^s_{(\delta,\delta_0)}} \epsilon_k (v^k-v_{\lambda^c})+r_{\mu}, \mu \geq 0 \textrm{ and } \epsilon \in \Lambda\}.$$

Now consider the set $P^l(\lambda^c)$ 
obtained from $P(\lambda^c)$ by also 
considering the multipliers on the 
vertices of $P$ indexed by $V^s_{(\delta,\delta_0)}$, and 
the extreme rays of $P$
$$P^l(\lambda^c):=\{ (x,\epsilon,\mu) \in \R^{n+|V|+|E|} : x=v_{\lambda}+\sum_{k \in V^s_{(\delta,\delta_0)}} \epsilon_k (v^k-v_{\lambda^c})+r_{\mu}, \mu \geq 0 \textrm{ and }\epsilon \in \Lambda\}.$$

The scalars 
${\alpha}_j'(\delta,{\delta}_0,\lambda^c)$ for $j \in E$ and 
${\beta}_i'(\delta,{\delta}_0,\lambda^c)$ for $i \in V$ 
give an alternative description of the set of points in 
$P$ that satisfy 
$\delta^T x \geq \delta_0$ in a higher dimensional space. 
\begin{lemma}\label{fix_lambda}
(Lemma 2 in \cite{andcorli}). 
Let ${\delta}^T x \ge {\delta}_0$ be a non-negative 
cut for $P$. For any $\lambda^c \in \Lambda^c_{(\delta,\delta_0)}$, 
we have
$$\{ (x,\epsilon,\mu) \in P^l(\lambda^c) : \delta^T x \geq \delta_0
\}= \{ (x,\epsilon,\mu) \in P^l(\lambda^c) : \sum_{j \in E}
\frac{\mu_j}{{\alpha}_j'(\delta,{\delta}_0,\lambda^c)}+ \sum_{k \in V} \frac{\epsilon_k}{{\beta}_k'(\delta,{\delta}_0,\lambda^c)} \geq 1 \}.$$
\end{lemma}
\begin{proof}
We have $(\bar{x},\bar{\epsilon},\bar{\mu}) \in P^l(\lambda^c)$ and 
$\delta^T \bar{x} \ge \delta_0 \iff$ 
$\bar{x}=v_{\lambda^c}+\sum_{k \in V^s_{(\delta,\delta_0)}}
\bar{\epsilon}_k (v^k-v_{\lambda^c})+r_{\bar{\mu}}$, where
$\bar{\epsilon},\bar{\mu} \ge 0$, 
$\sum_{k \in V^s_{(\delta,\delta_0)}} \bar{\epsilon}_k \leq 1$ 
and $\delta^T \bar{x} \ge \delta_0 \iff$
$\bar{x}=v_{\lambda^c}+\sum_{k \in V^s_{(\delta,\delta_0)}}
\bar{\epsilon}_k (v^k-v_{\lambda^c})+r_{\bar{\mu}}$,
$\bar{\epsilon},\bar{\mu} \ge 0$, $\sum_{k \in V^s_{(\delta,\delta_0)}} \bar{\epsilon}_k \leq 1$ and 
$\sum_{k \in V^s_{(\delta,\delta_0)}} \bar{\epsilon}_k \delta^T (v^k-v_{\lambda^c})+\sum_{j \in E} \bar{\mu}_j (\delta^T r^j) \ge $
$(\delta_0 - \delta^T v_{\lambda^c})\iff$
$(\bar{x},\bar{\epsilon},\bar{\mu}) \in P^l(\lambda^c)$ and 
$\sum_{k \in V} \bar{\epsilon}_k / {\beta}_k'(\delta,\delta_0,\lambda^c)+$
${\sum}_{j \in E} \bar{\mu}_j / {\alpha}_j'(\delta,\delta_0,\lambda^c)
\ge 1$.
\end{proof}

Based on the above result, we can now 
characterize the vertices of 
$\{ x \in P : \delta^T x \geq \delta_0 \}$. Specifically we show that 
every vertex of $\{ x \in P : \delta^T x \geq \delta_0 \}$ is either 
a vertex of $P$ that satisfies $\delta^T x \geq \delta_0$, or an 
intersection point obtained from a vertex of $P$ that violates 
$\delta^T x \geq \delta_0$.
\begin{lemma}\label{vert_prime_char}
Let $\delta^T x \geq \delta_0$ be a non-negative cut for $P$. 
The vertices of $\{ x \in P : \delta^T x \geq \delta_0 \}$ are:
\begin{enumerate}
\item[(i)] vertices $v^k$ of $P$ with $k \in V^s_{(\delta,\delta_0)}$,
\item[(ii)] intersection points $v^i+\beta_k'(\delta,{\delta}_0,e^i) (v^k-v^i)$, 
where $i \in V^c_{(\delta,\delta_0)}$ and $k \in V^s_{(\delta,\delta_0)}$, and 
\item[(iii)] intersection points $v^i+\alpha_j'(\delta,{\delta}_0,e^i) r^j$, 
where $i \in V^c_{(\delta,\delta_0)}$ and $j \in E$ satisfies 
$\delta^T r^j > 0$.
\end{enumerate}
\end{lemma}
\begin{proof}
Let $\bar{x} \in \{ x \in P : \delta^T x \geq \delta_0 \}$ be a vertex of 
$\{ x \in P : \delta^T x \geq \delta_0 \}$. 
Also let $\lambda^c \in \Lambda^c_{(\delta,\delta_0)}$ 
and $(\bar{\epsilon},\bar{\mu})$ 
be such that $(\bar{x},\bar{\epsilon},\bar{\mu}) \in P^l(\lambda^c)$ and 
$\bar{x}=v_{\lambda^c}+\sum_{k \in V^s_{(\delta,\delta_0)}} \bar{\epsilon_k} (v^k-v_{\lambda^c})+\sum_{j \in E} \bar{\mu}_j r^j$.
Since $P(\lambda^c) \subseteq P$, we must have that 
$\bar{x}$ is a vertex of $\{ x \in P(\lambda^c) : \delta^T x \geq
\delta_0 \}$. We first show that 
$\bar{x}$ must be either of the form: (a) a vertex $v^k$ of $P$ with 
$k \in V^s_{(\delta,\delta_0)}$, (b) an intersection point 
$v_{\lambda^c}+\beta_k'(\delta,{\delta}_0,\lambda^c) (v^k-v_{\lambda^c})$ with 
$k \in V^s_{(\delta,\delta_0)}$, or 
(c) an intersection point 
$v_{\lambda^c}+\alpha_j'(\delta,{\delta}_0,\lambda^c) r^j$ 
with $j \in E$ satisfying $\delta^T r^j > 0$. 

Clearly, if $\bar{x}$ is a vertex of 
$\{ x \in P(\lambda^c) : \delta^T x \geq \delta_0 \}$ which is 
\emph{not} a vertex of $P$, we must have 
that $\bar{x}$ satisfies $\delta^T x \geq \delta_0$ 
with equality. From $\delta^T \bar{x} = \delta_0$, it follows from Lemma 
\ref{fix_lambda} that 
$(\bar{x},\bar{\epsilon},\bar{\mu}) \in P^l(\lambda^c)$ and 
$$\sum_{j \in E}
\frac{\bar{\mu}_j}{{\alpha}_j'(\delta,{\delta}_0,\lambda^c)}+\sum_{k
\in V^s_{(\delta,\delta_0)}}
\frac{\bar{\epsilon}_k}{{\beta}_k'(\delta,{\delta}_0,\lambda^c)} =1.$$
We can now write 
$$\bar{x}=\sum_{j \in E \setminus E^0} \eta_j (v_{\lambda^c} + \alpha_j'(\delta,{\delta}_0,\lambda^c) r^j)+\sum_{k \in V^s_{(\delta,\delta_0)}} \gamma_k (v_{\lambda^c}+\beta_k'(\delta,{\delta}_0,\lambda^c)(v^k - v_{\lambda^c}))+\sum_{j \in E^0} \bar{\mu}_j r^j,$$ 
where $E^0:=\{ j \in E : \delta^T r^j = 0 \}$, 
$\eta_j:=\frac{\bar{\mu}_j}{\alpha_j'(\delta,{\delta}_0,\lambda^c)}$ for $j \in E \setminus E^0$, 
$\gamma_k:=\frac{\bar{\epsilon}_k}{{\beta}_k'(\delta,{\delta}_0,\lambda^c)}$ for 
$k \in V^s_{(\delta,\delta_0)}$ and 
$\sum_{j \in E \setminus E^0} \eta_j+\sum_{k \in V^s_{(\delta,\delta_0)}} \gamma_k = 1$. 
Hence $\bar{x}$ must be of one of the forms (a)-(c) above.

We now show (i)-(iii). If $\bar{x}$ is a vertex $v^k$ of $P$, where 
$k \in V^s_{(\delta,\delta_0)}$, we are done, so we may assume 
that either 
$\bar{x}=v_{\lambda^c}+\beta_k'(\delta,{\delta}_0,\lambda^c) (v^k-v_{\lambda^c})$, 
where $k \in V^s_{(\delta,\delta_0)}$, or 
$\bar{x}=v_{\lambda^c}+\alpha_j'(\delta,{\delta}_0,\lambda^c) r^j$, 
where $j \in E$ satisfies 
$\alpha_j'(\delta,{\delta}_0,\lambda^c) < +\infty$. If $\bar{x}$ is of the 
form $\bar{x}=v_{\lambda^c}+\alpha_j'(\delta,{\delta}_0,\lambda^c) r^j$, 
we may write 
$\bar{x}=v_{\lambda^c}+\alpha_j'(\delta,{\delta}_0,\lambda^c) r^j=$
$v_{\lambda^c}+\frac{\delta_0-\delta^T v_{\lambda^c}}{\delta^T r^j} r^j=$
$\sum_{i \in V^c_{(\delta,\delta_0)}} \lambda_i (v^i+\frac{\delta_0-\delta^T v^i}{\delta^T r^j} r^j)$. 
Since $\alpha_j'(\delta,{\delta}_0,e^i)=\frac{\delta_0-\delta^T v^i}{\delta^T r^j}$ 
and $\bar{x}$ is a vertex of $\{ x \in P : \delta^T x \geq \delta_0 \}$, 
this implies $\lambda_{\bar{i}}=1$ for some $\bar{i} \in V^c_{(\delta,\delta_0)}$. 
Finally, 
if $\bar{x}$ is of the form 
$\bar{x}=v_{\lambda^c}+\beta_k'(\delta,{\delta}_0,\lambda^c) (v^k-v_{\lambda^c})$, 
then Lemma \ref{beta_prime_properties} shows that $\bar{x}$ is of the form 
$v^{\bar{i}}+\beta_k'(\delta,{\delta}_0,e^{\bar{i}}) (v^k-v^{\bar{i}})$ for 
some $\bar{i} \in V^c_{(\delta,\delta_0)}$ and 
$k \in V^s_{(\delta,\delta_0)}$.
\end{proof}

Lemma \ref{vert_prime_char} motivates the following notation 
for those intersection points 
$v_{\lambda^c}+\alpha_j'(\delta,{\delta}_0,\lambda^c) r^j$ and 
$v_{\lambda^c}+\beta_k'(\delta,{\delta}_0,\lambda^c) (v^k-v_{\lambda^c})$, 
where $\lambda^c$ is a unit vector. Given 
$i \in V^c_{(\delta,\delta_0)}$ and $j \in E$, define 
$\alpha'_{i,j}(\delta,{\delta}_0):=\alpha_j'(\delta,{\delta}_0,e^i)$, and 
given $i \in V^c_{(\delta,\delta_0)}$ and $k \in V^s_{(\delta,\delta_0)}$, 
define 
$\beta'_{i,k}(\delta,{\delta}_0):=\beta_k'(\delta,{\delta}_0,e^i)$.

\subsection{Dominance and equivalence between cuts}

Given two non-negative cuts $(\delta^1)^T x \geq \delta^1_0$ 
and $(\delta^2)^T x \geq \delta^2_0$ for $P$, it is not clear how to 
compare them in the space of the $x$ variables. 
By including the multipliers on the 
extreme rays and on the satisfied vertices in the description, such a comparison 
is possible. We assume all non-negative cuts considered in 
this section all cut off exactly 
the same set of vertices $V^c \subseteq V$ of $P$. 
Our notion of dominance is the following.
\begin{definition}\label{dom_def}
Let 
$({\delta}^1)^T x \ge {\delta}^1_0$ and 
$({\delta}^2)^T x \ge {\delta}^2_0$ be two 
non-negative cuts for $P$ that cut off the 
same set om vertices 
$V^c=V^c_{(\delta^1,\delta_0^1)}=V^c_{(\delta^2,\delta_0^2)}$ 
of $P$. 
\begin{enumerate}
\item[(i)] The cutting plane $({\delta}^1)^T x \ge {\delta}^1_0$ 
\emph{dominates} $({\delta}^2)^T x \ge {\delta}^2_0$ on $P$ iff 
$\{ x \in P : ({\delta}^1)^T x \ge {\delta}^1_0 \} \subseteq$
$\{ x \in P : ({\delta}^2)^T x \ge {\delta}^2_0 \}$. 
\item[(ii)] If 
$({\delta}^1)^T x \ge {\delta}^1_0$ dominates 
$({\delta}^2)^T x \ge {\delta}^2_0$ on $P$, and 
$({\delta}^2)^T x \ge {\delta}^2_0$ dominates 
$({\delta}^1)^T x \ge {\delta}^1_0$ on $P$, we say 
$({\delta}^1)^T x \ge {\delta}^1_0$ and 
$({\delta}^2)^T x \ge {\delta}^2_0$ are \emph{equivalent} 
on $P$.
\end{enumerate}
\end{definition}

We now show that an equivalent definition of 
dominance between a pair of non-negative cuts 
is possible, which is based on intersection points. 
\begin{lemma}\label{dominance_lemma}
Let 
$({\delta}^1)^T x \ge {\delta}^1_0$ and 
$({\delta}^2)^T x \ge {\delta}^2_0$ be 
non-negative cuts for $P$ 
satisfying 
$V^c=V^c_{(\delta^1,\delta_0^1)}=V^c_{(\delta^2,\delta_0^2)}$. 
Then $({\delta}^1)^T x \ge {\delta}^1_0$ dominates 
$({\delta}^2)^T x \ge {\delta}^2_0$ on $P$ if and only if
\begin{enumerate}
\item[(i)] The inequality 
$\frac{1}{\alpha_{i,j}'(\delta^1,\delta_0^1)} \leq$
$\frac{1}{\alpha_{i,j}'(\delta^2,\delta_0^2)}$ holds for $j \in E$ and 
$i \in V^c$.\\
(The halfline $\{ v^i + \alpha r^j : \alpha \geq 0 \}$ is intersected later by 
$({\delta}^1)^T x \ge {\delta}^1_0$ than $({\delta}^2)^T x \ge
{\delta}^2_0$) 
\item[(ii)] The inequality $\frac{1}{\beta_{i,k}'(\delta^1,\delta_0^1)} \leq$
$\frac{1}{\beta_{i,k}'(\delta^2,\delta_0^2)}$ holds for $k \in V \setminus V^c$ and 
$i \in V^c$.\\
\mbox{(The halfline $\{ v^i + \beta (v^k-v^i) : \beta \geq 0 \}$ is intersected later by 
$({\delta}^1)^T x \ge {\delta}^1_0$ than $({\delta}^2)^T x \ge
{\delta}^2_0$)} 
\end{enumerate}
\end{lemma}
\begin{proof}
Define 
$Q^1:=\{ x \in P : ({\delta}^1)^T x \ge {\delta}^1_0 \}$ and 
$Q^2:=\{ x \in P : ({\delta}^2)^T x \ge {\delta}^2_0 \}$. 
First suppose $({\delta}^1)^T x \ge {\delta}^1_0$ 
dominates $({\delta}^2)^T x \ge {\delta}^2_0$ on $P$, {\it i.e.}, 
suppose 
$Q^1 \subseteq Q^2$. We will verify 
that (i) and (ii) are satisfied. First let $i \in V^c$ 
and $j \in E$ be arbitrary. If $\alpha'_{i,j}(\delta^1,\delta^1_0)=+\infty$, 
then trivially $0=\frac{1}{\alpha_{i,j}'(\delta^1,\delta_0^1)} \leq$ 
$\frac{1}{\alpha_{i,j}'(\delta^2,\delta_0^2)}$. If 
$\alpha'_{i,j}(\delta^1,\delta^1_0) < +\infty$, then the intersection 
point $\bar{y}:=v^i+\alpha'_{i,j}(\delta^1,\delta^1_0) r^j$ 
satisfies $({\delta}^1)^T \bar{y} = {\delta}^1_0$, and therefore 
$\bar{y} \in Q^1 \subseteq Q^2$. Hence we have 
$({\delta}^2)^T \bar{y} =$
$({\delta}^2)^T v^i + \alpha'_{i,j}(\delta^1,\delta^1_0) ({\delta}^2)^T r^j \geq {\delta}^2_0$, which implies 
$\frac{1}{\alpha_{i,j}'(\delta^1,\delta_0^1)} \leq$ 
$\frac{1}{\alpha_{i,j}'(\delta^2,\delta_0^2)}$.

Now let $i \in V^c$ and $k \in V \setminus V^c$ be arbitrary. 
The 
point $\bar{z}:=v^i+\beta'_{i,k}(\delta^1,\delta^1_0) (v^k-v^i)$ 
satisfies $({\delta}^1)^T \bar{z} = {\delta}^1_0$. Hence 
$\bar{z} \in Q^1 \subseteq Q^2$, and therefore 
$({\delta}^2)^T \bar{z} =$
$({\delta}^2)^T v^i + \beta'_{i,k}(\delta^1,\delta^1_0) ({\delta}^2)^T (v^k-v^i) \geq {\delta}^2_0$, which implies 
$\frac{1}{\beta_{i,k}'(\delta^1,\delta_0^1)} \leq$ 
$\frac{1}{\beta_{i,k}'(\delta^2,\delta_0^2)}$.

Conversely suppose (i) and (ii) are satisfied. Since 
$V^c=V^c_{(\delta^1,\delta_0^1)}=V^c_{(\delta^2,\delta_0^2)}$, 
every vertex $v^k$ of $P$ with $k \in V \setminus V^c$ is a vertex 
of both $Q^1$ and $Q^2$. Furthermore, (i) ensures that every vertex 
of $Q^1$ of the form $v^i+\alpha_{i,j}'(\delta^1,\delta_0^1) r^j$ 
belongs to $Q^2$, where $i \in V^c$, $j \in E$ and 
$\alpha_{i,j}'(\delta^1,\delta_0^1) < +\infty$. Finally, (ii) 
ensures that every vertex 
of $Q^1$ of the form $v^i+\beta_{i,k}'(\delta^1,\delta_0^1) (v^k-v^i)$ 
belongs to $Q^2$, where $i \in V^c$ and $k \in V \setminus V^c$. We have 
therefore shown that every vertex of $Q^1$ belongs to $Q^2$. 
Since 
the sets $Q^1$ and 
$Q^2$ have the same 
extreme rays $\{ r^j {\}}_{j \in E}$, we therefore have 
$Q^1 \subseteq Q^2$.
\end{proof}

Let $V^c \subseteq V$ be arbitrary, and let 
$\{ ({\delta}^1)^T x \ge {\delta}^1_0 {\}}_{l=1}^m$ 
a finite set of non-negative cuts. We assume 
$V^c(\delta^l,\delta^l_0)=V^c$ for all 
$l \in \{ 1,2,\ldots,m \}$. 
We now derive a dominance result for the 
polyhedron $Q(V^c)$ 
$$Q(V^c):=\{ x \in P : ({\delta}^l)^T x \ge {\delta}^l_0 
\textrm{ for } l=1,2,\ldots,m \}.$$


\begin{lemma} \label{lem3} 
(This lemma is a generalization of Lemma 3 in \cite{andcorli})\\
Assume $Q(V^c) \ne \emptyset$. 
Let $\delta^T x \ge \delta_0$ be a non-negative cut for $P$ 
satisfying $V^c(\delta,\delta_0)=V^c$. 
Then $\delta^T x \ge \delta_0$ is valid for $Q(V^c)$ iff 
there exists a non-negative 
cut $(\delta')^T x \ge \delta_0'$ for $P$ 
that satisfies
\begin{enumerate}
\item[(i)] $(\delta')^T x \ge \delta_0'$ is a 
convex combination of the inequalities 
$(\delta^l)^T x \ge \delta^l_0$ for $l=1,2,\ldots,m$, 
\item[(ii)] $(\delta')^T x \ge \delta_0'$ 
dominates $\delta^T x \ge \delta_0$ on $P$.
\end{enumerate}
\end{lemma}
\begin{proof}
Consider the linear program (LP) 
given by $\min\{ {\delta}^T x : x \in Q(V^c) \}$. 
The assumption $Q(V^c) \ne \emptyset$ and 
the validity of $\delta^T x \ge \delta_0$ 
for $Q(V^c)$ implies that 
(LP) is feasible and bounded. We can formulate (LP) as follows.
\begin{alignat}{3}
\min & \,\,\, \delta^T x & \notag\\
x & = \sum_{i \in \textrm{V}^c} \lambda_i v^i + \sum_{i \in \textrm{V}^c} \sum_{k \in V \setminus \textrm{V}^c} \epsilon_k^i (v^k - v^i) + \sum_{i
  \in \textrm{V}^c} \sum_{j \in E} \mu_j^i r^j, & \qquad(u)\notag\\
(\delta^l)^T x & \geq \delta^l_0\textrm{ for all }l \in
\{1,2,\ldots,m\}, & \qquad(w_l)\notag\\
\sum_{k \in V \setminus \textrm{V}^c} \epsilon_k^i & \leq
\lambda_i\textrm{ for all }i \in \textrm{V}^c, & \qquad(z_i)\notag\\
\sum_{i \in \textrm{V}^c} \lambda_i & = 1, & \qquad(u_0)\notag\\
\epsilon^i, \mu^i, \lambda & \geq 0\textrm{ for all }i \in \textrm{V}^c.\notag
\end{alignat}

From the dual of (LP), we obtain 
$\bar{u} \in {\mathbb R}^n$, $\bar{w} \in {\mathbb R}^m$, 
$\bar{z} \in \R^{|\textrm{V}^c|}$ and $\bar{u}_0 \in \R$ 
that satisfy

\begin{enumerate}
\item[(i)] $\bar{u}_0 + $
$\sum_{l=1}^m \bar{w}_l {\delta}_0^l \ge \delta_0$, 
\item[(ii)] $-\bar{u} = \delta$, 
\item[(iii)] $\bar{u}^T v^i+\sum_{l=1}^m \bar{w}_l (\delta^l)^T
  v^i+\bar{z}_i+\bar{u}_0 \leq 0$ for all $i \in \textrm{V}^c$.
\item[(iv)] $\bar{u}^T (v^k-v^i)+\sum_{l=1}^m \bar{w}_l (\delta^l)^T
  (v^k-v^i)-\bar{z}_i \leq 0$ for all $i \in \textrm{V}^c$ and $k \in
  V \setminus \textrm{V}^c$.
\item[(v)] $\bar{u}^T r^j+\sum_{l=1}^m \bar{w}_l (\delta^l)^T r^j \leq
  0$ for all $j \in E$.
\item[(vi)] $\bar{w} \ge 0$ and $\bar{z} \geq 0$. 
\end{enumerate}

Let $\bar{\delta}:=\sum_{l=1}^m \bar{w}_l \delta^l$ and 
$\bar{\delta}_0:=\sum_{l=1}^m {\delta}_0^l \bar{w}_l$. Since 
$\bar{\delta}^T x \ge \bar{\delta}_0$ is a non-negative 
combination of the inequalities 
$\{ ({\delta}^l)^T x \ge {\delta}^l_0 {\}}_{l=1}^m$, 
we have that $\bar{\delta}^T x \ge \bar{\delta}_0$ is valid for 
$Q(V^c)$. 
Furthermore, the inequality $\bar{\delta}^T x \ge \bar{\delta}_0$ is 
a non-negative combination of non-negative 
cuts for $P$, and therefore $\bar{\delta}^T x \geq \bar{\delta}_0$ is also 
a non-negative cut for $P$. Finally, since 
$V^c(\delta^l,\delta^l_0)=V^c$ for all 
$l \in \{ 1,2,\ldots,m \}$, we have 
$V^c(\bar{\delta},\bar{\delta}_0)=V^c$. 
We will show that $\bar{\delta}^T x \ge \bar{\delta}_0$ dominates 
${\delta}^T x \ge {\delta}_0$ on $P$. The 
system (i)-(vi) implies the following inequalities.
\begin{enumerate}
\item[(a)] $\bar{u}_0 + \bar{\delta}_0 \geq \delta_0$.
\item[(b)] $-\delta^T v^i+\bar{\delta}^T v^i+\bar{z}_i + \bar{u}_0 \leq
  0$ for all $i \in \textrm{V}^c$.
\item[(c)] $-\delta^T (v^k-v^i)+\bar{\delta}^T (v^k-v^i)-\bar{z}_i
  \leq 0$ for all $i \in \textrm{V}^c$ and $k \in V \setminus \textrm{V}^c$.
\item[(d)] $-\delta^T r^j+\bar{\delta}^T r^j \leq 0$ for all $j \in E$.
\item[(e)] $\bar{w} \ge 0$ and $\bar{z} \geq 0$.
\end{enumerate}

We first show 
$\frac{1}{\alpha'_{i,j}(\bar{\delta},\bar{\delta}_0)} \leq$
$\frac{1}{\alpha'_{i,j}(\delta,\delta_0)}$ for all 
$i \in V^c$ and $j \in E$. 
Therefore let $\bar{i} \in V^c$ and 
$\bar{j} \in E$. If 
$\alpha'_{\bar{i},\bar{j}}(\delta,\delta_0)=+\infty$, then 
$\delta^T r^{\bar{j}}=0$, which by (d) implies 
that also $\bar{\delta}^T r^{\bar{j}}=0$, and therefore 
$0 = \frac{1}{\alpha'_{\bar{i},\bar{j}}(\bar{\delta},\bar{\delta}_0)} =$
$\frac{1}{\alpha'_{\bar{i},\bar{j}}(\delta,\delta_0)}$. Furthermore, 
if $\bar{\delta}^T r^{\bar{j}}=0$, then trivially 
$0=\frac{1}{\alpha'_{\bar{i},\bar{j}}(\bar{\delta},\bar{\delta}_0)} \leq$
$\frac{1}{\alpha'_{\bar{i},\bar{j}}(\delta,\delta_0)}$. We can 
therefore 
assume $\alpha'_{\bar{i},\bar{j}}(\delta,\delta_0)<+\infty$ and 
$\bar{\delta}^T r^{\bar{j}} > 0$. 
Multiplying the inequality 
of (d) corresponding to $\bar{j}$ 
with $\alpha'_{\bar{i},\bar{j}}(\delta,\delta_0)$ 
and adding the result to the 
inequality of (b) corresponding to $\bar{i}$ gives 
$-\delta^T (v^{\bar{i}}+\alpha'_{\bar{i},\bar{j}}(\delta,\delta_0) r^{\bar{j}})+$
$\bar{\delta}^T (v^{\bar{i}}+\alpha'_{\bar{i},\bar{j}}(\delta,\delta_0) r^{\bar{j}}) \leq$
$-\bar{u}_0-\bar{z}_{\bar{i}} \leq \bar{\delta}_0-\delta_0$. Since we have 
$\delta^T (v^{\bar{i}}+\alpha'_{\bar{i},\bar{j}}(\delta,\delta_0) r^{\bar{j}})=\delta_0$, 
this implies 
$\bar{\delta}^T (v^{\bar{i}}+\alpha'_{\bar{i},\bar{j}}(\delta,\delta_0) r^{\bar{j}}) \leq$
$\bar{\delta}_0$. Now, $\alpha'_{\bar{i},\bar{j}}(\bar{\delta},\bar{\delta}_0)$ is defined 
as the smallest value of $\alpha$ such that 
$\bar{\delta}^T (v^{\bar{i}}+\alpha r^{\bar{j}})=$
$\bar{\delta}_0$. Since 
$\bar{\delta}^T (v^{\bar{i}}+\alpha'_{\bar{i},\bar{j}}(\delta,\delta_0) r^{\bar{j}}) \leq$
$\bar{\delta}_0$, this means we must have 
$\alpha'_{\bar{i},\bar{j}}(\bar{\delta},\bar{\delta}_0) \geq$
$\alpha'_{\bar{i},\bar{j}}(\delta,\delta_0)$, and therefore 
$\frac{1}{\alpha'_{\bar{i},\bar{j}}(\bar{\delta},\bar{\delta}_0)} \leq$
$\frac{1}{\alpha'_{\bar{i},\bar{j}}(\delta,\delta_0)}$. Hence 
condition (i) of Lemma \ref{dominance_lemma} is satisfied.

We now show $\frac{1}{\beta'_{i,k}(\bar{\delta},\bar{\delta}_0)} \leq$
$\frac{1}{\beta'_{i,k}(\delta,\delta_0)}$ for all 
$i \in V^c$ and $k \in V \setminus V^c$.
Therefore let $\bar{i} \in V^c$ and 
$\bar{k} \in V \setminus V^c$. Multiplying the inequality 
of (c) corresponding to $(\bar{i},\bar{k})$ 
with $\beta'_{\bar{i},\bar{k}}(\delta,\delta_0)$ 
and adding the result to the 
inequality of (b) corresponding to $\bar{i}$ gives 
$-\delta^T (v^{\bar{i}}+\beta'_{\bar{i},\bar{k}}(\delta,\delta_0) (v^{\bar{k}}-v^{\bar{i}}))+$
$\bar{\delta}^T (v^{\bar{i}}+\beta'_{\bar{i},\bar{k}}(\delta,\delta_0)
(v^{\bar{k}}-v^{\bar{i}}) ) \leq$
$-\bar{u}_0-\bar{z}_{\bar{i}} \leq \bar{\delta}_0-\delta_0$.
Since 
$\delta^T (v^{\bar{i}}+\beta'_{\bar{i},\bar{k}}(\delta,\delta_0) (v^{\bar{k}}-v^{\bar{i}}))=\delta_0$, 
this implies 
$\bar{\delta}^T (v^{\bar{i}}+\beta'_{\bar{i},\bar{k}}(\delta,\delta_0) (v^{\bar{k}}-v^{\bar{i}})) \leq$
$\bar{\delta}_0$. We have that 
$\beta'_{\bar{i},\bar{k}}(\bar{\delta},\bar{\delta}_0)$ is defined 
as the smallest value of $\beta$ s.t. 
$\bar{\delta}^T (v^{\bar{i}}+\beta (v^{\bar{k}}-v^{\bar{i}}) )=$
$\bar{\delta}_0$, and since 
$\bar{\delta}^T (v^{\bar{i}}+\alpha'_{\bar{i},\bar{j}}(\delta,\delta_0) r^{\bar{j}}) \leq$
$\bar{\delta}_0$, this implies 
$\beta'_{\bar{i},\bar{k}}(\bar{\delta},\bar{\delta}_0) \geq$
$\beta'_{\bar{i},\bar{k}}(\delta,\delta_0)$. It follows that 
$\frac{1}{\beta'_{\bar{i},\bar{k}}(\bar{\delta},\bar{\delta}_0)} \leq$
$\frac{1}{\beta'_{\bar{i},\bar{k}}(\delta,\delta_0)}$. 
Hence condition (ii) of Lemma \ref{dominance_lemma} is also satisfied, 
and therefore $\bar{\delta}^T x \geq \bar{\delta}_0$ 
dominates $\delta^T x \geq \delta_0$ on $P$.

To finish the proof, we will argue that we can 
choose $\bar{\delta}^T x \geq \bar{\delta}_0$ to 
be a convex combination of the inequalities 
$\{ ({\delta}^l)^T x \ge {\delta}^l_0 {\}}_{l=1}^m$. 
Observe that, if $\sum_{l=1}^m \bar{w}_l \neq 0$, 
then the inequality $(\delta')^T x \geq \delta'_0$ 
defined by $(\delta',\delta'_0):=$
$\frac{1}{\sum_{l=1}^m \bar{w}_l}(\bar{\delta},\bar{\delta}_0)$ 
is a convex combination of the inequalities 
$\{ ({\delta}^l)^T x \ge {\delta}^l_0 {\}}_{l=1}^m$ and 
$(\delta')^T x \geq \delta'_0$ is 
equivalent to $\bar{\delta}^T x \geq \bar{\delta}_0$ on 
$P$. We therefore only have to show 
$\sum_{l=1}^m \bar{w}_l \neq 0$. 
If $\sum_{l=1}^m \bar{w}_l = 0$, 
then (i)-(iii) give $\bar{u}_0 \geq \delta_0$ and 
$-\delta^T v^i + \bar{z}_i+\bar{u}_0 \leq 0$ for all 
$i \in V^c$, which implies $\delta^T v^i \geq \delta_0$ for 
all $i \in V^c$. Furthermore, (iv) reads 
$-\delta^T (v^k - v^i)-\bar{z}_i \leq 0$ for all 
$i \in V^c$ and $k \in V \setminus V^c$. Given 
$\bar{i} \in V^c$ and $\bar{k} \in V \setminus V^c$, adding the 
inequality 
$-\delta^T (v^{\bar{k}} - v^{\bar{i}})-\bar{z}_{\bar{i}} \leq 0$ 
of (iv) to 
the inequality of (iii) corresponding to $\bar{i}$ gives 
$-\delta^T v^{\bar{k}} \leq -\bar{u}_0 \leq -\delta_0$. Hence 
$\delta^T x \geq \delta_0$ is satisfied by all vertices of 
$P$, which contradicts that ${\delta}^T x \ge {\delta}_0$ is 
a cut for $P$. Hence $\sum_{l=1}^m \bar{w}_l \neq 0$.
\end{proof}

\subsection{A sufficient condition for polyhedrality}

We now consider the addition of an \emph{infinite} family of non-negative 
cuts to the polyhedron $P$. Specifically, consider
the convex set 
$$X:=\{ x \in P : ({\delta}^l)^T x \ge {\delta}^l_0 
\textrm{ for } l \in I \},$$
where $I$ is now allowed to be an \emph{infinite} index set. The goal in 
this section is to provide a sufficient condition 
for $X$ to be a polyhedron. For this purpose, we can assume 
$V^c(\delta^l,\delta^l_0)=V^c$ for all $l \in I$, {\it i.e.}, we 
can assume all cuts cut off the same vertices. Indeed, 
if the cuts $l \in I$ do not cut off the same set 
of vertices, then define the set 
$$I^c(S):=\{ l \in I : V^c(\delta^l,\delta^l_0)=S \}$$
for every $S \subseteq V$, and let 
$\mathcal{S}:=\{ S \subseteq V : I^c(S) \neq \emptyset \}$. 
We can then write 
$$X=\cap_{S \in \mathcal{S}} \{ x \in P : ({\delta}^l)^T x \ge {\delta}^l_0 
\textrm{ for } l \in S \}$$
Since $\mathcal{S}$ is finite, we have that $X$ is a 
polyhedron if and only if $X$ is a polyhedron under the 
assumption that $V^c(\delta^l,\delta^l_0)=V^c$ for all $l \in I$. 

For simplicity let 
$\alpha'_{i,j,l}:=\alpha'_{i,j}(\delta^l,\delta^l_0)$ 
for all $(i,j,l) \in V^c \times E \times I$, and 
$\beta'_{i,k,l}:=\beta'_{i,k}(\delta^l,\delta^l_0)$ 
for all $(i,k,l) \in V^c \times (V \setminus V^c) \times I$. 
Furthermore, for any $l \in I$, let $\alpha'_{.l}$ denote the vector in 
$\R^{|V^c| \times |E|}$ whose coordinates are $\alpha'_{i,j,l}$ 
for $(i,j) \in V^c \times E$, and let $\beta'_{.l}$ denote the vector in 
$\R^{|V^c| \times |V \setminus V^c|}$ whose coordinates are 
 $\beta'_{i,k,l}$ for $(i,k) \in V^c \times (V \setminus V^c)$. 

We will show 
that $X$ is a polyhedron when the following assumption holds.
\begin{assumption}\label{bas_ass}
Let $\alpha^* > 0$ and $\beta^* \in ]0,1]$ be arbitrary.
\begin{enumerate}
\item[(1)] For all $(i,j) \in V^c \times E$, 
the set $\textrm{IP}^e_{(i,j)}(\alpha^*):=\{ \alpha'_{i,j,l} \geq \alpha^* : l \in I \}$ is finite\\
(There is only a finite number of intersection points 
between the inequalities $(\delta^l)^T x \geq \delta^l_0$ for $l \in I$ and 
the halfline $\{v^i + \alpha r^j : \alpha \geq \alpha^* \}$).
\item[(2)] For all $(i,k) \in V^c \times V \setminus V^c$, the set 
$\textrm{IP}^v_{(i,k)}(\beta^*):=\{ \beta'_{i,k,l} \geq \beta^* : l \in I \}$ is finite\\
(There is only a finite number of intersection points 
between the inequalities $(\delta^l)^T x \geq \delta^l_0$ for $l \in I$ and 
the halfline $\{v^i + \beta (v^k-v^i) : \beta \geq \beta^* \}$).
\end{enumerate}
\end{assumption}

The main theorem is the following.
\begin{theorem}\label{pol_thm}
Suppose $\{ (\delta^l)^T x \geq \delta^l_0 {\}}_{l \in I}$ is a family 
of non-negative cuts for $P$ that satisfies Assumption 
\ref{bas_ass}, and suppose 
$V^c=V^c(\delta^l,\delta^l_0)$ for all $l \in I$. Then 
the set $X$ is a polyhedron.
\end{theorem}

We will prove Theorem \ref{pol_thm} by induction on $|V \setminus V^c|+|E|$.

\subsubsection{The basic step of the induction}

We first consider 
the case when $|V \setminus V^c|+|E|=1$. The proof of Theorem \ref{pol_thm} in this 
special case is by induction on $|V^c|$, and this proof is essentially the same 
for both the case when $|V \setminus V^c|=1$, and the case when $|E|=1$. We therefore assume 
$E=\{ 1 \}$ and $|V \setminus V^c|=0$ in the remainder of this subsection. 
We first consider the case when $|V^c|=1$.

\begin{lemma}\label{basic_basic_step}
(Lemma 7 in \cite{andcorli}). 
Suppose $|V^c|=1$, $|V \setminus V^c|=0$ and $E=\{ 1 \}$. Then there 
exists $\bar{l} \in I$ such that 
$X=\{ x \in P : ({\delta}^{\bar{l}})^T x \ge {\delta}^{\bar{l}}_0 \}$. 
\end{lemma}
\begin{proof}
For simplicity assume $V^c=\{ 1 \}$. 
We have 
\begin{alignat}{2}
P & =\{ x \in {\mathbb R}^n : x=v^1+\mu_1 r^1\textrm{ and }\mu_1 \ge 0 \},\textrm{ and}\notag\\
\{ x \in P : ({\delta}^l)^T x \ge {\delta}^l_0 \} & = \{ x \in \R^n : x=v^1+\mu_1 r^1,\mu_1 \ge 0\textrm{ and }\frac{\mu_1}{\alpha'_{1,1}(\delta^l,\delta_0^l)} \geq 1 \}\notag 
\end{alignat}
for all $l \in I$. 
Defining 
$\alpha^*_{1,1} := \sup \{ \alpha'_{1,1,l} : l \in I \}$ then gives
\begin{alignat}{2}
X & = \{ x \in \R^n : x=v^1+\mu_1 r^1,\mu_1 \ge 0\textrm{ and }\frac{\mu_1}{\alpha^*_{1,1}} \geq 1 \}.\notag 
\end{alignat}
Hence the only issue that needs to be verified is that 
the value $\alpha^*_{1,1}$ is attained for some $l \in I$. 
If there exists $l \in I$ satisfying $(\delta^l)^T r^1 = 0$, 
we are done, so we may assume $\alpha'_{1,1,l} < +\infty$ 
for all $l \in I$. 
Choosing $l' \in I$ arbitrarily, Assumption \ref{bas_ass}.(i) shows the 
set $\textrm{IP}^e_{(1,1)}(\alpha_{1,1,l'})$ is finite. 
Therefore the supremum is achieved.
\end{proof}

The induction hypothesis is as follows. For every $i \in V^c$, 
define 
\begin{alignat}{2}
P^i & :=\conv(\{ v^{i'} {\}}_{i' \in V^c \setminus \{ i \}})+\cone(\{ r^1 \}),\textrm{ and}\notag\\
X^i & :=\{ x \in P^i : ({\delta}^l)^T x \ge {\delta}^l_0\textrm{ for all } l \in I \}.\notag
\end{alignat}
The induction hypothesis is that $X^i$ is a polyhedron 
for all $i \in V^c$. Hence, for every $i \in V^c$, we 
can choose a \emph{finite} subset $I^i \subseteq I$ such that 
$X^i=\{ x \in P^i : ({\delta}^l)^T x \ge {\delta}^l_0\textrm{ for all } l \in I^i \}$. 

Let $\bar{I}:=\cup_{i \in V^c} I^i$ denote the set of 
\emph{all} inequalities needed to describe the sets $X^i$ 
for $i \in V^c$. Also let 
$\overline{X}:=\{ x \in P : ({\delta}^l)^T x \ge {\delta}^l_0 \textrm{ for } l \in \bar{I} \}$ be 
the approximation of $X$ obtained from the finite set of inequalities 
indexed by $\bar{I}$, and define the numbers 
\begin{alignat}{2}
\alpha^*_i & := \min \{ \alpha'_{i,1,l} : l \in
\bar{I} \} > 0\textrm{ for all }i \in V^c.\notag
\end{alignat}
The number $\alpha^*_i$ gives the intersection point 
$v^i+\alpha'_{i,1,l} r^1$ 
which is closest 
to $v^i$ over all inequalities $l \in \bar{I}$. Based 
on the induction hypothesis, we now show that $X$ is a polyhedron 
when $|V \setminus V^c|=0$ and $E= \{1 \}$.
\begin{lemma}
(Lemma 8 in \cite{andcorli}). 
If $|V \setminus V^c|=0$ and $E= \{ 1 \}$, then $X$ is a polyhedron. 
\end{lemma}
\begin{proof}
Consider an inequality $l' \in I \setminus \bar{I}$. 
We will show that 
$({\delta}^{l'})^T x \ge {\delta}^{l'}_0$ is valid for 
$\overline{X}$ if there exists $i' \in V^c$ such that 
$\alpha_{i',1,l'} \leq \alpha^*_{i'}$. 
This implies that it is sufficient to consider inequalities 
$l \in I \setminus \bar{I}$ that satisfy 
$\alpha_{i,1,l} > \alpha^*_i$ for 
all $i \in V^c$ to obtain $X$ from $\overline{X}$. 
Since the sets $\textrm{IP}^e_{(i,1)}(\alpha^*_i)$ 
for $i \in V^c$ are finite (Assumption \ref{bas_ass}.(i)), 
and since two inequalities 
with exactly the same intersection points are equivalent 
(Lemma \ref{dominance_lemma}), 
this shows that only a finite number of inequalities 
from $I \setminus \bar{I}$ are needed to obtain 
$X$ from $\overline{X}$. 

Therefore suppose $l' \in I \setminus \bar{I}$ 
and $i' \in V^c$ satisfies 
$\alpha_{i',1,l'} \leq \alpha^*_{i'}$. For simplicity 
let $(\delta',\delta'_0):=({\delta}^{l'},{\delta}^{l'}_0)$. 
Since $(\delta')^T x \ge \delta'_0$ is a 
non-negative cut for $P^{i'}$ that is valid for $X^{i'}$ (the induction
hypothesis), 
Lemma \ref{lem3} shows there exists an inequality $\bar{\delta}^T x \ge $
$\bar{\delta}_0$ that dominates $(\delta')^T x \ge \delta'_0$ 
on $P^{i'}$, and that $\bar{\delta}^T x \ge \bar{\delta}_0$ 
can be chosen as a 
convex combination of the inequalities 
$({\delta}^l)^T x \ge {\delta}^l_0$ for 
$l \in \bar{I}$. We therefore have 
\begin{alignat}{2}
\bar{\delta} & = \sum_{l \in \bar{I}} \lambda_{l} \delta^l,\textrm{ and}\notag\\
\bar{\delta}_0 & = \sum_{l \in \bar{I}} \lambda_{l} \delta^l_0,\textrm{ where}\notag\\
\sum_{l \in \bar{I}} \lambda_{l} =
1 & \textrm{ and }\lambda_{l} \ge 0 \textrm{ for all }l \in
\bar{I}.\notag
\end{alignat}

We will show that $\bar{\delta}^T x \ge $
$\bar{\delta}_0$ dominates 
$(\delta')^T x \ge \delta'_0$ on $P$ by 
verifying that condition (i) of Lemma \ref{dominance_lemma} is satisfied. 
We know $\bar{\delta}^T x \ge \bar{\delta}_0$ dominates 
$(\delta')^T x \ge \delta'_0$ on $P^{i'}$. 
Lemma \ref{dominance_lemma} therefore gives 
\begin{alignat}{2}
\frac{1}{\alpha'_{i,1}(\bar{\delta},\bar{\delta}_0)} & \leq
\frac{1}{\alpha'_{i,1}({\delta'},\delta'_0)} \textrm{ for all }i
\in V^c \setminus \{ i' \}.\notag
\end{alignat}
To finish the proof, we will show
$\frac{1}{\alpha'_{i',1}(\bar{\delta},\bar{\delta}_0)} \leq
\frac{1}{\alpha'_{i',1}({\delta'},\delta'_0)}$.
The definition of $\alpha_{i'}^*$ gives 
$$\alpha^*_{i'} \leq \alpha'_{i',1,l} = \frac{\delta^l_0 - (\delta^l)^T v^{i'}}{(\delta^l)^T
  r^1}$$
for all $l \in \bar{I}$. Since 
$\bar{\delta}_0 - \bar{\delta}^T v^{i'} = $
$\sum_{l \in \bar{I}} \lambda_{l} $
$(\delta^l_0 - (\delta^l)^T v^{i'})$ and
$\bar{\delta}^T r^1 = $
$\sum_{l \in \bar{I}} \lambda_{l}$
$(\delta^l)^T r^1$, we obtain 
$\bar{\delta}_0 - \bar{\delta}^T v^{i'} \ge \alpha^*_{i'}
\bar{\delta}^T r^1$, and therefore 
$\alpha_{i'}^* \leq \alpha'_{i',1}(\bar{\delta},\bar{\delta}_0)$. 
The choice of 
$\alpha_{i'}^*$ and 
$(\delta')^T x \geq \delta'_0$ gives 
$\alpha_{i',1}(\delta',\delta'_0) \leq$ 
$\alpha_{i'}^*$. Hence 
$\alpha_{i',1}'(\delta',\delta'_0) \leq $
$\alpha_{i',1}'(\bar{\delta},\bar{\delta}_0)$, which implies 
$1/{\alpha}_{i',1}'(\bar{\delta},\bar{\delta}_0) \leq $
$1/{\alpha}_{i',1}'(\delta',\delta'_0)$.
\end{proof}

\subsubsection{The induction hypothesis}

We now present the induction hypothesis. 
Given a vertex $v^{k}$ of $P$ with $k \in V \setminus V^c$, consider the 
polyhedron obtained from $P$ by deleting $v^{k}$
$$P^{k}:=\conv(\{ v^i {\}}_{i \in V \setminus \{ k \}})+\cone(\{ r^j {\}}_{j \in E}),$$
and given an extreme ray $r^{j}$ of $P$ with $j \in E$, consider the 
polyhedron obtained from $P$ by deleting $r^j$
$$P^{j}:=\conv(\{ v^i {\}}_{i \in V})+\cone(\{ r^{j'}
{\}}_{j' \in E \setminus \{ j \}}).$$
From the inequalities $(\delta^l)^T x \geq \delta^l_0$ for $l \in I$, 
and the polyhedra $P^k$ and $P^j$, we can define 
the following subsets of $X$. 
$$X^{k}:=\{ x \in P^k : ({\delta}^l)^T x \ge {\delta}^l_0 
\textrm{ for } l \in I \},\textrm{ and}$$
$$X^{j}:=\{ x \in P^j : ({\delta}^l)^T x \ge {\delta}^l_0 
\textrm{ for } l \in I \},$$
The induction hypothesis is that the sets $X^k$ and $X^j$ are 
polyhedra for all $k \in V \setminus V^c$ and $j \in E$. This implies that 
for every $k \in V \setminus V^c$, there exists a finite set $I^k \subseteq I$ such that 
$$X^k=\{ x \in P^k : ({\delta}^l)^T x \ge {\delta}^l_0 \textrm{ for }
l \in I^k \},$$
and for every $j \in E$, there exists a finite set $I^j \subseteq I$ such that 
$$X^j=\{ x \in P^j : ({\delta}^l)^T x \ge {\delta}^l_0 \textrm{ for }
l \in I^j \}.$$
Define $\bar{I}:=(\cup_{k \in V \setminus V^c} I^k) \cup (\cup_{j
  \in E} I^j)$ to be the set of \emph{all} inequalities involved above.
The set $\bar{I}$ gives the following 
approximation $\overline{X}$ of $X$.
\begin{alignat}{2}
\overline{X} & :=\{ x \in P : (\delta^l)^T x \geq \delta^l_0 \textrm{ for all }l \in \bar{I} \}\notag
\end{alignat}

\subsubsection{The inductive proof}

We now use the induction hypothesis to prove that $X$ is a polyhedron. 
The idea of the proof is based on counting 
the number $|\textrm{SIP}(I')|$ 
of intersection points that are 
shared by \emph{all} cuts in a family 
$I' \subseteq I$ of cutting planes. This number is given by 
$|\textrm{SIP}(I')|=|\textrm{SIP}^e(I')|+|\textrm{SIP}^v(I')|$, where 
the sets $\textrm{SIP}^e(I')$ 
and $\textrm{SIP}^v(I')$ are defined by
\begin{alignat}{2}
\textrm{SIP}^e(I') & :=\{ (i,j) \in V^c \times E : \alpha'_{i,j,l_1} =
\alpha'_{i,j,l_2}\textrm{ for all }l_1,l_2 \in I' \},\textrm{ and}\notag\\
\textrm{SIP}^v(I') & :=\{ (i,k) \in V^c \times (V \setminus V^c) : \beta'_{i,k,l_1}=
\beta'_{i,k,l_2}\textrm{ for all }l_1,l_2 \in I' \}.\notag
\end{alignat}

Clearly we have $0 \leq |\textrm{SIP}(I')| \leq $
$|V^c \times E|+|V^c \times (V \setminus V^c)|$ for all $I' \subseteq
I$. 
Furthermore, if $|\textrm{SIP}(I')| =$
$|V^c \times E|+|V^c \times (V \setminus V^c)|$, then 
\emph{all} cuts 
indexed by $I'$ share \emph{all}
intersection points with the halflines 
$\{ v^i + \alpha r^j : \alpha \geq 0 \}$ and 
$\{ v^i + \beta (v^k-v^i) : \beta \geq 0 \}$ for 
$i \in V^c$, $j \in E$ and $k \in V \setminus V^c$. This 
then implies that all cuts 
indexed by $I'$ are equivalent on $P$ (Lemma \ref{dominance_lemma}). 
Therefore, if $|\textrm{SIP}(I')| =$
$|V^c \times E|+|V^c \times (V \setminus V^c)|$, 
then the set $X'$ given by 
$X':=\{ x \in P : (\delta^l)^T x \geq \delta^l_0\textrm{ for all }l \in I' \}$ 
is a polyhedron that can be described with exactly one cut 
from the family $I'$.

The main idea of our proof can now be presented. 
Clearly we can assume that the family $\bar{I}$ does 
not give a complete description 
of $X$ (otherwise there is nothing to prove). We will 
show the following lemma.
\begin{lemma}\label{main_goal}
Assume the sets $\{ X^k {\}}_{k \in V \setminus V^c}$ and 
$\{ X^j {\}}_{j \in E}$ are polyhedra. There exists a 
covering of $I$ into a finite number of subsets 
$\{ I^q {\}}_{q=1}^{\textrm{ns}}$ such that
$$\textrm{ for all }q \in \{ 1,2,\ldots,\textrm{ns} \},  \textrm{
either }I^q \subseteq \bar{I},\textrm{ or }|\textrm{SIP}(I^q)|>|\textrm{SIP}(I)|,$$ 
where $\textrm{ns}$ denotes the number of subsets in this 
covering.
\end{lemma}

The goal of the remainder of this section is to prove 
Lemma \ref{main_goal}. We first argue that Lemma 
\ref{main_goal} implies that $X$ is a polyhedron. 
The fact that $\{ I^q {\}}_{q=1}^{\textrm{ns}}$ 
is a covering of $I$ implies 
$$X = \cap_{q=1}^{\textrm{ns}} X(I^q),$$
where 
$X(I^q):=\{ x \in P :$ 
$(\delta^l)^T x \geq \delta^l_0 \textrm{ for all }l \in I^q \}$. 
Therefore $X$ is a polyhedron if $X(I^q)$ is a polyhedron for 
all $q \in \{ 1,2,\ldots,\textrm{ns} \}$. Since 
$|\textrm{SIP}(I^q)|$ is larger than 
$|\textrm{SIP}(I)|$ for all $q \in \{ 1,2,\ldots,\textrm{ns} \}$
satisfying $I^q \nsubseteq \bar{I}$, 
recursively applying Lemma \ref{main_goal} will create a tree of subcases, where 
the sets corresponding to the leaves of this tree must be polyhedra. 
It then follows that $X$ is a polyhedron. 

We now proceed to prove Lemma \ref{main_goal}. 
The covering of $I$ is based 
on the following positive numbers that measure how close the cuts 
$\{ (\delta^l)^T x \geq \delta^l_0 {\}}_{l \in \bar{I}}$ cut to a vertex 
$v^i$ of $P$.
\begin{alignat}{2}
\alpha^*_j & := \min \{ \alpha'_{i,j,l} : i \in V^c\textrm{ and }l \in
\bar{I} \}\textrm{ for }j \in E,\textrm{ and}\notag\\
\beta^*_k & := \min \{ \beta'_{i,k,l} : i \in V^c\textrm{ and }l \in
\bar{I} \} \textrm{ for } k \in V \setminus V^c.\notag
\end{alignat}

Given $\bar{j} \in E$, the number $\alpha^*_{\bar{j}}$ corresponds to a 
vertex $v^{\bar{i}}$ of $P$ 
and a cut 
$(\delta^{\bar{l}})^T x \geq \delta^{\bar{l}}_0$ for which the 
intersection point 
$v^{\bar{i}}+\alpha'_{\bar{i},\bar{j},\bar{l}} r^{\bar{j}}$ 
is as close to $v^{\bar{i}}$ as possible. Similarly, given 
$\bar{k} \in V \setminus V^c$, the number $\beta^*_{\bar{k}}$
corresponds to a 
vertex $v^{\bar{i}}$ of $P$ and a cut 
$(\delta^{\bar{l}})^T x \geq \delta^{\bar{l}}_0$ for which the 
intersection point 
$v^{\bar{i}}+$
$\beta'_{\bar{i},\bar{k},\bar{l}} (v^{\bar{k}}-v^{\bar{i}})$ 
is as close to $v^{\bar{i}}$ as possible.

The numbers $\{ \alpha^*_j {\}}_{j \in E}$ and 
$\{ \beta^*_k {\}}_{k \in V \setminus V^c}$ allow us to provide 
the following conditon that the 
cuts $(\delta^l)^T x \geq \delta^l_0$ for $l \in I \setminus \bar{I}$ 
must satisfy in order to cut off a region of 
$\overline{X}$. Clearly cuts that are valid for 
$\overline{X}$ can be removed from $I \setminus \bar{I}$, since 
they do not contribute 
anything further to the description of $X$ than the cuts 
indexed by $\bar{I}$.

\begin{lemma}\label{cat_lem}
(Lemma 8 in \cite{andcorli}). 
Assume the sets $\{ X^k {\}}_{k \in V \setminus V^c}$ and 
$\{ X^j {\}}_{j \in E}$ are polyhedra, and let $\bar{l} \in I \setminus \bar{I}$ 
be arbitrary. If either 
\begin{enumerate}
\item[(i)] There exists $\bar{j} \in E$ such that $\max\{ \alpha'_{i,\bar{j},\bar{l}} : i \in
  V^c \textrm{ and }(i,\bar{j}) \notin \textrm{SIP}^e(I) \} \leq \alpha^*_{\bar{j}}$\\
(The cut $(\delta^{\bar{l}})^T x \geq \delta^{\bar{l}}_0$ does not cut
off any point of the form $v^i+\alpha^*_{\bar{j}} r^{\bar{j}}$ which is not an
intersection point that is shared by all cuts in $I$), or 
\item[(ii)] There exists $\bar{k} \in V \setminus V^c$ such that $\max\{ \beta'_{i,\bar{k},\bar{l}} : i \in
  V^c\textrm{ and }(i,\bar{k}) \notin \textrm{SIP}^v(I) \} \leq \beta^*_{\bar{k}}$\\
(The cut $(\delta^{\bar{l}})^T x \geq \delta^{\bar{l}}_0$ does not cut
off any point of the form $v^i+\beta^*_{\bar{k}} (v^{\bar{k}}-v^i)$ which is not an
intersection point that is shared by all cuts in $I$),

\end{enumerate}
then the cut 
$({\delta}^{\bar{l}})^T x \ge {\delta}^{\bar{l}}_0$ is valid for $\overline{X}$.
\end{lemma}

The proof of Lemma \ref{cat_lem} will be given 
at the end of this section. We first argue that 
Lemma \ref{cat_lem} can be used to prove Lemma \ref{main_goal}, 
which thereby finishes the proof of Theorem \ref{pol_thm}.

Lemma \ref{cat_lem} shows we can partition the cuts 
$\{ (\delta^l)^T x \geq \delta^l_0 {\}}_{l \in I}$ into 
three categories.
\begin{enumerate}
\item[(1)] The cuts indexed by $\bar{I}$ that define $\bar{X}$.
\item[(2)] The cuts, indexed by some set 
$I^r \subseteq I \setminus \bar{I}$, that satisfy either 
Lemma \ref{cat_lem}.(i) or Lemma \ref{cat_lem}.(ii), and these 
cuts are valid for $\overline{X}$.
\item[(3)] The remainder of the cuts indexed by
$I \setminus (\bar{I} \cup I^r)$. Every cut 
$l \in I \setminus (\bar{I} \cup I^r)$ satisfies: 
\begin{enumerate}
\item[(i)] For all $j \in E$, the cut 
$(\delta^l)^T x \geq \delta^l_0$ cuts off some intersection point of
the form $v^i+\alpha^*_j r^{\bar{j}}$, which is not an
intersection point that is shared by all cuts in $I$.
\item[(ii)] For all $k \in V \setminus V^c$, the cut 
$(\delta^l)^T x \geq \delta^l_0$ cuts off some intersection point of
the form $v^i+\beta^*_k (v^k-v^i)$, which is not an
intersection point that is shared by all cuts in $I$.
\end{enumerate}
\end{enumerate}

Clearly we can assume $I^r=\emptyset$. 
Let $i \in V^c$, $j \in E$ and $k \in V \setminus V^c$ 
be arbitrary. Recall that the sets 
$\textrm{IP}^e_{(i,j)}(\alpha^*_j)$ and 
$\textrm{IP}^v_{(i,k)}(\beta^*_k)$ identify the 
intersection points between the hyperplanes 
$\{ (\delta^l)^T x = \delta^l_0 {\}}_{l \in I}$ and 
the halflines $\{ v^i+\alpha r^j : \alpha \geq \alpha^*_j \}$ and 
$\{ v^i+\beta (v^k-v^i) : \beta \geq \beta^*_k \}$ 
respectively. Hence 
we may write $\textrm{IP}^e_{(i,j)}(\alpha^*_j)$ and 
$\textrm{IP}^v_{(i,k)}(\beta^*_k)$ in the form
\begin{alignat}{2}
\textrm{IP}^e_{(i,j)}(\alpha^*_j) & = \{ \alpha_{i,j}^1,
\alpha_{i,j}^2,\ldots,\alpha_{i,j}^{n^e(i,j)} \}\textrm{ and}\notag\\
\textrm{IP}^v_{(i,k)}(\beta^*_k) & = \{ \beta_{i,k}^1,
\beta_{i,k}^2,\ldots,\beta_{i,k}^{n^v(i,k)} \},\notag
\end{alignat}
where the numbers $n^e(i,j):=|\textrm{IP}^e_{(i,j)}(\alpha^*_j)|$ and 
$n^v(i,k):=|\textrm{IP}^v_{(i,k)}(\beta^*_k)|$ denote the sizes 
of the two sets. 
For simplicity let 
$N^e_{(i,j)}:=\{1,2,\ldots,n^e(i,j) \}$ and 
$N^v_{(i,k)}:=\{1,2,\ldots,n^v(i,k) \}$ index the 
intersection points. 

All intersection points between a hyperplane 
$(\delta^l)^T x = \delta^l_0$ with $l \in I \setminus \bar{I}$ and a halfline either of 
the form $\{ v^i+\alpha r^j : \alpha \geq \alpha^*_j \}$ , 
or of the form 
$\{ v^i+\beta (v^k-v^i) : \beta \geq \beta^*_k \}$, can be 
identified with elements of the index sets
\begin{alignat}{2}
\textrm{AIP}^e & :=\{ (i,j,q) : i \in V^c, j \in E\textrm{ and }q \in
N^e_{(i,j)} \},\textrm{ and}\notag\\
\textrm{AIP}^v & :=\{ (i,k,q) : i \in V^c, k \in V \setminus
V^c\textrm{ and }q \in N^v_{(i,k)} \}.\notag
\end{alignat}

For a specific cut $(\delta^l)^T x \geq \delta^l_0$ with $l \in I$,
let the sets 
\begin{alignat}{2}
\textrm{IP}^e(l) & :=\{ (i,j,q) \in \textrm{AIP}^e : \alpha'_{i,j,l}=\alpha^q_{i,j} \}\notag\\
\textrm{IP}^v(l) & :=\{ (i,k,q) \in \textrm{AIP}^v : \beta'_{i,k,l}=\beta^q_{i,k} \}.\notag
\end{alignat}
index the intersection points between $(\delta^l)^T x = \delta^l_0$ 
and the halflines $\{ v^i+\alpha r^j : \alpha \geq \alpha^*_j \}$ and 
$\{ v^i+\beta (v^k-v^i) : \beta \geq \beta^*_k \}$ for $i \in V^c$, 
$j \in E$ and $k \in V \setminus V^c$.

Observe that, from the definitions of $\alpha^*_j$ and $\beta^*_k$ for 
$j \in E$ and $k \in V \setminus V^c$, we have $\textrm{IP}^e(l) \neq
\emptyset$ and $\textrm{IP}^v(l) \neq \emptyset$ for all $l \in
\bar{I}$. Furthermore, property (3) above and the assumption 
$I^r =\emptyset$ ensures that $\textrm{IP}^e(l) \neq \emptyset$ and 
$\textrm{IP}^v(l) \neq \emptyset$ for all 
$l \in I \setminus \bar{I}$. Hence we have 
$\textrm{IP}^e(l) \neq \emptyset$ and 
$\textrm{IP}^v(l) \neq \emptyset$ for all $l \in I$.

Given a pair 
$(S^e,S^v) \subseteq \textrm{AIP}^e \times \textrm{AIP}^v$, the sets 
$S^e$ and $S^v$ may or may not denote the index sets for all intersection points 
between a specific hyperplane $(\delta^l)^T x = \delta^l_0$ and the halflines 
$\{ v^i+\alpha r^j : \alpha \geq \alpha^*_j \}$ and 
$\{ v^i+\beta (v^k-v^i) : \beta \geq \beta^*_k \}$ for $i \in V^c$, 
$j \in E$, $k \in V \setminus V^c$ and $l \in I$. Let 
\begin{alignat}{2}
{\cal{S}}^* & := \{ (S^e,S^v) \subseteq \textrm{AIP}^e \times \textrm{AIP}^v : S^e=\textrm{IP}^e(l) \textrm{ and } S^v=\textrm{IP}^v(l) \textrm{ for some }l \in I \}\notag
\end{alignat}
denote the set of \emph{all} pairs $(S^e,S^v)$ that 
describe the index sets for the 
intersection points for \emph{some} cutting plane $l \in I$. 
For a given pair $(S^e,S^v) \in {\cal{S}}^*$, let
\begin{alignat}{2}
\textrm{CA}(S^e,S^v) & := \{ l \in I : S^e=\textrm{IP}^e(l) \textrm{ and } S^v=\textrm{IP}^e(l) \}\notag
\end{alignat}
denote the set of \emph{all} cuts associated with the pair
$(S^e,S^v)$, {\it i.e.}, the set of all cuts whose intersection 
points with the halflines $\{ v^i+\alpha r^j : \alpha \geq \alpha^*_j \}$ and 
$\{ v^i+\beta (v^k-v^i) : \beta \geq \beta^*_k \}$ for $i \in V^c$, 
$j \in E$ and $k \in V \setminus V^c$ are characterized by the pair 
$(S^e,S^v)$.

We claim that the finite number of sets $\{ \textrm{CA}(S^e,S^v) {\}}_{(S^e,S^v) \in
{\cal{S}}^*}$ provides the covering of $I$ that is claimed 
to exist in Lemma \ref{main_goal}. Indeed, the fact that $\textrm{IP}^e(l) \neq
\emptyset$ and $\textrm{IP}^v(l) \neq \emptyset$ for all $l \in I$ 
implies that every cut $l \in I$ belongs to \emph{some} set 
$\textrm{CA}(S^e,S^v)$ with $(S^e,S^v) \in {\cal{S}}^*$. 
Hence 
$\{ \textrm{CA}(S^e,S^v) {\}}_{(S^e,S^v) \in {\cal{S}}^*}$ 
is a covering of $I$. 

Let $(S^e,S^v) \in {\cal{S}}^*$ be arbitrary. If 
$\textrm{CA}(S^e,S^v) \subseteq \bar{I}$, then clearly 
the condition in Lemma \ref{main_goal} is satisfied for $(S^e,S^v)$, so we may assume 
$\textrm{CA}(S^e,S^v)$ contains cuts from $I \setminus \bar{I}$. 
Furthermore, 
we clearly have 
$\textrm{SIP}(I) \subseteq \textrm{SIP}(\textrm{CA}(S^e,S^v))$, 
since $\textrm{CA}(S^e,S^v)$ 
is a subset of $I$. To finish the proof of Lemma \ref{main_goal}, 
we need to show that $|\textrm{SIP}(\textrm{CA}(S^e,S^v))| >$
$|\textrm{SIP}(I)|$.

Lemma \ref{cat_lem}.(i) shows 
that for every $l \in I \setminus \bar{I}$ and 
$j \in E$ , there exists 
$i \in V^c$ such that $(i,j) \notin \textrm{SIP}^e(I)$ and 
$\alpha'_{i,j}(\delta^l,\delta^l_0) > \alpha^*_j$. 
Furthermore, \ref{cat_lem}.(ii) shows 
that for every $l \in I \setminus \bar{I}$ and 
$k \in V \setminus V^c$ , there exists 
$i \in V^c$ such that $(i,k) \notin \textrm{SIP}^v(I)$ and 
$\beta'_{i,k}(\delta^l,\delta^l_0) > \beta^*_k$.
This shows the existence of a cut 
$\bar{l} \in \textrm{CA}(S^e,S^v)$ that satisfies 
$\bar{l} \notin \textrm{SIP}(I)$, and therefore 
$|\textrm{SIP}(\textrm{CA}(S^e,S^v))| >$
$|\textrm{SIP}(I)|$. This completes the proof of Theorem \ref{pol_thm}.\\ 

\begin{proofof}{Lemma \ref{cat_lem}} The proof of (ii) is the 
same as the proof of (i), so we only show (i). 
Therefore suppose the cut $\bar{l} \in I \setminus \bar{I}$ 
and the extreme ray $\bar{j} \in E$ satisfies the inequality 
$\max\{ \alpha'_{i,\bar{j},\bar{l}} : i \in
  V^c \textrm{ and }(i,\bar{j}) \notin \textrm{SIP}^e(I)\} \leq \alpha^*_{\bar{j}}$. For simplicity 
let $(\delta',\delta'_0):=({\delta}^{l'},{\delta}^{l'}_0)$. 
 
Since $(\delta')^T x \geq \delta'_0$ is a non-negative cut for 
$P^{\bar{j}}$, Lemma \ref{lem3} shows there exists an inequality $\bar{\delta}^T x \ge $
$\bar{\delta}_0$ that dominates $(\delta')^T x \ge \delta'_0$ 
on $P^{\bar{j}}$, and that this inequality can be chosen to be a 
convex combination of the inequalities 
$({\delta}^l)^T x \ge {\delta}^l_0$ for 
$l \in \bar{I}$. Hence 
\begin{alignat}{2}
\bar{\delta} & = \sum_{l \in \bar{I}} \lambda_{l} \delta^l,\notag\\
\bar{\delta}_0 & = \sum_{l \in \bar{I}} \lambda_{l} \delta^l_0,\textrm{ where}\notag\\
\sum_{l \in \bar{I}} \lambda_{l} =
1 & \textrm{ and }\lambda_{l} \ge 0 \textrm{ for all }l \in
\bar{I}.\notag
\end{alignat}

We will show that $\bar{\delta}^T x \ge $
$\bar{\delta}_0$ also dominates 
$(\delta')^T x \ge \delta'_0$ on $P$ by 
verifying that conditions (i) and (ii) of Lemma \ref{dominance_lemma} are satisfied. 
Since $\bar{\delta}^T x \ge \bar{\delta}_0$ dominates 
$(\delta')^T x \ge \delta'_0$ on $P^{\bar{j}}$, 
we have 
\begin{alignat}{2}
\frac{1}{{\alpha}_{i,j}'(\bar{\delta},\bar{\delta}_0)} & \leq
\frac{1}{{\alpha}_{i,j}'({\delta'},\delta'_0)} \textrm{ for all }i
\in V^c\textrm{ and }j \in E \setminus \{ \bar{j} \},\textrm{ and}\notag\\
\frac{1}{{\beta}_{i,k}'(\bar{\delta},\bar{\delta}_0)} & \leq
\frac{1}{{\alpha}_{i,k}'({\delta'},\delta'_0)} \textrm{ for all }i
\in V^c\textrm{ and }k \in V \setminus V^c.\notag
\end{alignat}
We also know 
\begin{alignat}{2}
\frac{1}{\alpha'_{i,\bar{j}}(\bar{\delta},\bar{\delta}_0)} & =
\frac{1}{\alpha'_{i,\bar{j}}({\delta'},\delta'_0)} \textrm{ for all }i
\in V^c\textrm{ such that }(i,\bar{j}) \in \textrm{SIP}^e(I).\notag
\end{alignat}
To finish the proof, it suffices to show 
$$\frac{1}{{\alpha}_{i,\bar{j}}'(\bar{\delta},\bar{\delta}_0)} \leq
\frac{1}{{\alpha}_{i,\bar{j}}'({\delta'},\delta'_0)}\textrm{ for all
}i \in V^c\textrm{ such that }(i,\bar{j}) \notin \textrm{SIP}^e(I).$$
From the definition of $\alpha_{\bar{j}}^*$, we have the inequality 
$$\alpha^*_{\bar{j}} \leq \alpha'_{i,\bar{j},l} = \frac{\delta^l_0 - (\delta^l)^T v^{i}}{(\delta^l)^T
  r^{\bar{j}}}$$
for all $l \in \bar{I}$ and $i \in V^c$. The equalities 
$\bar{\delta}_0 - \bar{\delta}^T v^{i} = $
$\sum_{l \in \bar{I}} \lambda_{l} $ 
$(\delta^l_0 - (\delta^l)^T v^{i})$ for all $i \in V^c$, and 
$\bar{\delta}^T r^{\bar{j}} = $
$\sum_{l \in \bar{I}} \lambda_{l}$
$(\delta^l)^T r^{\bar{j}}$ imply 
$\bar{\delta}_0 - \bar{\delta}^T v^{i} \ge \alpha^*_{\bar{j}}
\bar{\delta}^T r^{\bar{j}}\textrm{ for all }i \in V^c$, 
and therefore 
$\alpha_{\bar{j}}^* \leq
{\alpha}_{i,\bar{j}}'(\bar{\delta},\bar{\delta}_0)$ for all $i \in V^c$. 
The definition of 
$\alpha_{\bar{j}}^*$ and the choice of the cut 
$(\delta')^T x \geq \delta'_0$ imply 
$\alpha_{\bar{i},\bar{j}}'(\delta',\delta'_0) \leq$ 
$\max\{ \alpha'_{i,\bar{j}}(\delta',\delta'_0) : i \in
  V^c \textrm{ such that }(i,\bar{j}) \notin \textrm{SIP}^e(I) \} \leq
\alpha_{\bar{j}}^*$ for all $\bar{i} \in V^c$ such that
$(\bar{i},\bar{j}) \notin \textrm{SIP}^e(I)$, and 
therefore $\alpha_{i,\bar{j}}'(\delta',\delta'_0) \leq $
$\alpha_{i,\bar{j}}'(\bar{\delta},\bar{\delta}_0)$ for all $i \in V^c$ 
such that $(i,\bar{j}) \notin \textrm{SIP}^e(I)$. Hence 
$1/{\alpha}_{i,\bar{j}}'(\bar{\delta},\bar{\delta}_0) \leq $
$1/{\alpha}_{i,\bar{j}}'(\delta',\delta'_0)$ for all $i \in V^c$ 
satisfying $(i,\bar{j}) \notin \textrm{SIP}^e(I)$, which 
gives that $\bar{\delta}^T x \geq \bar{\delta}_0$ 
dominates $(\delta')^T x \geq \delta'_0$ on $P$.
\end{proofof}

\section{The structure of polyhedral relaxations obtained from mixed integer split polyhedra}

We now describe the polyhedral structure 
of the polyhedron $R(L,P)$ for a 
mixed integer split polyhedron $L$. Throughout this section, 
$L$ denotes an arbitrary mixed integer split polyhedron. 
Also, $V^{\textrm{in}}(L):= \{ i \in V : v^i \in \intt(L) \}$ denotes the 
vertices of $P$ in the interior of $L$ and 
$V^{\textrm{out}}(L):= V \setminus V^{\textrm{in}}(L)$ denotes the 
vertices of $P$ that are \emph{not} in the interior of $L$. We assume 
$V^{\textrm{in}}(L) \neq \emptyset$, since otherwise $R(L,P)=P$ 
(Lemma \ref{int_vert_nec}). The set 
$\Lambda:=\{ \lambda \in \R^{|V|} :$
$\lambda\geq 0\textrm{ and }\sum_{i \in V} \lambda_i = 1 \}$ is used 
to form convex combinations of the vertices of $P$, and the set 
$\Lambda^{\textrm{in}}(L):=\{ \lambda \in \Lambda :$ 
$\sum_{i \in V^{\textrm{in}}(L)} \lambda_i = 1 \}$ is used to 
form convex combinations 
of the vertices in $V^{\textrm{in}}(L)$. 

\subsection{Intersection points}

Now 
consider possible intersection points between a halfline of the form 
$\{ v_{\lambda^{\textrm{in}}}+\alpha r^j : \alpha\geq 0 \}$ 
and the boundary of $L$, where 
$\lambda^{\textrm{in}} \in \Lambda^{\textrm{in}}(L)$ and $j \in E$. Given 
$\lambda^{\textrm{in}} \in \Lambda^{\textrm{in}}(L)$ and $j \in E$, define:
\begin{alignat}{2}
\alpha_j(L,\lambda^{\textrm{in}}) & := \sup\{ \alpha : v_{\lambda^{\textrm{in}}}+\alpha r^j \in L \}.\label{alpha_def2}
\end{alignat}

The number $\alpha_j(L,\lambda^{\textrm{in}}) > 0$ 
determines the closest point 
$v_{\lambda^{\textrm{in}}}+\alpha_j(L,\lambda^{\textrm{in}}) r^j$ 
(if any) to $v_{\lambda^{\textrm{in}}}$ 
on the halfline $\{ v_{\lambda^{\textrm{in}}}+\alpha r^j : \alpha\geq 0 \}$ 
which is \emph{not} in the interior of $L$. Observe that if 
$\{ v_{\lambda^{\textrm{in}}}+\alpha r^j : \alpha\geq 0 \} \subseteq \intt(L)$, 
then $\alpha_j(L,\lambda^{\textrm{in}})=+\infty$. When 
$\alpha_j(L,\lambda^{\textrm{in}}) < +\infty$, 
the point $v_{\lambda^{\textrm{in}}}+$ 
$\alpha_j(L,\lambda^{\textrm{in}}) r^j$ is called an 
\emph{intersection point}. 

The value $\alpha_j(L,\lambda^{\textrm{in}})$ 
is a function of $\lambda^{\textrm{in}}$. This function 
has the following important property. Given any 
convex set $C \subseteq \R^{n+1}$, it is well known 
(see Rockafellar \cite{Rockafellar}) that the function 
$f : \R^n \rightarrow \R$ defined by 
$$f(x):=\sup\{ \mu : (x,\mu) \in C \}$$
is a concave function. Now, given any $\lambda^{\textrm{in}} \in \Lambda^{\textrm{in}}(L)$ 
and $j \in E$, we may write 
$$\alpha_j(L,\lambda^{\textrm{in}})=\sup\{ \alpha : (\lambda^{\textrm{in}},\alpha) \in \tilde{P}(L) \},$$
where $\tilde{P}(L)$ is the convex polyhedron 
$\tilde{P}(L):=\{ (\lambda^{\textrm{in}},\alpha) \in \R^{|V^{\textrm{in}}(L)|+1} : v_{\lambda^{\textrm{in}}}+\alpha r^j \in L \}$. We therefore have that 
the function 
$\alpha_j(L,\lambda^{\textrm{in}})$ has the following property.
\begin{lemma}\label{alpha_concave}
Let $L$ be a mixed integer split polyhedron satisfying 
$V^{\textrm{in}}(L) \neq \emptyset$, and let $j \in E$. 
The function $\alpha_j(L,\lambda^{\textrm{in}})$ is concave 
in $\lambda^{\textrm{in}}$, 
{\it i.e.}, for every $\lambda^1, \lambda^2 \in \Lambda^{\textrm{in}}(L)$ and 
$\mu \in [0,1]$, we have 
$\alpha_j(L,\mu \lambda^1 + (1-\mu)\lambda^2) \geq$
$\mu \alpha_j(L,\lambda^1)+(1-\mu) \alpha_j(L,\lambda^2)$.
\end{lemma}

Given a convex combination 
$\lambda^{\textrm{in}} \in \Lambda^{\textrm{in}}(L)$, 
and a vertex $k \in V^{\textrm{out}}(L)$, the line 
between $v_{\lambda^{\textrm{in}}}$ and $v^k$ 
intersects the boundary of $L$. 
For $k \in V^{\textrm{out}}(L)$ and 
$\lambda^{\textrm{in}} \in \Lambda^{\textrm{in}}(L)$, 
define  
\begin{alignat}{2} \label{beta_def2}
\beta_k(L,\lambda^{\textrm{in}}) & := \sup\{ \beta : v_{\lambda^{\textrm{in}}}+\beta(v^k-v_{\lambda^{\textrm{in}}}) \in L \}.
\end{alignat}
\noindent The number ${\beta}_k(L,\lambda^{\textrm{in}})$ denotes 
the value of $\beta$ for which the point 
$v_{\lambda^{\textrm{in}}}+\beta (v^k-v_{\lambda^{\textrm{in}}})$ 
is on the boundary of $L$. The point 
$v_{\lambda^{\textrm{in}}}+\beta (v^k-v_{\lambda^{\textrm{in}}})$ is 
also called an intersection point, and we 
observe that 
$\beta_k(L,\lambda^{\textrm{in}}) \in ]0,1]$. The intersection point 
$v_{\lambda^{\textrm{in}}}+\beta_k(L,\lambda^{\textrm{in}}) (v^k-v_{\lambda^{\textrm{in}}})$ 
has the following important property.
\begin{lemma}\label{beta_properties}
Let $L$ be a mixed integer split polyhedron satisfying 
$V^{\textrm{in}}(L) \neq \emptyset$, and let $k \in V^{\textrm{out}}(L)$. 
For every $\lambda^{\textrm{in}} \in \Lambda^{\textrm{in}}(L)$, the intersection 
point 
$v_{\lambda^{\textrm{in}}}+\beta_k(L,\lambda^{\textrm{in}}) (v^k-v_{\lambda^{\textrm{in}}})$ 
is a convex combination of $v^k$ and the intersection points 
$v^i+\beta_k(L,e^i) (v^k-v^i)$ 
for $i \in V^{\textrm{in}}(L)$. 
\end{lemma}
\begin{proof}
Define 
$C:=\conv(\{ v^i+\beta_k(L,e^i) (v^k-v^i) {\}}_{i \in V^{\textrm{in}}(L)})$. 
We first show that the halfline 
$\{ v_{\lambda^{\textrm{in}}} + \beta (v^k - v_{\lambda^{\textrm{in}}}) : \beta \geq 0 \}$ 
intersects $C$ for some $\beta^* > 0$. We have that 
$\{ v_{\lambda^{\textrm{in}}} + \beta (v^k - v_{\lambda^{\textrm{in}}}) : \beta \geq 0 \} \cap C \neq \emptyset$ if and only if the following $\textrm{LP}$ is feasible.
\begin{alignat}{2}
\min & \,\,\, 0 & \notag\\
& \sum_{i \in V^{\textrm{in}}(L)} \eta_i (v^i+\beta_k(L,e^i) (v^k - v^i))+\beta (v_{\lambda^{\textrm{in}}} - v^k) = v_{\lambda^{\textrm{in}}},\label{beta_proof1}\\
& \sum_{i \in V^{\textrm{in}}(L)} \eta_i = 1,\label{beta_proof2}\\
& \eta,\beta \geq 0.\label{beta_proof3}
\end{alignat}
The dual of this $\textrm{LP}$ is given by
\begin{alignat}{2}
\max & \,\,\, \delta^T v_{\lambda^{\textrm{in}}} - \delta_0 & \notag\\
& \delta^T (v_{\lambda^{\textrm{in}}}-v^k) \leq 0,\label{beta_proof4}\\
& \delta^T (v^i+\beta_k(L,e^i) (v^k - v^i)) - \delta_0 \leq 0,\textrm{ for all }i \in V^{\textrm{in}}(L).\label{beta_proof5}
\end{alignat}
Let $(\bar{\delta},\bar{\delta}_0)$ be a solution to 
(\ref{beta_proof4})-(\ref{beta_proof5}). Suppose, for a 
contradiction, that $\bar{\delta}^T v_{\lambda^{\textrm{in}}} - \bar{\delta}_0 > 0$. 
Adding (\ref{beta_proof4}) to the inequality of (\ref{beta_proof5}) 
corresponding to $\bar{i} \in V^{\textrm{in}}(L)$ gives 
$\bar{\delta}^T v_{\lambda^{\textrm{in}}} - \bar{\delta}_0+(1-\beta_k(L,e^{\bar{i}}))\bar{\delta}^T (v^{\bar{i}}-v^k) \leq 0$. 
Since by assumption $\bar{\delta}^T v_{\lambda^{\textrm{in}}} - \bar{\delta}_0 > 0$, this 
implies $\bar{\delta}^T (v^{\bar{i}}-v^k) < 0$. Hence we have $\bar{\delta}^T (v^i-v^k) < 0$ 
for all $i \in V^{\textrm{in}}(L)$. Now, for all $i \in V^{\textrm{in}}(L)$, 
inequality (\ref{beta_proof5}) gives $\bar{\delta}_0-\bar{\delta}^T v^i \geq$
$\beta_k(L,e^i) \bar{\delta}^T (v^k - v^i)$. Since 
$\bar{\delta}^T (v^k - v^i) > 0$ for all $i \in V^{\textrm{in}}(L)$, this 
implies $\bar{\delta}_0-\bar{\delta}^T v^i > 0$ for all $i \in V^{\textrm{in}}(L)$. 
Multiplying each of the inequalities 
$\bar{\delta}_0-\bar{\delta}^T v^i > 0$ for $i \in V^{\textrm{in}}(L)$ 
with $\lambda^{\textrm{in}}_i$ and adding the resulting inequalities together then gives 
$\bar{\delta}_0-\bar{\delta}^T v_{\lambda^{\textrm{in}}} > 0$. This 
contradicts our initial assumption that 
$\bar{\delta}^T v_{\lambda^{\textrm{in}}} - \bar{\delta}_0 > 0$.

Therefore there exists $\beta^* \geq 0$ s.t. 
$v_{\lambda^{\textrm{in}}}+\beta^* (v^k-v_{\lambda^{\textrm{in}}}) \in C$. 
Observe that, since $v^i+\beta_k(L,e^i) (v^k-v^i) \in L$ for all 
$i \in V^{\textrm{in}}(L)$, we have 
$v_{\lambda^{\textrm{in}}}+\beta^* (v^k-v_{\lambda^{\textrm{in}}}) \in L$. 
If 
$v_{\lambda^{\textrm{in}}} + \beta^* (v^k - v_{\lambda^{\textrm{in}}}) \in \intt(L)$, 
then $\beta_k(L,\lambda^{\textrm{in}}) > \beta^*$, and therefore 
$v_{\lambda^{\textrm{in}}}+$
$\beta_k(L,\lambda^{\textrm{in}}) (v^k-v_{\lambda^{\textrm{in}}}) \in \conv(C \cup \{ v^k \})$. 
If $v_{\lambda^{\textrm{in}}} + \beta^* (v^k - v_{\lambda^{\textrm{in}}})$ is on 
the boundary of $L$, then $\beta_k(L,\lambda^{\textrm{in}}) = \beta^*$, which 
implies 
$v_{\lambda^{\textrm{in}}}+\beta_k(L,\lambda^{\textrm{in}})
(v^k-v_{\lambda^{\textrm{in}}}) \in \conv(C \cup \{ v^k \})$.
\end{proof}

\begin{figure}
\centering
\mbox{\subfigure[The polytope $P$ and the split polyhedron $L$ from Figure
\ref{first_fig}]{\includegraphics[width=5cm]{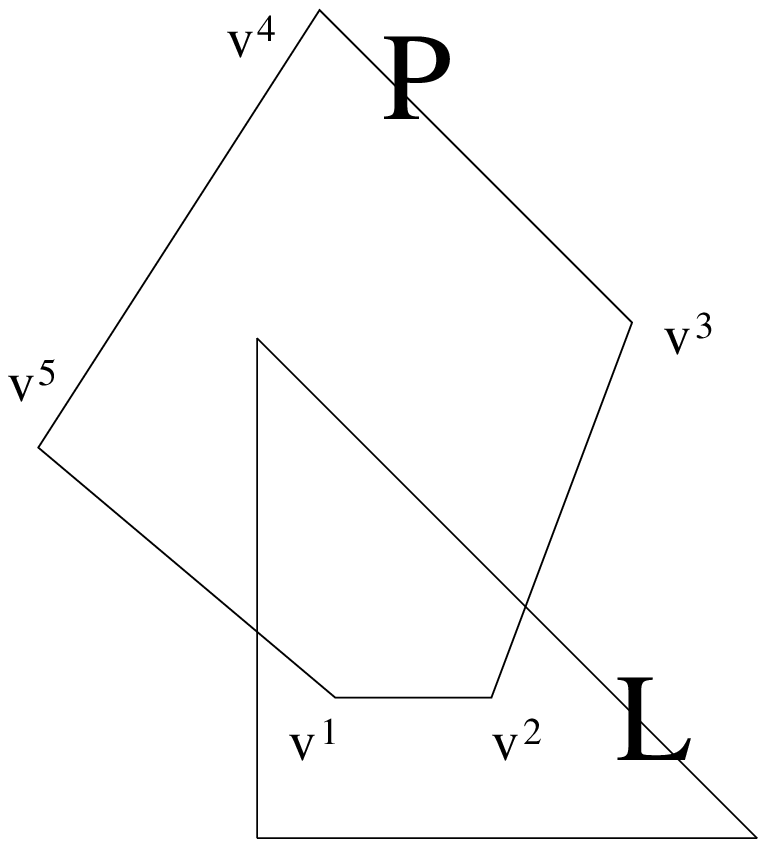}}\quad
\subfigure[Intersection points from $v^1$]{\includegraphics[width=5cm]{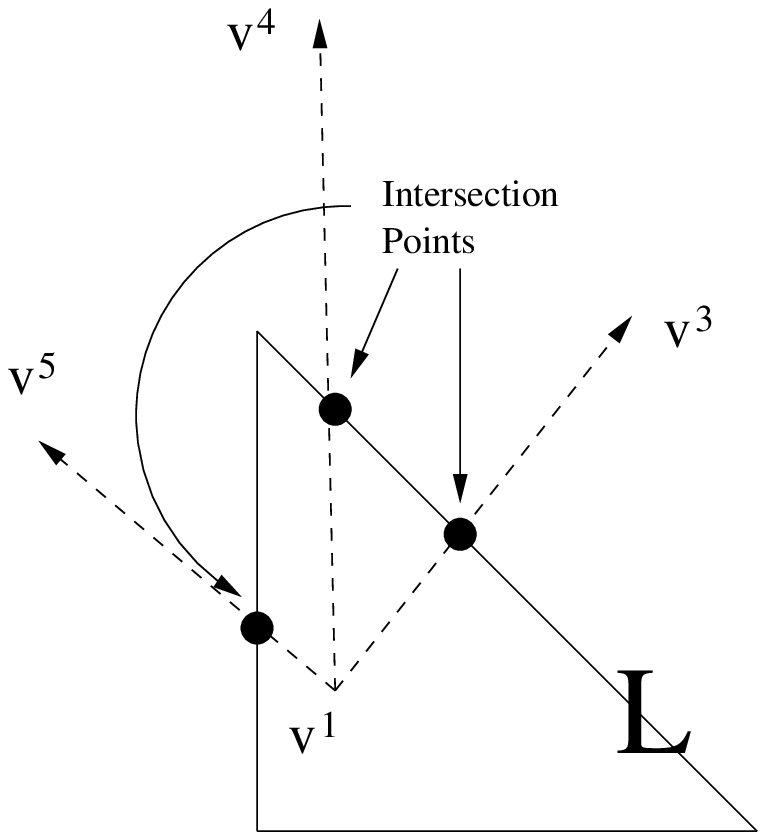}}
\subfigure[Intersection points from $v^2$]{\includegraphics[width=5cm]{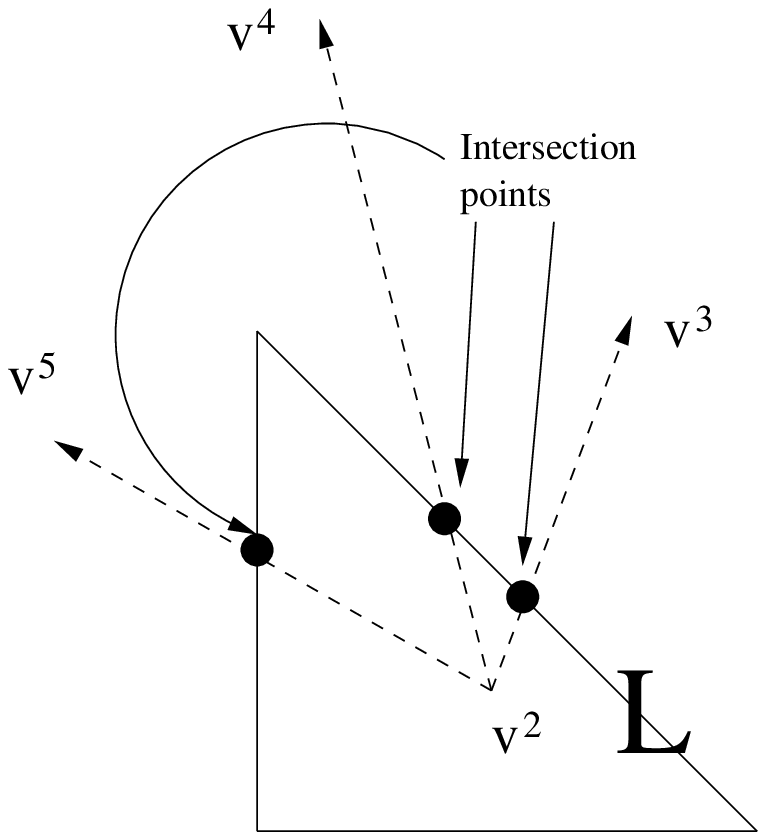}}}
\caption{Determining the intersection points from a linear relaxation $P$ and a split
polyhedron $L$}
\label{inter_points}
\end{figure}

Lemma \ref{beta_properties} shows that the only intersection 
points of the form
$v_{\lambda^{\textrm{in}}}+\beta_k(L,\lambda^{\textrm{in}})
(v^k-v_{\lambda^{\textrm{in}}})$ that can be vertices of 
$R(L,P)$ are those where $\lambda^{\textrm{in}}$ is a 
unit vector. Figure \ref{inter_points} gives all the 
intersection points which can potentially be vertices 
of $R(L,P)$ for the example of Figure \ref{first_fig}.

\subsection{The intersection cut}

In \cite{balasint}, Balas considered a mixed integer set defined 
from the translate of a polyhedral cone, and 
a mixed integer split polyhedron was used to 
derive a valid inequality for this set called 
the \emph{intersection cut}. We 
now consider a subset $P(\lambda^{\textrm{in}})$ 
of $P$ defined from a fixed convex combination 
$\lambda^{\textrm{in}} \in \Lambda^{\textrm{in}}(L)$ of 
the vertices in the interior of $L$, and we show that 
the intersection cut gives a complete 
description of the set $R(L,P(\lambda^{\textrm{in}}))$ 
in a higher dimensional space. Specifically, 
given any fixed convex combination 
$\lambda^{\textrm{in}} \in \Lambda^{\textrm{in}}(L)$, 
we have the following subset $P(\lambda^{\textrm{in}})$ of $P$ 
$$P(\lambda^{\textrm{in}})=\{ x \in \R^n : x= v_{\lambda^{\textrm{in}}}+\sum_{k \in V^{\textrm{out}}(L)} \epsilon_k (v^k-v_{\lambda^{\textrm{in}}})+\sum_{j \in E} \mu_j r^j, \mu \geq 0 \textrm{ and } \epsilon \in \Lambda^{\textrm{out}}_{\leq}\},$$
where $\Lambda^{\textrm{out}}_{\leq}:=$
$\{ \lambda \in \R^{|V|} : \sum_{k \in V^{\textrm{out}}} \lambda_k \leq 1\textrm{ and }\lambda\geq 0 \}$. The corresponding lifted image $P^l(\lambda^{\textrm{in}})$ of 
$P(\lambda^{\textrm{in}})$ in $(x,\epsilon,\mu)$ space is given by 
$$P^l(\lambda^{\textrm{in}})=\{ (x,\epsilon,\mu) \in \R^{n+|V|+|E|} : x= v_{\lambda^{\textrm{in}}}+\sum_{k \in V^{\textrm{out}}(L)} \epsilon_k (v^k-v_{\lambda^{\textrm{in}}})+\sum_{j \in E} \mu_j r^j, \mu \geq 0 \textrm{ and } \epsilon \in \Lambda^{\textrm{out}}_{\leq}\}.$$

The set $P(\lambda^{\textrm{in}})$ and the mixed integer 
split polyhedron $L$ gives a relaxation 
$R(L,P(\lambda^{\textrm{in}}))$ of the set of mixed integer 
points in $P(\lambda^{\textrm{in}})$ 
$$R(L,P(\lambda^{\textrm{in}}))=\conv( \{ x \in P(\lambda^{\textrm{in}}) : x \notin \intt(L) \}).$$
The lifted version 
$R^l(L,P(\lambda^{\textrm{in}}))$ of $R(L,P(\lambda^{\textrm{in}}))$ 
in $(x,\epsilon,\mu)$ space is then defined to be the 
set 
$R^l(L,P(\lambda^{\textrm{in}})):=$
$\conv( \{ (x,\epsilon,\mu) \in P^l(\lambda^{\textrm{in}}) :$
$x \notin \intt(L) \})$. 
Given $\lambda^{\textrm{in}} \in \Lambda^{\textrm{in}}(L)$, 
and the corresponding intersection points, Balas \cite{balasint} derived the 
\emph{intersection cut}
\begin{alignat}{2}
\sum_{j \in E} \frac{\mu_j}{\alpha_j(L,\lambda^{\textrm{in}})}+ \sum_{k \in V^{\textrm{out}}(L)} \frac{\epsilon_k}{\beta_k(L,\lambda^{\textrm{in}})}& \geq 1\label{int_cut}
\end{alignat}
and showed that the intersection cut is valid for $R^l(L,P(\lambda^{\textrm{in}}))$. 
We now show that, in fact, the intersection cut gives a 
complete description of $R^l(L,P(\lambda^{\textrm{in}}))$.
\begin{theorem}\label{fixed_lambda_lem}
Let $L$ be a mixed integer split polyhedron 
satisfying $V^{\textrm{in}}(L) \neq \emptyset$, and let 
$\lambda^{\textrm{in}} \in \Lambda^{\textrm{in}}(L)$. 
$$R^l(L,P(\lambda^{\textrm{in}}))=\{ (x,\epsilon,\mu) \in P^l(\lambda^{\textrm{in}}) : {\sum_{j \in E}} \frac{\mu_j}{\alpha_j(L,\lambda^{\textrm{in}})}+ \sum_{k \in V^{\textrm{out}}(L)} \frac{\epsilon_k}{\beta_k(L,\lambda^{\textrm{in}})} \geq 1\}.$$
\end{theorem}
\begin{proof} 
Since (\ref{int_cut}) 
is valid for $R^l(L,P(\lambda^{\textrm{in}}))$, we have 
$$R^l(L,P(\lambda^{\textrm{in}})) \subseteq \{ (x,\epsilon,\mu) \in P^l(\lambda^{\textrm{in}}) : {\sum_{j \in E}} \frac{\mu_j}{\alpha_j(L,\lambda^{\textrm{in}})}+ \sum_{k \in V^{\textrm{out}}(L)} \frac{\epsilon_k}{\beta_k(L,\lambda^{\textrm{in}})} \geq 1\}.$$
Conversely suppose 
$(\bar{x},\bar{\epsilon},\bar{\mu}) \in P^l(\lambda^{\textrm{in}})$ and 
$\sum_{j \in E} \frac{\bar{\mu}_j}{\alpha_j(L,\lambda^{\textrm{in}})}+$
$\sum_{k \in V^{\textrm{out}}(L)} \frac{\bar{\epsilon}_k}{\beta_k(L,\lambda^{\textrm{in}})} \geq 1$. We 
will show that $(\bar{x},\bar{\epsilon},\bar{\mu}) \in$ 
$R^l(L,P(\lambda^{\textrm{in}}))$. 
Define $E^{\infty}:=\{ j \in E : \alpha_j(L,\lambda^{\textrm{in}})=+\infty \}$. 
We distinguish four cases.
\begin{enumerate}
\item[(1)] First suppose 
$\sum_{j \in E} \frac{\bar{\mu}_j}{\alpha_j(L,\lambda^{\textrm{in}})}+$ 
$\sum_{k \in V^{\textrm{out}}(L)} \frac{\bar{\epsilon}_k}{\beta_k(L,\lambda^{\textrm{in}})}=1$. 
We 
can write 
\begin{alignat}{2}
\left( \begin{array}{c} \bar{x} \\ \bar{\epsilon} \\ \bar{\mu} \end{array} \right)= & \sum_{k \in V^{\textrm{out}}(L)} \bar{\eta}_k \left( \begin{array}{c} v_{\lambda^{\textrm{in}}}+\beta_k(L,\lambda^{\textrm{in}}) (v^k-v_{\lambda^{\textrm{in}}}) \\ \beta_k(L,\lambda^{\textrm{in}}) e^k \\ 0 \end{array} \right)+\notag\\
& \sum_{j \in E \setminus E^{\infty}}
\bar{\kappa}_j \left( \begin{array}{c} v_{\lambda^{\textrm{in}}}+\alpha_j(L,\lambda^{\textrm{in}}) r^j \\ 0 \\ \alpha_j(L,\lambda^{\textrm{in}}) e^j \end{array} \right)+\sum_{j \in E^{\infty}} \bar{\mu}_j \left( \begin{array}{c} r^j \\ 0 \\ e^j \end{array} \right),\notag
\end{alignat}
where
$\bar{\kappa}_j:=\frac{\bar{\mu}_j}{\alpha_j(L,\lambda^{\textrm{in}})}$ for 
$j \in E \setminus E^{\infty}$ and 
$\bar{\eta}_k:=\frac{\bar{\epsilon}_k}{\beta_k(L,\lambda^{\textrm{in}})}$ 
for $k \in V^{\textrm{out}}(L)$. 
Since $v_{\lambda^{\textrm{in}}}+$
$\alpha_j(L,\lambda^{\textrm{in}}) r^j \notin$
$\intt(L)$ for 
$j \in E \setminus E^{\infty}$, 
$v_{\lambda^{\textrm{in}}}+$
$\beta_k(L,\lambda^{\textrm{in}}) (v^k-v_{\lambda^{\textrm{in}}}) \notin$
$\intt(L)$ for $k \in V^{\textrm{out}}(L)$ and 
$(r^j,0,e^j)$ is in the 
recession cone of $R^l(L,P(\lambda^{\textrm{in}}))$ for $j \in E^{\infty}$, 
we have $(\bar{x},\bar{\epsilon},\bar{\mu}) \in R^l(L,P(\lambda^{\textrm{in}}))$.
\item[(2)] Now suppose 
$\sum_{j \in E} \frac{\bar{\mu}_j}{\alpha_j(L,\lambda^{\textrm{in}})}+$
$\sum_{k \in V^{\textrm{out}}(L)} \frac{\bar{\epsilon}_k}{\beta_k(L,\lambda^{\textrm{in}})}>1$ 
and $\sum_{k \in V^{\textrm{out}}(L)} \bar{\epsilon}_k = 1$. 
This implies $\bar{x}=\sum_{k \in V^{\textrm{out}}(L)} \bar{\epsilon}_k v^k+$
$\sum_{j \in E} \bar{\mu}_j r^j$. Since $(r^j,0,e^j)$ is in the recession cone 
of $R^l(L,P,\lambda^{\textrm{in}})$ for $j \in E$, and since 
$v^k \notin \intt(L)$ for $k \in V^{\textrm{out}}(L)$, 
we have $(\bar{x},\bar{\epsilon},\bar{\mu}) \in R^l(L,P(\lambda^{\textrm{in}}))$.
\item[(3)] Next suppose 
$\sum_{j \in E} \frac{\bar{\mu}_j}{\alpha_j(L,\lambda^{\textrm{in}})}+$
$\sum_{k \in V^{\textrm{out}}(L)} \frac{\bar{\epsilon}_k}{\beta_k(L,\lambda^{\textrm{in}})}>1$ 
and $0 < \sum_{k \in V^{\textrm{out}}(L)} \bar{\epsilon}_k < 1$. 
Let $\bar{\delta} \in ]0,1[$ be such that 
$\bar{y}:=\bar{\delta} v_{\lambda^{\textrm{in}}} + (1-\bar{\delta}) \bar{x}=$ 
$v_{\lambda^{\textrm{in}}}+$
$\sum_{k \in V^{\textrm{out}}(L)}$ 
$(1-\bar{\delta}) \bar{\epsilon}_k (v^k-v_{\lambda^{\textrm{in}}})+$
$\sum_{j \in E} (1-\bar{\delta}) \bar{\mu}_j r^j$ satisfies 
$\sum_{j \in E} \frac{(1-\bar{\delta})\bar{\mu}_j}{\alpha_j(L,\lambda^{\textrm{in}})}+$
$\sum_{k \in V^{\textrm{out}}(L)} \frac{(1-\bar{\delta})\bar{\epsilon}_k}{\beta_k(L,\lambda^{\textrm{in}})}=1$. 
It follows from (1) that 
$(\bar{y},(1-\bar{\delta})\bar{\epsilon},(1-\bar{\delta})\bar{\mu}) \in R^l(L,P(\lambda^{\textrm{in}}))$. 
Let $d:=\bar{y}-v_{\lambda^{\textrm{in}}}$, 
and consider the halfline $\{ v_{\lambda^{\textrm{in}}} + \alpha d : \alpha \geq 0 \}$. For 
$\alpha_{\bar{y}}:=1$, we have $v_{\lambda^{\textrm{in}}} + \alpha_{\bar{y}} d=\bar{y}$, and 
for $\alpha_{\bar{x}}:=\frac{1}{1-\bar{\delta}}$, we have 
$v_{\lambda^{\textrm{in}}} + \alpha_{\bar{x}} d=\bar{x}$. Consider the 
point $\bar{z}:=v_{\lambda^{\textrm{in}}} + \alpha_{\bar{z}} d$, where 
$\alpha_{\bar{z}}:=$
$\frac{1}{(1-\bar{\delta})\sum_{k \in V^{\textrm{out}}(L)} \bar{\epsilon}_k}$. 
Since 
$\sum_{k \in V^{\textrm{out}}(L)} \bar{\epsilon}_k \in ]0,1[$, 
we have $\alpha_{\bar{y}} < \alpha_{\bar{x}} < \alpha_{\bar{z}}<+\infty$. Hence 
$\bar{x}$ is a convex combination of $\bar{y}$ 
and $\bar{z}$. We may write 
$\bar{z}=v_{\lambda^{\textrm{in}}} + \alpha_{\bar{z}} d=$
$v_{\lambda^{\textrm{in}}}+$
$\sum_{k \in V^{\textrm{out}}(L)}$ 
$\alpha_{\bar{z}} (1-\bar{\delta}) \bar{\epsilon}_k (v^k - v_{\lambda^{\textrm{in}}})+$
$\sum_{j \in E} \alpha_{\bar{z}} (1-\bar{\delta}) \bar{\mu}_j r^j$. Observe that 
$\sum_{k \in V^{\textrm{out}}(L)}$ 
$\alpha_{\bar{z}} (1-\bar{\delta}) \bar{\epsilon}_k=1$. Hence we can 
write 
$\bar{z}=\sum_{k \in V^{\textrm{out}}(L)} \bar{\eta}_k v^k+$
$\sum_{j \in E} \alpha_{\bar{z}} \bar{\mu}_j r^j$, where 
$\bar{\eta}_k:=\alpha_{\bar{z}} (1-\bar{\delta}) \bar{\epsilon}_k$ for 
$k \in V^{\textrm{out}}(L)$ and 
$\sum_{k \in V^{\textrm{out}}(L)} \bar{\eta}_k =1$. 
Since $r^j$ is in the recession cone 
of $R(L,P(\lambda^{\textrm{in}}))$ for $j \in E$, and since 
$v^k \in R(L,P(\lambda^{\textrm{in}}))$ for $k \in V^{\textrm{out}}(L)$, 
we have $\bar{z} \in R(L,P(\lambda^{\textrm{in}}))$. Since $\bar{x}$ is 
a convex combination of $\bar{y} \in R(L,P(\lambda^{\textrm{in}}))$ and 
$\bar{z} \in R(L,P(\lambda^{\textrm{in}}))$, we have 
$\bar{x} \in R(L,P(\lambda^{\textrm{in}}))$.
\item[(4)] Finally suppose $\sum_{k \in V^{\textrm{out}}(L)} \bar{\epsilon}_k = 0$ 
and $\sum_{j \in E} \frac{\bar{\mu}_j}{\alpha_j(L,\lambda^{\textrm{in}})}>1$. 
As in (3), let $\bar{\delta} \in ]0,1[$ be s.t. 
$\bar{y}:=\bar{\delta} v_{\lambda^{\textrm{in}}} + (1-\bar{\delta}) \bar{x}=$ 
$v_{\lambda^{\textrm{in}}}+\sum_{j \in E} (1-\bar{\delta}) \bar{\mu}_j r^j$ satisfies 
$\sum_{j \in E} \frac{(1-\bar{\delta}) \bar{\mu}_j}{\alpha_j(L,\lambda)}=1$. 
From (1) we have 
$(\bar{y},0,(1-\bar{\delta}) \bar{\mu}) \in R^l(L,P(\lambda^{\textrm{in}}))$, 
and since $\bar{y}=\bar{\delta} v_{\lambda^{\textrm{in}}} + (1-\bar{\delta})
\bar{x}$, we have $\bar{x}=v_{\lambda^{\textrm{in}}}+$
$\bar{\sigma}(\bar{y}-v_{\lambda^{\textrm{in}}})$, where 
$\bar{\sigma}:=\frac{1}{1-\bar{\delta}}$. Since 
$(\bar{y},0,\frac{\bar{\mu}}{\bar{\sigma}}) \in$ 
$R^l(L,P(\lambda^{\textrm{in}}))$ satisfies 
$\sum_{j \in E} \frac{(1-\bar{\delta})\bar{\mu}_j}{\alpha_j(L,\lambda^{\textrm{in}})}=1$, 
(1) shows 
\begin{alignat}{2}
\left( \begin{array}{c} \bar{y} \\ 0 \\ \frac{\bar{\mu}}{\bar{\sigma}} \end{array}\right) = & \sum_{j \in E \setminus E^{\infty}} \bar{\kappa}_j \left( \begin{array}{c} v_{\lambda^{\textrm{in}}}+\alpha_j(L,\lambda^{\textrm{in}}) r^j \\ 0 \\ \alpha_j(L,\lambda^{\textrm{in}}) e^j \end{array} \right)+\sum_{j \in E^{\infty}} \bar{\gamma}_j \left( \begin{array}{c} r^j \\ 0 \\ e^j \end{array} \right),\notag
\end{alignat}
where 
$\sum_{j \in E \setminus E^{\infty}} \bar{\kappa}_j = 1$, 
$\bar{\kappa}_j \geq 0$ for $j \in E \setminus E^{\infty}$ 
and 
$\bar{\gamma}_j \geq 0$ for $j \in E^{\infty}$. We can now write 
\begin{alignat}{2}
\left( \begin{array}{c} \bar{x} \\ 0 \\ \bar{\mu} \end{array}\right) = & \sum_{j \in E \setminus E^{\infty}} \bar{\kappa}_j \left( \begin{array}{c} v_{\lambda^{\textrm{in}}}+\bar{\sigma} \alpha_j(L,\lambda^{\textrm{in}}) r^j \\ 0 \\  \bar{\sigma} \alpha_j(L,\lambda^{\textrm{in}}) e^j \end{array} \right)+\sum_{j \in E^{\infty}} \bar{\sigma}  \bar{\gamma}_j \left( \begin{array}{c} r^j \\ 0 \\ e^j \end{array} \right).\notag
\end{alignat}
\end{enumerate}
\end{proof}

\subsection{The vertices of $R(L,P)$}

The proof of Theorem \ref{fixed_lambda_lem} allows us to 
characterize the vertices of $R(L,P)$. Observe that in the 
proof of Theorem \ref{fixed_lambda_lem}, 
every point in $R(L,P(\lambda^{\textrm{in}}))$ is expressed 
in terms of intersection points, 
vertices of $P$ that are not in the interior of $L$ and the extreme 
rays $r^j$ of $R(L,P(\lambda^{\textrm{in}}))$ for $j \in E$. 
Hence the proof of 
Theorem \ref{fixed_lambda_lem} provides a characterization 
of the vertices of $R(L,P(\lambda^{\textrm{in}}))$. 
\begin{corollary}\label{vert_char_lambda}
Let $L$ be a mixed integer split polyhedron 
satisfying $V^{\textrm{in}}(L) \neq \emptyset$, 
and let $\lambda^{\textrm{in}} \in \Lambda^{\textrm{in}}(L)$. 
Define $E^{\infty}(\lambda^{\textrm{in}}):=$
$\{ j \in E : \alpha_j(L,\lambda^{\textrm{in}})=+\infty \}$.
A vertex of $R(L,P(\lambda^{\textrm{in}}))$ is of one of the 
following forms.\pagebreak
\begin{enumerate}
\item[(i)] A vertex $v^k$ of $P$, where $k \in V^{\textrm{out}}(L)$,
\item[(ii)] An intersection point $v_{\lambda^{\textrm{in}}}+$
$\beta_k(L,\lambda^{\textrm{in}})(v^k-v_{\lambda^{\textrm{in}}})$, 
where $k \in V^{\textrm{out}}(L)$, or 
\item[(iii)] An intersection point 
$v_{\lambda^{\textrm{in}}}+\alpha_j(L,\lambda^{\textrm{in}}) r^j$, where 
$j \in E \setminus E^{\infty}(\lambda^{\textrm{in}})$.
\end{enumerate}
\end{corollary}

By using the properties of $\alpha_j(L,\lambda^{\textrm{in}})$ 
and $\beta_k(L,\lambda^{\textrm{in}})$ for 
$\lambda^{\textrm{in}} \in \Lambda^{\textrm{in}}(L)$ 
given in Lemma \ref{alpha_concave} and 
Lemma \ref{beta_properties}, 
we can use Corollary \ref{vert_char_lambda} to characterize 
the vertices of $R(L,P)$. In the following, for simplicity 
let $\alpha_{i,j}(L):=\alpha_j(L,e^i)$ and 
$\beta_{i,k}(L):=\beta_k(L,e^i)$ for $i \in V^{\textrm{in}}(L)$, 
$j \in E$ and $k \in V^{\textrm{out}}(L)$. Also let 
$E^{\infty}(L):=$ 
$\{ j \in E : \alpha_{i,j}(L)=+\infty \textrm{ for some } i \in V^{\textrm{in}}(L)\}$ 
denote those extreme rays of $P$ that are also rays of $L$. 

\begin{lemma}\label{rel_descr}
Let $L$ be a mixed integer split polyhedron 
satisfying $V^{\textrm{in}}(L) \neq \emptyset$. 
Every vertex of $R(L,P)$ is of one of the following the forms.
\begin{enumerate}
\item[(i)] A vertex $v^k$ of $P$, where $k \in V^{\textrm{out}}(L)$,
\item[(ii)] An intersection point $v^i+$
$\beta_{i,k}(L)(v^k-v^i)$, 
where $i \in V^{\textrm{in}}(L)$ and $k \in V^{\textrm{out}}(L)$, or 
\item[(iii)] An intersection point 
$v^i+\alpha_{i,j}(L) r^j$, where 
$i \in V^{\textrm{in}}(L)$ and $j \in E \setminus E^{\infty}(L)$.
\end{enumerate}
\end{lemma}
\begin{proof}
Let $\bar{x} \in R(L,P)$ be a vertex of 
$R(L,P)$, and let $(\bar{\lambda}^{\textrm{in}},\bar{\epsilon},\bar{\mu}) \in \R^{|V|+|E|}$ satisfy 
$\bar{x}=v_{\bar{\lambda}^{\textrm{in}}}+$
$\sum_{k \in V^{\textrm{out}}(L)} \bar{\epsilon}_k (v^k-v_{\bar{\lambda}^{\textrm{in}}})+$
$\sum_{j \in E} \bar{\mu}_j r^j$, $\bar{\epsilon} \geq 0$, 
$\bar{\mu} \geq 0$, $\bar{\lambda}^{\textrm{in}} \in \Lambda^{\textrm{in}}(L)$ 
and $\sum_{k \in V^{\textrm{out}}(L)} \bar{\epsilon}_k \leq 1$. Now, we have 
$\bar{x} \in R(L,P(\bar{\lambda}^{\textrm{in}}))$, and since 
$R(L,P(\bar{\lambda}^{\textrm{in}})) \subseteq R(L,P)$, 
we must have that $\bar{x}$ is 
a vertex of $R(L,P(\bar{\lambda}^{\textrm{in}}))$. It follows that $\bar{x}$ is of one of the forms Corollary \ref{vert_char_lambda}.(i)-(iii). If $\bar{x}$ is of the form 
$\bar{x}=v^k$ for some $k \in V^{\textrm{out}}(L)$, we are done. 
Furthermore, if 
$\bar{x}=v_{\bar{\lambda}^{\textrm{in}}}+\beta_k(L,\bar{\lambda}^{\textrm{in}})(v^k-v_{\bar{\lambda}^{\textrm{in}}})$ 
for some $k \in V^{\textrm{out}}(L)$, then Lemma \ref{beta_properties} shows 
that either $\bar{x}=v^k$, or 
$\bar{x}=v^{\bar{i}}+\beta_{\bar{i},k}(L)(v^k-v^{\bar{i}})$ for some 
$\bar{i} \in V^{\textrm{out}}(L)$. 

Finally consider the case when $\bar{x}$ is of the form 
$\bar{x}=v_{\bar{\lambda}^{\textrm{in}}}+$
$\alpha_{\bar{j}}(L,\bar{\lambda}^{\textrm{in}}) r^{\bar{j}}$ 
for some $\bar{j} \in E \setminus E^{\infty}(L)$. 
Since $\alpha_{\bar{j}}(L,\bar{\lambda}^{\textrm{in}})$ 
is concave in $\bar{\lambda}^{\textrm{in}}$, 
we have $\alpha_{\bar{j}}(L,\bar{\lambda}^{\textrm{in}}) \geq$
$\sum_{i \in V^{\textrm{in}}(L)} \bar{\lambda}^{\textrm{in}}_i \alpha_{i,\bar{j}}(L)$. 
Let $\delta\geq 0$ 
satisfy $\alpha_{\bar{j}}(L,\bar{\lambda}^{\textrm{in}})=$
$\sum_{i \in V^{\textrm{in}}(L)} \bar{\lambda}^{\textrm{in}}_i (\alpha_{i,\bar{j}}(L)+\delta)$. We can 
now write $\bar{x}$ in the form 
$\bar{x}=v_{\bar{\lambda}^{\textrm{in}}}+\alpha_{\bar{j}}(L,\bar{\lambda}^{\textrm{in}}) r^{\bar{j}}=$
$\sum_{i \in V^{\textrm{in}}(L)} \bar{\lambda}^{\textrm{in}}_i (v_i+(\alpha_{i,\bar{j}}(L)+\delta) r^{\bar{j}})$. 
Since $v_i+(\alpha_{i,\bar{j}}(L)+\delta) r^{\bar{j}} \notin \intt(L)$ 
for all $i \in V^{\textrm{in}}(L)$, 
and $\bar{x}$ is a vertex of $R(L,P)$, we must have $\delta=0$ and 
$\bar{\lambda}^{\textrm{in}}_{\bar{i}}=1$ for some 
$\bar{i} \in V^{\textrm{in}}(L)$.
\end{proof}

An important consequence of Lemma \ref{rel_descr} is 
the following. For two mixed integer 
split polyhedra $L^1$ and $L^2$, 
if $V^{\textrm{in}}(L^1)=V^{\textrm{in}}(L^2)$, and 
if all the halflines $\{ v^i+\alpha r^j : \alpha \geq 0 \}$ and 
$\{ v^i + \beta (v^k-v^i) \}$ for 
$i \in V^{\textrm{in}}(L^1)=V^{\textrm{in}}(L^2)$, $j \in E$ and 
$k \in V^{\textrm{out}}(L^1)=V^{\textrm{out}}(L^2)$ intersect the boundaries of 
$L^1$ and $L^2$ at the 
same points, then $R(L^1,P)=R(L^2,P)$. In other words, the relaxation 
of $P_I$ obtained from $L^1$ is the same as the relaxation 
of $P_I$ obtained from $L^2$. 
\begin{corollary}
Let $L^1$ and $L^2$ be mixed integer split polyhedra satisfying 
$V^{\textrm{in}}(L^1)=V^{\textrm{in}}(L^2) \neq \emptyset$. 
If $\alpha_{i,j}(L^1)=\alpha_{i,j}(L^2)$ and 
$\beta_{i,k}(L^1)=\beta_{i,k}(L^2)$
for all 
$i \in V^{\textrm{in}}(L^1)=V^{\textrm{in}}(L^2)$, $j \in E$ and 
$k \in V^{\textrm{out}}(L^1)=V^{\textrm{out}}(L^2)$, then $R(L^1,P)=R(L^2,P)$.
\end{corollary}

Another consequence of Lemma \ref{rel_descr} is that it is 
possible to write $R(L,P)$ as the convex hull of the union 
of the polyhedra $R(L,P(e^i))$ for $i \in V^{\textrm{in}}(L)$. 

\begin{corollary}\label{x_space_char}
Let $L$ be a mixed integer split polyhedron 
satisfying $V^{\textrm{in}}(L) \neq \emptyset$. 
We have
$$R(L,P)=\conv(\cup_{i \in V^{\textrm{in}}(L)} R(L,P(e^i)).$$
\end{corollary}
\begin{proof}
Lemma  \ref{rel_descr} shows that every 
vertex of $R(L,P)$ is a vertex of a set $R(L,P(e^i))$ 
for some $i \in V^{\textrm{in}}(L)$. Furthermore, the 
union of the vertices of the sets $R(L,P(e^i))$ over all 
$i \in V^{\textrm{in}}(L)$ is exactly the set of vertices of 
$R(L,P)$. Since the extreme rays of $R(L,P)$ and 
the sets $R(L,P(e^i))$ for $i \in V^{\textrm{in}}(L)$ 
are the same, namely the vectors $\{ r^j {\}}_{j \in E}$, 
the result follows.
\end{proof}

Figure \ref{union_fig} illustrates 
Corollary \ref{x_space_char} on the example of 
Figure \ref{first_fig}. The sets $P(e^1)$ and $P(e^2)$ 
corresponding to the two vertices $v^1$ and $v^2$ of $P$ 
that are in the interior of $L$ are shown in 
Figure \ref{union_fig}.(b) and Figure \ref{union_fig}.(c). 
Observe that the sets $R(L,P(e^1))$ and $R(L,P(e^2))$ are 
both described by adding exactly one cut to 
$P(e^1)$ and $P(e^2)$ respectively. Corollary \ref{x_space_char} 
then shows that $R(L,P)$ can be obtained by taking the convex 
hull of the union of the sets $R(L,P(e^1))$ and $R(L,P(e^2))$.

\begin{figure}
\centering
\mbox{\subfigure[The polytope $P$ and the split polyhedron $L$ from
Figure \ref{first_fig}]{\includegraphics[width=5cm]{IMA3.eps}}\quad
\subfigure[The set $P(e^1)$ constructed from $v^1$]{\includegraphics[width=5cm]{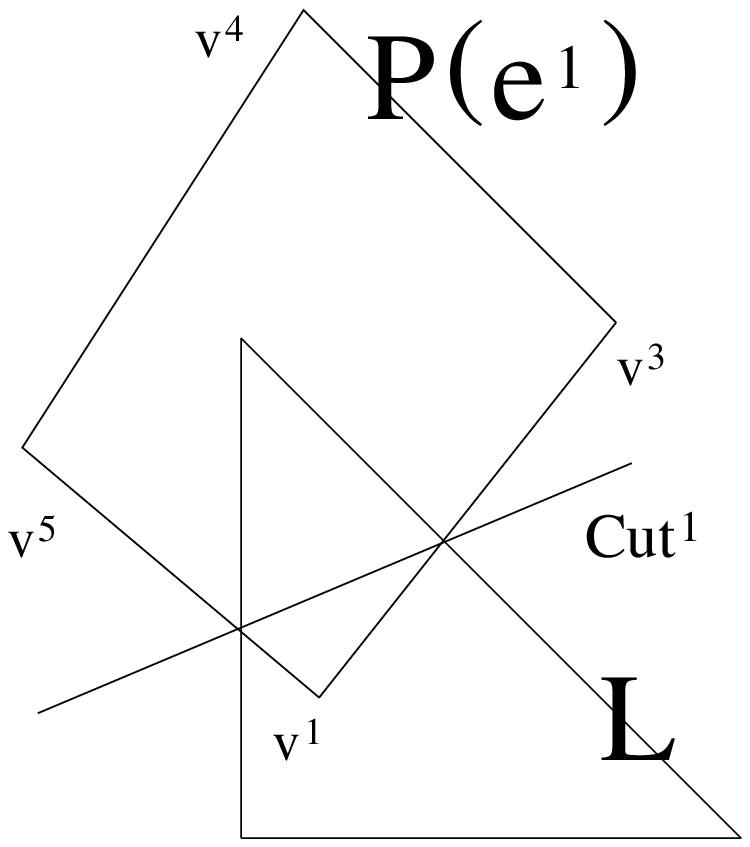}}
\subfigure[The set $P(e^2)$ constructed from $v^2$]{\includegraphics[width=5cm]{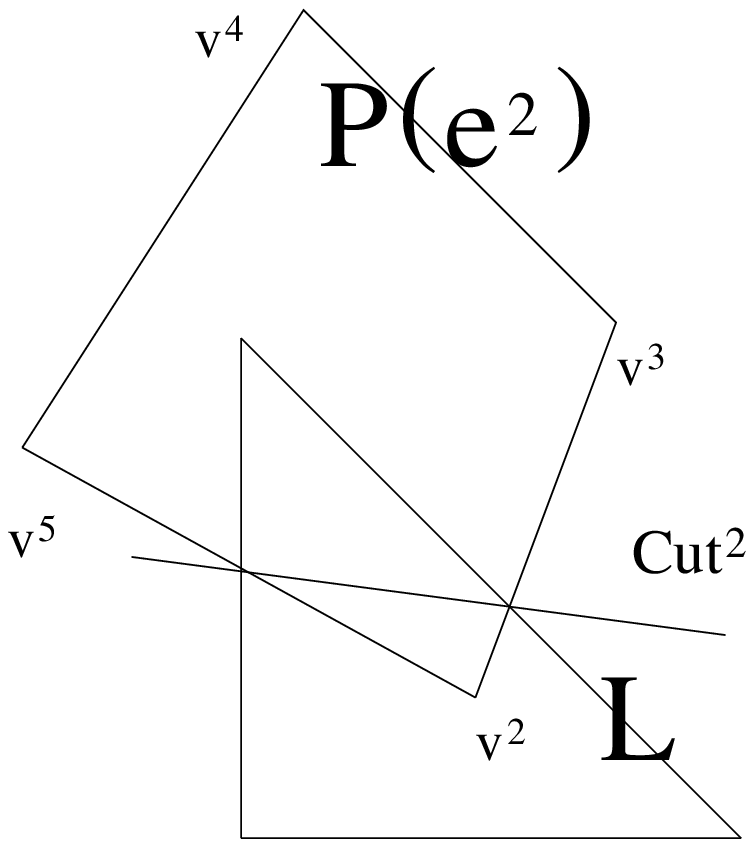}}}
\caption{Constructing $R(L,P)$ as the convex hull of the union of polyhedra}
\label{union_fig}
\end{figure}

\subsection{Polyhedrality of the
$w^{\textrm{th}}$ split closure}

We now use Theorem \ref{pol_thm} to prove that the
$w^{\textrm{th}}$ split closure of $P$ is a polyhedron. 
Let $L \in {\cal{L}}^w$ be an arbitrary mixed integer split polyhedron,
where $w>0$, and let $\delta^T x \geq \delta_0$ be a valid
inequality for $R(L,P)$ with integral coefficients which is not 
valid for $P$. To use Theorem \ref{pol_thm}, we 
consider potential intersection points between 
the hyperplane $\delta^T x = \delta_0$ and 
halflines of the form 
$\{ v^i + \alpha r^j : \alpha\geq 0 \}$, and 
of the form $\{ v^i + \beta (v^k-v^i) : \beta \geq 0 \}$, 
where $i \in V^{\textrm{in}}(L)$, $j \in E$ and 
$k \in V^{\textrm{out}}(L)$. 
The properties we derive of these intersection points do not 
depend on the particular halfline, so we only consider 
the halfline $\{ v^1 + \alpha r^1 : \alpha\geq 0 \}$. 
We will show that the rationality of 
$v^1$ and $r^1$ can 
be used to limit 
the number of possible intersection points. This then allows us 
to conclude that the $w^{\textrm{th}}$ split closure is a polyhedron. 
We first give a representation of 
$\alpha'_{1,1}(\delta,\delta_0)$ for a given 
valid inequality $\delta^T x \geq \delta_0$ for 
$v^1+\alpha_{1,1}(L) r^1$. 

\begin{lemma}\label{rep_alpha}
(Lemma 5 in \cite{andcorli}). 
Let $L \in {\cal{L}}^w$ be a mixed integer split polyhedron with max-facet-width at most 
$w>0$. Suppose $v^1 \in \intt(L)$ 
and $\alpha_{1,1}(L)<+\infty$, and also suppose $\delta^T x \geq
\delta_0$ is a non-negative cut for $\{ v^1 + \alpha r^1 : \alpha \geq
0 \}$ with integral 
coefficients that is valid for $v^1+\alpha_{1,1}(L) r^1$. 
\begin{enumerate}
\item[(i)] $0 < \alpha'_{1,1}(\delta,\delta_0) \leq \alpha_{1,1}(L)
< w$, and
\item[(ii)] $\alpha'_{1,1}(\delta,\delta_0)=\frac{s(\delta,\delta_0)}{g t(\delta,\delta_0)}$, where 
$g$, $s(\delta,\delta_0)$,$t(\delta,\delta_0) >0$ are 
integers satisfying $s(\delta,\delta_0) < g w$.
\end{enumerate}
\noindent (Note that the integer $g$ is independent of both $L$ and 
$\delta^T x \geq \delta_0$).
\end{lemma}
\begin{proof}
We may write 
$L=\{ x \in \R^n : (\pi^k)^T x \geq \pi^k_0 \textrm{ for
}k \in N_f \}$, where 
$N_f:=\{1,2,\ldots,n_f \}$, 
$n_f$ denotes 
the number of facets of $L$ and $(\pi^k,\pi^k_0) \in \Z^{n+1}$ 
for $k \in N_f$. Since $v^1 \in \intt(L)$, 
we have $(\pi^k)^T v^1 < \pi^k_0$ for all 
$k \in N_f$, and therefore 
$\alpha_{1,1}(L)=\frac{\pi^{\bar{k}}_0-(\pi^{\bar{k}})^T
  v^1}{(\pi^{\bar{k}})^T r^1}$ for some 
$\bar{k} \in N_f$. Since 
$L$ has max-facet-width at most $w$ and $v^1 \in \intt(L)$, we have 
$0< \pi^{\bar{k}}_0-(\pi^{\bar{k}})^T v^1 < w$. 
Hence, since $(\pi^{\bar{k}})^T r^1$ is integer, we have 
$(\pi^{\bar{k}})^T r^1 \geq 1$, and therefore $\alpha_{1,1}(L) < w$. 
Furthermore, since $\delta^T x \geq \delta_0$ is a 
non-negative cut for the set $\{ v^1 + \alpha r^1 : \alpha \geq 0 \}$ 
that is valid 
for $v^1+\alpha_{1,1}(L) r^1$, we have 
$\alpha'_{1,1}(\delta,\delta_0) \leq \alpha_{1,1}(L)$.

Recall that we assumed $v^1 \in \Q^n$ and $r^1 \in \Z^n$. We can 
therefore write 
$v^1=(\frac{p_1}{q_1},\frac{p_2}{q_2},\ldots,\frac{p_n}{q_n})$, 
where $p_k \in \Z$ and 
$q_k \in \N$ for $k=1,2,\ldots,n$. 
Define the integers $g:=\Pi_{k=1}^n q_k$, 
$d_m:=\Pi_{k=1, k \neq m}^n q_k$ for $m \in \{1,2,\ldots,n \}$, 
$s(\delta,\delta_0):=g \delta_0-\sum_{m=1}^n d_m p_m
\delta_m$ and $t(\delta,\delta_0):=\delta^T r^1$. 
Observe that 
$\frac{s(\delta,\delta_0)}{g}=\delta_0-\delta^T v^1$. With 
these choices, (ii) is satisfied.
\end{proof}

By using the above lemma, we can now bound the number of 
possible intersection points with a halfline of the form 
$\{ v^1 + \alpha r^1 : \alpha\geq \alpha^* \}$ for some 
$\alpha^* > 0$.

\begin{lemma}\label{finite_over}
(Lemma 6 in \cite{andcorli}).
Let $\alpha^* > 0$ and $w>0$. Also let 
$\{ (\delta^l)^T x \geq \delta^l_0 {\}}_{l \in I}$ be 
a set of non-negative cuts for $\{ v^1 + \alpha r^1 : \alpha \geq 0
\}$ 
with integral coefficients that are all valid for 
a point $v^1+\alpha_{1,1}(L) r^1$ 
for some $L \in {\cal{L}}^w$. The set 
$\{ \alpha'_{1,1}(\delta^l,\delta^l_0) : l \in I \textrm{ and }
\alpha'_{1,1}(\delta^l,\delta^l_0) \geq \alpha^*\}$ is finite.
\end{lemma}
\begin{proof}
Let $l \in I$ satisfy 
$\alpha^* \leq \alpha'_{1,1}(\delta^l,\delta^l_0) \leq +\infty$. 
We may assume $\alpha^*$ is of the form $\alpha^*=\frac{s^*}{g t^*}$ 
for some integers $s^*,t^* >0$ satisfying $0 < s^* < g w$. 

Let $s(\delta^l,\delta^l_0)$ and $t(\delta^l,\delta^l_0)$ 
be as in Lemma \ref{rep_alpha}. Hence we have 
$\alpha'_{1,1}(\delta^l,\delta^l_0)=\frac{s(\delta^l,\delta^l_0)}{g t(\delta^l,\delta^l_0)}$. 
This implies $s(\delta^l,\delta^l_0)  \in \{ 1,2,\ldots,(g w-1) \}$,
so there is only a finite number of possible values for $s(\delta^l,\delta^l_0)$. 
Finally, Lemma \ref{rep_alpha}.(i) and 
$\alpha'_{1,1}(\delta^l,\delta^l_0) \geq \alpha^*$ 
gives 
$\frac{s^*}{g t^*} \leq \frac{s(\delta^l,\delta^l_0)}{g
t(\delta,\delta_0)} < w$, and therefore 
$\frac{s(\delta^l,\delta^l_0)}{g w} < t(\delta^l,\delta^l_0) \leq
\frac{s(\delta^l,\delta^l_0) t^*}{s^*}$. Hence, for a fixed value 
$s(\delta^l,\delta^l_0)  \in \{ 1,2,\ldots,(g w-1) \}$, 
there is only a finite number of possible values 
for $t(\delta^l,\delta^l_0)$.
\end{proof}

By using Lemma \ref{finite_over}, we can now conclude that 
the $w^{\textrm{th}}$ split closure is a polyhedron.
\begin{theorem}\label{poly_spl_cl}
Let $\bar{\cal{L}} \subseteq {\cal{L}}^w$ be any family of mixed 
integer split polyhedra that have max-facet-width at most $w>0$. 
The set $\cap_{L \in} R(L,P)$ is a polyhedron.
\end{theorem}
\begin{proof}
Let $\{ (\delta^l)^T x \geq \delta^l_0 {\}}_{l \in I}$ denote 
the family of all cuts for $P$ that are valid and 
facet defining for $R(L,P)$ for some $L \in \bar{\cal{L}}$. As 
discussed in Sect. 3.3, we can partition the cuts in $I$ 
into a finite number of subsets $I^c(S) \subseteq I$ according to 
which set $S \subseteq V$ of vertices they cut off. Lemma 
\ref{finite_over} shows that Assumption \ref{bas_ass} of Sect. 3.3 is 
satisfied by each set $\{ (\delta^l)^T x \geq \delta^l_0 {\}}_{l \in
I^c(S)}$.
\end{proof}

\section{Finite split polyhedron proofs}

Mixed integer split polyhedra can be used to design finite cutting 
plane proofs for the validity of an inequality for $P_I$ 
as follows. Let 
$\delta^T x \geq \delta_0$ be a 
valid inequality for $P_I$. 
Observe that, if $\delta^T x \geq \delta_0$ 
is valid for $R(L,P)$ for some mixed integer split 
polyhedron $L$, then $L$ provides 
a finite cutting plane proof of validity of 
$\delta^T x \geq \delta_0$ 
for $P_I$. More generally, 
a family $\mathcal{S}$ of mixed integer split polyhedra 
gives an approximation of $P_I$ of the form
\begin{equation*}
\textrm{Cl}(\mathcal{S},P) := \bigcap_{L \in \mathcal{S}} R(L,P).
\end{equation*}
The set $\textrm{Cl}(\mathcal{S},P)$ gives the closure 
wrt. the family $\mathcal{S}$. Improved 
approximations of $P_I$ can be obtained by 
iteratively computing closures 
$P^1(\mathcal{S},P),P^2(\mathcal{S},P),P^3(\mathcal{S},P)\ldots$, 
where $P^0(\mathcal{S},P)=P$, 
$P^1(\mathcal{S},P)=\textrm{Cl}(\mathcal{S},P^0(\mathcal{S},P))$, 
$P^2(\mathcal{S},P)=\textrm{Cl}(\mathcal{S},P^1(\mathcal{S},P))$ 
etc. A \emph{finite split polyhedron proof} of validity of 
$\delta^T x \geq \delta_0$ for 
$P_I$ is a finite family $\mathcal{S}$ of mixed integer split 
polyhedra such that 
$\delta^T x \geq \delta_0$ is 
valid for $P^k(\mathcal{S},P)$ for some $k<\infty$, and 
a finite cutting plane proof is given from a finite 
split polyhedron proof by the valid inequalities 
for the polyhedron $P^k(\mathcal{S},P)$. 

A measure of the complexity of a finite split polyhedron 
proof $\mathcal{S}$ is the max-facet-width of the mixed integer 
split polyhedron 
$L \in \mathcal{S}$ with the largest max-facet-width. We 
call this number the \emph{width size} of a split polyhedron 
proof. A measure of the complexity of a valid inequality 
$\delta^T x \geq \delta_0$ for $P_I$ 
is then the smallest number $w$ 
for which there exists 
a finite split polyhedron proof of validity of 
$\delta^T x \geq \delta_0$ for $P_I$ 
of width size $w$. 
This number is called the width size of 
$\delta^T x \geq \delta_0$, and 
it is denoted $\textrm{width-size}(\delta,\delta_0)$.
Finally, 
since validity of every facet defining inequality for 
$\conv(P_I)$ must be proved to generate $\conv(P_I)$, 
the largest of the numbers $\textrm{width-size}(\delta,\delta_0)$ over 
all facet defining inequalities 
$\delta^T x \geq \delta_0$ 
for $\conv(P_I)$ gives a measure of 
the complexity of $P_I$. We call this number the width 
size of $P_I$, and it is denoted 
$\textrm{width-size}(P_I)$. We give 
an example to show that $\textrm{width-size}(P_I)$ 
can be as large as the number of integer constrained variables 
at the end of this section.

We now characterize exactly which max-facet-width is necessary 
to prove validity of an inequality $\delta^T x \geq \delta_0$
for $P_I$ with a finite split polyhedron proof, {\it i.e.}, 
we characterize the number $\textrm{width-size}(\delta,\delta_0)$. 
We will partition the inequality $\delta^T x \geq \delta_0$ into 
its integer part and its continuous part. 
Throughout the remainder of this section, 
$(\delta^x)^T x + (\delta^y)^T y \geq \delta_0$ denotes 
an arbitrary valid inequality for $P_I$, 
where $\delta^x \in \Q^p$, $\delta^y \in \Q^q$ and 
$\delta_0 \in \Q$. We assume 
$(\delta^x)^T x + (\delta^y)^T y \geq \delta_0$ 
is tight at a mixed integer point of $P_I$. 

It is possible to prove validity of 
$(\delta^x)^T x + (\delta^y)^T y \geq \delta_0$ 
for $\conv(P_I)$ by solving the 
mixed integer linear problem (MIP)
\begin{alignat}{3}
\min & \,(\delta^x)^T x + (\delta^y)^T y & &\notag\\
s.t. & & &\notag\\
& (x,y) \in \,\,P_I. & &\notag
\end{alignat}

The following notation is used. 
The point $(x^*,y^*) \in P_I$ denotes an 
optimal solution to MIP, 
and $(x^{\textrm{lp}},y^{\textrm{lp}}) \in P$ denotes an optimal 
solution to the linear relaxation of MIP. 
We assume $\delta_0 = (\delta^x)^T x^* + (\delta^y)^T y^*$ 
and 
$(\delta^x)^T x^{\textrm{lp}} + (\delta^y)^T y^{\textrm{lp}} < \delta_0$. 
From the inequality $(\delta^x)^T x + (\delta^y)^T y \geq \delta_0$, 
we can create the following subsets of $P$ and $P_I$
\begin{alignat}{2}
P(\delta,\delta_0) & := \{ (x,y) \in P : (\delta^x)^T x + (\delta^y)^T y \leq \delta_0 \}\textrm{ and }\notag\\
P_I(\delta,\delta_0) & := \{ (x,y) \in P(\delta,\delta_0) : x \in \Z^p \}.\notag
\end{alignat}

To prove validity of $(\delta^x)^T x + (\delta^y)^T y \geq \delta_0$ 
for $\conv(P_I)$, we consider the following projections 
of $P(\delta,\delta_0)$ and $P_I(\delta,\delta_0)$ 
onto the space of the integer constrained $x$ variables 
\begin{alignat}{2}
P^x(\delta,\delta_0) & := \{ x \in \R^p : \exists y\in\R^q\textrm{ such that } (x,y)
\in P(\delta,\delta_0) \}\textrm{ and }\notag\\
P^x_I(\delta,\delta_0) & := P^x(\delta,\delta_0) \cap \Z^p.\notag
\end{alignat}

The validity proofs we derive for 
$(\delta^x)^T x + (\delta^y)^T y \geq \delta_0$ are based on the 
following important property.
\begin{lemma}
The polyhedron $P^x(\delta,\delta_0)$ is lattice point free.
\end{lemma}
\begin{proof}
The relative interior of $P^x(\delta,\delta_0)$ is 
given by
$$\ri(P^x(\delta,\delta_0))=\{ x \in \R^p : \exists y \in \R^q\textrm{
  such that }(x,y) \in \ri(P)\textrm{ and }(\delta^x)^T x + (\delta^y)^T y
< \delta_0 \}.$$
Since $\delta_0$ is the optimal objective value of MIP, 
$\ri(P^x(\delta,\delta_0))$ does not contain lattice points.
\end{proof}

It is well known that split 
polyhedra with max-facet-width equal to one 
are sufficient to generate the integer 
hull of a pure integer set. It follows from 
this result that there exists a finite number 
of split polyhedra with max-facet-width equal to one such that 
a polyhedron $\overline{P}$ can be obtained 
in a finite number of iterations that satisfies 
$\overline{P}^x(\delta,\delta_0)=\conv(\overline{P}^x_I(\delta,\delta_0))$. 
Hence, since the purpose in this section is to provide finite 
split polyhedron 
proofs, we can assume 
$P^x(\delta,\delta_0)=\conv(P^x_I(\delta,\delta_0))$ 
in the remainder of this section. 

The split polyhedra that are needed to 
prove validity of  
$(\delta^x)^T x+(\delta^y)^T y \geq \delta_0$ for $P_I$ 
depend on the facial structure of 
$P^x(\delta,\delta_0)$. To obtain a description of the 
faces of $P^x(\delta,\delta_0)$, we 
need the following reformulation of 
$P^x(\delta,\delta_0)$.

\begin{lemma}\label{reform_P_x_alpha_beta}
Assume $P^x(\delta,\delta_0)=\conv(P^x_I(\delta,\delta_0))$. 
For every $x \in P^x(\delta,\delta_0)$, there 
exists $y \in \R^q$ such that 
$(x,y) \in P$ and 
$(\delta^x)^T x + (\delta^y)^T y = \delta_0$. Hence 
\begin{alignat}{2}
P^x(\delta,\delta_0) & =\{ x \in \R^p : \textrm{ there exists }y \in \R^q
\textrm{ s.t. }(x,y) \in P\textrm{ and } (\delta^x)^T x + (\delta^y)^T y =
\delta_0 \}.\notag
\end{alignat}
\end{lemma}
\begin{proof}
First suppose $\bar{x} \in P^x(\delta,\delta_0)$ is integer. By 
definition of $P^x(\delta,\delta_0)$, 
there exists $\bar{y} \in\R^q$ such that 
$(\bar{x},\bar{y}) \in P$ and 
$(\delta^x)^T \bar{x} + (\delta^y)^T \bar{y} \leq \delta_0$. 
We can not have 
$(\delta^x)^T \bar{x} + (\delta^y)^T \bar{y} < \delta_0$, 
since $\delta_0$ is the optimal 
objective of MIP. Hence 
$(\bar{x},\bar{y}) \in P$ 
and $(\delta^x)^T \bar{x} + (\delta^y)^T \bar{y} = \delta_0$.

Now suppose $x^r \in \Q^p$ is a ray of $P^x(\delta,\delta_0)$. 
We claim that for every $\mu \geq 0$ and 
$\bar{x} \in P_I^x(\delta,\delta_0)$, there exists $\bar{y} \in \R^q$ 
such that $(\bar{x}+\mu x^r,\bar{y}) \in P$ 
and $(\delta^x)^T (\bar{x}+\mu x^r)+$
$(\delta^y)^T \bar{y}=\delta_0$.  
Indeed, let $\mu \geq 0$ and $\bar{x} \in P_I^x(\delta,\delta_0)$ 
be arbitrary. We can choose 
a non-negative integer $\mu^I \geq \mu$ 
such that $\bar{x}+\mu^I x^r$ is integer. We 
therefore have that there exists $y^1 \in \R^q$ 
such that $(\bar{x}+\mu^I x^r,y^1) \in P$ 
and $(\delta^x)^T (\bar{x}+\mu^I x^r)+(\delta^y)^T y^1 =\delta_0$. 
Since $\bar{x} \in P_I^x(\delta,\delta_0)$, we also have that there 
exists $y^2 \in \R^q$ such that 
$(\bar{x},y^2) \in P$ and $(\delta^x)^T \bar{x}+(\delta^y)^T y^2 =\delta_0$. 
By choosing $\lambda:=\frac{\mu}{\mu_I}$ 
and $\bar{y}:=\lambda y^1+(1-\lambda) y^2$, we 
have $(\bar{x}+\mu x^r,\bar{y})=$
$\lambda (\bar{x}+\mu^I x^r,y^1) +$
$(1-\lambda) (\bar{x},y^2)$, and therefore 
$(\bar{x}+\mu x^r,\bar{y}) \in P$. 
In addition we have that 
$(\delta^x)^T (\bar{x}+\mu x^r)+(\delta^y)^T \bar{y}=\delta_0$.

Finally let $\bar{x} \in P^x(\delta,\delta_0)$ be arbitrary. 
We may write 
$\bar{x}=\sum_{i=1}^k \lambda_i x^i + d=$
$\sum_{i=1}^k \lambda_i (x^i+d)$, 
where $\{ x^i {\}}_{i=1}^k$ are the vertices of 
$P^x(\delta,\delta_0)$, $d \in \Q^p$ is a non-negative 
combination of the extreme rays of $P^x(\delta,\delta_0)$, 
$\lambda_1,\lambda_2,\ldots,\lambda_k \geq 0$ and 
$\sum_{i=1}^k \lambda_i = 1$. From what was shown 
above, we have that for every $i \in \{1,2,\ldots,k \}$, 
there exists $y^i \in \R^q$ such that 
$(x^i+d,y^i) \in P$ and 
$(\delta^x)^T (x^i+d) + (\delta^y)^T y^i = \delta_0$. 
By letting $\bar{y}:=\sum_{i=1}^k \lambda_i y^i$, 
we have that $(\bar{x},\bar{y}) \in P$ and 
$(\delta^x)^T \bar{x} + (\delta^y)^T \bar{y} = \delta_0$.
\end{proof}

The faces of $P^x(\alpha,\beta)$ can now be characterized. 
Let 
$P = \{ (x,y) \in \R^p \times \R^q : A x + D y \leq b \}$ be 
an outer description of $P$, 
where $A \in \Q^{m \times p}$, 
$D \in \Q^{m \times q}$ and $b \in \Q^m$, and let 
$M:=\{1,2,\ldots,m \}$. Lemma 
\ref{reform_P_x_alpha_beta} shows that 
$P^x(\delta,\delta_0)$ can be written in the form 
\begin{alignat}{3}
P^x(\delta,\delta_0) & = \{ x \in \R^p : & \,a_{i.}^T x + d_{i.}^T y = b_i, & \,i\in M^=,\notag\\
& & \,a_{i.}^T x + d_{i.}^T y \leq b_i, & \,i \in M \setminus
M^=,\notag\\
& & \,(\delta^x)^T x + (\delta^y)^T y = \delta_0 & \},\notag
\end{alignat}
where $M^= \subseteq M$ denotes those constrains 
$i \in M$ for which $a_{i.}^T x + d_{i.} y = b_i$ 
for all $(x,y) \in P(\delta,\delta_0)$ 
that satisfy $(\delta^x)^T x + (\delta^y)^T y = \delta_0$. 
Also, for every $i \in M \setminus M^=$, there exists 
$(x,y) \in P(\delta,\delta_0)$ that satisfies 
$(\delta^x)^T x + (\delta^y)^T y = \delta_0$ 
and $a_{i.}^T x + d_{i.} y < b_i$.

A non-empty face $F$ of $P^x(\delta,\delta_0)$ can be 
characterized by a set $M^F \subseteq M$ of inequalities 
that satisfies 
$M^= \subseteq M^F$. 
Every face $F$ of $P^x(\delta,\delta_0)$ 
can be written in the form
\begin{alignat}{3}
F & = \{ x \in \R^p : & \,a_{i.}^T x + d_{i.}^T y = b_i, & \,i\in M^F,\notag\\
& & \,a_{i.}^T x + d_{i.}^T y \leq b_i, & \,i \in M \setminus
M^F,\notag\\
& & \,(\delta^x)^T x + (\delta^y)^T y = \delta_0 & \}.\notag
\end{alignat}

Consider an arbitrary proper face $F$ of $P^x(\delta,\delta_0)$. 
In order for $(\delta^x)^T x + (\delta^y)^T y \geq \delta_0$ 
to be valid for $P_I$, $(\delta^x)^T x + (\delta^y)^T y \geq \delta_0$ 
must be valid for all $(x,y) \in P$ such that 
$x \in F$. The following lemma shows that $F$ is 
of exactly one of two types depending on the 
coefficient vectors on the continous variables 
in the tight constraints.

\begin{lemma}\label{facelem-morecon} (A characterization of the faces
  of $P^x(\delta,\delta_0)$)\newline
Assume $P^x(\delta,\delta_0)=\conv(P_I^x(\delta,\delta_0))$. 
Let $F$ be a face of $P^x(\delta,\delta_0)$. 
\begin{enumerate}
\item[(i)] If 
$\delta^y \notin \spann(\{ d_{i.} {\}}_{i \in M^F})$:
\begin{enumerate}
\item[(a)] $F$ is lattice point free.
\item[(b)] For every $x \in \ri(F)$, 
there exists $y \in \R^q$ 
s.t. $(x,y) \in P$ and 
$(\delta^x)^T x + (\delta^y)^T y < \delta_0$. 
\end{enumerate}
\item[(ii)] If 
$\delta^y \in \spann(\{ d_{i.} {\}}_{i \in M^F})$:\newline
The inequality $(\delta^x)^T x + (\delta^y)^T y \geq \delta_0$ 
holds for all $(x,y) \in P$ satisfying 
$x \in \ri(F)$. 
\end{enumerate}
\end{lemma}
\begin{proof}
(i) Suppose $\delta^y \notin \spann(\{ d_{i.} {\}}_{i \in M^F})$, 
and let 
$\bar{x} \in \ri(F)$ be arbitrary. 
This implies there exists $\bar{y} \in \R^q$ 
such that $a_{i.}^T \bar{x} + d_{i.}^T \bar{y} < b_i$ 
for all $i \in M \setminus M^F$. Since 
$\delta^y \notin \spann(\{ d_{i.} {\}}_{i \in M^F})$, 
the linear program 
$\min\{ (\delta^y)^T r : d_{i.}^T r = 0, \forall i \in M^F \}$ 
is unbounded. Choose $\bar{r} \in \R^q$ such that 
$(\delta^y)^T \bar{r} < 0$ and 
$d_{i.}^T \bar{r} = 0$ for all $i \in M^F$. 
We have that 
$(\delta^x)^T \bar{x} + (\delta^y)^T (\bar{y}+\mu \bar{r})<$ 
$(\delta^x)^T \bar{x} + (\delta^y)^T \bar{y}=\delta_0$ for every 
$\mu>0$. Furthermore, since $(\bar{x},\bar{y})$ 
satifies $a_{i.}^T \bar{x} + d_{i.}^T \bar{y} < b_i$ 
for all $i \in M \setminus M^F$, there exists 
$\bar{\mu} > 0$ such that 
$(\bar{x},\bar{y}+\bar{\mu} \bar{r}) \in P$ 
and 
$(\delta^x)^T \bar{x} + (\delta^y)^T (\bar{y}+\bar{\mu} \bar{r}) < \delta_0$. 
We can not have $\bar{x}$ integer, since this 
would contradict that $\delta_0$ is the optimal 
objective of MIP.\\

\noindent (ii) Let $(\bar{x},\bar{y}) \in P$ 
satisfy $\bar{x} \in \ri(F)$, 
and suppose $\delta^y \in \spann(\{ d_{i.} {\}}_{i \in M^F})$. 
If $(\delta^x)^T \bar{x} + (\delta^y)^T \bar{y} \geq \delta_0$, 
we are done, so suppose for a contradiction 
that 
$(\delta^x)^T \bar{x} + (\delta^y)^T \bar{y} < \delta_0$. 
Since $\bar{x} \in \ri(F)$, there 
exists $\tilde{y} \in \R^q$ such 
that $(\bar{x},\tilde{y}) \in P$, 
$(\delta^x)^T \bar{x} + (\delta^y)^T \tilde{y} = \delta_0$ 
and $a_{i.}^T \bar{x} + d_{i.}^T \tilde{y} < b_i$ 
for all $i \in M \setminus M^F$. Consider the vector 
$\bar{r}:=\bar{y}-\tilde{y}$. 
We have $d_{i.}^T \bar{r}=0$ for all $i \in M^F$ 
and $(\delta^y)^T \bar{r} < 0$. However, this contradicts 
$\delta^y \in \spann(\{ d_{i.} {\}}_{i \in M^F})$.
\end{proof}

We can now identify the mixed integer split 
polyhedra that are needed 
to provide a finite split polyhedron proof of validity of 
$(\delta^x)^T x + (\delta^y)^T y \geq \delta_0$ for 
$P_I$. Let $\mathcal{F}$ denote the finite set of 
all faces of $P^x(\delta,\delta_0)$, and let 
$\mathcal{F}^V:=\{ F \in \mathcal{F} :$
$\exists (x,y) \in P \textrm{ s.t. }$
$x \in F \textrm{ and }(\delta^x)^T x + (\delta^y)^T y < \delta_0 \}$ denote 
those faces $F \in \mathcal{F}$ for which there exists 
$(x,y) \in P$ such that $x \in F$ and $(x,y)$ violates the 
inequality $(\delta^x)^T x + (\delta^y)^T y \geq \delta_0$. A face 
$F \in \mathcal{F}^V$ is called a \emph{violated face}. 
Lemma \ref{facelem-morecon}.(i) shows that every 
violated face is lattice point free. 
A mixed integer split polyhedron 
$L \subseteq \R^n$ that satisfies 
$(\delta^x)^T x + (\delta^y)^T y \geq \delta_0$ for 
every $(x,y) \in R(L,P)$ such that $x \in F$ is said 
to \emph{prove validity of $(\delta^x)^T x + (\delta^y)^T y \geq
  \delta_0$ on $F$}. 
Given a violated face $F \in \mathcal{F}^V$, the following 
lemma gives a class of split polyhedra that can prove validity of 
$(\delta^x)^T x + (\delta^y)^T y \geq \delta_0$ on $F$.

\begin{lemma}(Split polyhedra for proving validity of 
$(\delta^x)^T x + (\delta^y)^T y \geq \delta_0$ on a face
  of $P^x(\delta,\delta_0)$)\newline 
Assume $P^x(\delta,\delta_0)=\conv(P_I^x(\delta,\delta_0))$. 
Let $F \in \mathcal{F}^V$ be a violated 
face of $P^x(\delta,\delta_0)$, and suppose 
$G \notin \mathcal{F}^V$ for every 
proper face $G$ of $F$.  Every mixed integer split polyhedron 
$L \subseteq \R^n$ that satisfies 
$\ri(F) \subseteq \intt(L)$ proves validity of 
$(\delta^x)^T x + (\delta^y)^T y \geq \delta_0$ on $F$.
\end{lemma}
\begin{proof}
Let $L$ be a mixed integer split polyhedron that 
satisfies $\ri(F) \subseteq \intt(L)$, 
and let 
$(\bar{x},\bar{y}) \in P$ satisfy 
$\bar{x} \in F$ and $\bar{x} \notin \intt(L)$. Since 
$\ri(F) \subseteq \intt(L)$, it follows 
that $\bar{x} \notin \ri(F)$. Since 
$\bar{x} \in F \setminus \ri(F)$, $\bar{x}$ must 
be on some proper face $G$ of $F$. 
Since $G \notin \mathcal{F}^V$, 
we have 
$(\delta^x)^T \bar{x} + (\delta^y)^T \bar{y} \geq \delta_0$. 
Since $R(L,F)=\conv(\{ (x,y) \in F : x \notin \intt(L) \})$, 
the result follows.
\end{proof}

By iteratively considering the finite number $|\mathcal{F}^V|$ 
of violated faces of $P^x(\delta,\delta_0)$, we obtain a finite 
split polyhedron proof for the validity of the inequality 
$(\delta^x)^T x + (\delta^y)^T y \geq \delta_0$ for $P_I$.
\begin{corollary} (Upper bound on the width size of the 
inequality $(\delta^x)^T x + (\delta^y)^T y \geq \delta_0$)\newline
There exists a split polyhedron proof for the validity of 
$(\delta^x)^T x + (\delta^y)^T y \geq \delta_0$ for $P_I$ 
of width size
\begin{alignat}{2}
& \max\{ \textrm{width-size}(F) : F
\in \mathcal{F}^V \}.\notag
\end{alignat}
\end{corollary}

We can now prove the main theorem of this section.
\begin{theorem}\label{split_dim_ineq}(A formula for the width size of the 
inequality $(\delta^x)^T x + (\delta^y)^T y \geq \delta_0$)\newline
Let $\textrm{width-size}(\delta,\delta_0)$ denote the smallest 
number $w$ for which there exists a finite split polyhedron proof 
of validity of $(\delta^x)^T x + (\delta^y)^T y \geq \delta_0$ 
for $P_I$ of width size $w$. Then 
\begin{alignat}{2}
\textrm{width-size}(\delta,\delta_0) & = \max\{ \textrm{width-size}(F) : F
\in \mathcal{F}^V \}.\notag
\end{alignat}
\end{theorem}
\begin{proof}
Let $L$ be a mixed integer split polyhedron of smaller 
width size than 
$\max\{ \textrm{width-size}(F) : F \in \mathcal{F}^V \}$. 
This implies there exists $F \in \mathcal{F}^V$ and 
$x' \in \ri(F)$ such that 
$x' \notin \intt(L)$. Furthermore, since 
$x' \in \ri(F)$, it 
follows from Lemma \ref{facelem-morecon}.(i) that there exists 
$y' \in \R^q$ such that $(x',y') \in P$ and 
$(\delta^x)^T x' + (\delta^y)^T y' < \delta_0$. We now 
have $(x',y') \in R(L,P)$ and 
$(\delta^x)^T x' + (\delta^y)^T y' < \delta_0$.
\end{proof}

\begin{example}
Consider the mixed integer linear program (MILP)
\begin{alignat}{3}
\max & \,y & &\notag\\
s.t. & & &\notag\\
-x_i + y & \leq 0, & \textrm{ for }i=1,2,\ldots,p,&\label{leq-0-in}\\
\sum_{i=1}^p x_i + y& \leq p, & &\label{leq-p-in}\\
y & \geq 0, & &\\
x_i & \textrm{ integer } & \textrm{ for }i=1,2,\ldots,p.&
\end{alignat}
The optimal solutions to MILP are of the 
form $(x^*,y^*)=(x^*,0)$ with $x^* \in S^p \cap \Z^p$, 
where $S^p:=\{ x \in \R^p : x \geq 0 \textrm{ and }\sum_{i=1}^p x_i \leq p \}$. 
The unique optimal solution to the LP relaxation of 
MILP is given by $x_i^{\textrm{lp}}=\frac{p}{p+1}$ for 
$i=1,2,\ldots,p$ and $y^{\textrm{lp}}=\frac{p}{p+1}$. 
Hence the 
only missing inequality to describe $\conv(P_I)$ 
is the inequality $y \leq 0$. We have 
$\delta^x=0$, $\delta^y=-1$ and $\delta_0=0$. 

Observe that any proper face $G$ of $P^x(\delta,\delta_0)$ 
contains mixed integer points in their relative interior. 
It follows that the inequality 
$y \leq 0$ is valid for every $(x,y) \in P$ such 
that $x$ belongs to a proper face of $P^x(\delta,\delta_0)$. 
Hence the only interesting face of $P^x(\delta,\delta_0)$ 
to consider is the improper face $F:=P^x(\delta,\delta_0)=S^p$. 
The only mixed integer split polyhedron $L$ that satisfies 
$\ri(F) \subseteq \intt(L)$ 
is the split polyhedron $L=S^p$, and this 
mixed integer split polyhedron has max-facet-width 
$p$. It follows from Theorem \ref{split_dim_ineq} that no cutting 
plane algorithm that only uses mixed integer split polyhedra 
of max-facet-width 
smaller than $p$ can solve MILP in a finite number of steps.
\end{example}

\small
\nocite{*}
\bibliography{template}
\bibliographystyle{plain}

\end{document}